\documentclass[a4paper,reqno]{amsart}
 \RequirePackage{etex}
\usepackage{amssymb,latexsym,amsmath,amscd,amsthm,amsfonts, enumerate}
\usepackage{multirow}
\usepackage{color}
\usepackage[all]{xy}
\usepackage{etex}
\usepackage{caption}
\usepackage{soul}
\usepackage{titletoc}
\usepackage[normalem]{ulem}
\usepackage{graphicx}
\usepackage{amsfonts}
\usepackage{amsmath}

\usepackage{amsthm}
\usepackage{mathrsfs}
\usepackage{color}
\usepackage{tikz}
\usetikzlibrary{intersections} 
\usepackage{mathrsfs,amscd,amssymb,amsthm,amsmath,bm,graphicx,psfrag,subfigure,url}
\usepackage{tikz-cd}
\usepackage{graphicx}
\usepackage{subfigure}
\usepackage{tabularx}
\usepackage{floatrow}
\usepackage{colortbl}

\usetikzlibrary{shapes.geometric}

\newtheorem{theorem}{Theorem}[section]
\newtheorem{lemma}[theorem]{Lemma}
\newtheorem{proposition}[theorem]{Proposition}
\theoremstyle{definition}
\newtheorem{definition}[theorem]{Definition}
\newtheorem{example}[theorem]{Example}

\newtheorem{corollary}[theorem]{Corollary}
\theoremstyle{remark}
\newtheorem{remark}[theorem]{Remark}

\newtheorem{Convention}{Convention}

\usepackage{algorithm}        
\usepackage{algpseudocode}    
\usepackage{mathtools}        
\usepackage{enumitem}         
\usepackage{tikz,tikz-cd}
\usepackage{mathrsfs}
\usepackage[all]{xy}
\usepackage{tikz}
\usepackage{extarrows}
\usepackage{tikz-cd}
\usetikzlibrary{calc}
\usetikzlibrary{matrix,arrows,decorations.pathmorphing}
\usetikzlibrary{shapes.geometric,positioning}
\usetikzlibrary{arrows,decorations.pathmorphing,decorations.pathreplacing}
\usetikzlibrary{positioning,shapes,shadows,arrows}

\usepackage[colorlinks=true,hyperindex]{hyperref}
\newcommand\myshade{85}
\colorlet{mylinkcolor}{violet}
\colorlet{mycitecolor}{red}
\colorlet{myurlcolor}{cyan}

\hypersetup{
  linkcolor  = mylinkcolor!\myshade!black,
  citecolor  = mycitecolor!\myshade!black,
  urlcolor   = myurlcolor!\myshade!black,
  colorlinks = true,
}

\numberwithin{equation}{section}

\def\co{{\mathcal O}}

\def\cohx{{\rm coh}\mbox{-}\mathbb{X}(p,q)}

\definecolor{dark-green}{RGB}{14,150,2}
\definecolor{red}{RGB}{250,0,0}

\UseRawInputEncoding
\begin{document}

\title[Combinatorial Approaches to Exceptional Sequences for Weighted Projective Lines]{Combinatorial Approaches to Exceptional Sequences for Weighted Projective Lines of Type $(p,q)$}
\author{Jianmin Chen}
\address{School of Mathematical Sciences, Xiamen University, 361005, Fujian, PR China.}
\email{chenjianmin@xmu.edu.cn}

\author{Yiting Zheng}
\address{School of Mathematical Sciences, Xiamen University, 361005, Fujian, PR China.}
\email{ytzhengxmu@163.com}

\subjclass[2020]{05E10, 16G20, 05C30, 18E10}


\dedicatory{}

\keywords{exceptional sequence,  geometric realization, weighted projective line, tilting sheaf, lattice path}

\begin{abstract}
 We provide a combinatorial description of morphisms    in the coherent sheaf category  $\cohx$  over weighted projective line   of type $(p,q)$ via a marked annulus. This leads to
a  geometric realization
 of exceptional sequences 
in  $\cohx$.
As  applications, we 
  present a classification of complete exceptional sequences,   an effective method  for enlarging exceptional sequences, 
  and a new proof of  the transitivity of the braid group action on complete exceptional sequences.
 Besides,  we  
offer a combinatorial description of tilting bundles  via lattice paths and count the number of tilting sheaves in
  $\cohx$, up  to the Auslander-Reiten translation.  
\end{abstract}

\maketitle
\tableofcontents
\section{Introduction}
Exceptional objects and exceptional sequences are popular objects in various mathematical  fields.  
They were initially   introduced by Gorodentsev and Rudakov in their study of vector bundles over the projective plane \cite{MR885779,MR1074777}.   Later, Crawley-Boevey  introduced exceptional sequences in the module category of a finite  dimensional hereditary algebra over an algebraically closed field \cite{MR1206935}. Within this framework,  an exceptional sequence is called  a complete exceptional sequence if its length coincides with the rank of the Grothendieck group of the category. Crawley-Boevey  further established
the transitive action of the braid group on complete exceptional sequences, a result that was later  extended to hereditary Artin algebras by Ringel \cite{MR1293154}. 
Additionally, exceptional sequences have been explored for weighted projective lines by Meltzer \cite{MR1318999} and for Cohen-Macaulay modules over one-dimensional graded Gorenstein rings with simple singularities by Araya \cite{MR1816620}.

Exceptional sequences have established connections to several other areas of mathematics, including:
\begin{itemize}
    \item chains in the lattice of noncrossing partitions \cite{MR2032983,MR2575093,MR3551191};
\item the   theory   of cluster algebra via $c$-matrices \cite{MR3183889};
\item  factorizations of Coxeter elements \cite{MR2596373};
\item  $t$-structures and derived categories \cite{MR1039961,MR1074777,MR1973365}. 
\end{itemize}

Among the various researches on exceptional sequences, their combinatorial properties have  attracted much attention. 
For example, the enumeration of complete exceptional sequences for Dynkin types and affine type  $A$  has been studied in \cite{MR3919623,MR1842420,MR3018188,MR3145512,MR4546139,Maresca2022}.
And the enlargement property  of  exceptional sequences  and  the  transitivity 
 of the braid group action on  complete exceptional sequences in
 the derived category of a  gentle algebra have been investigated in  \cite{MR4578471}.
In this paper, we focus on 
exceptional sequences in the coherent sheaf category  $\cohx$  over weighted projective line  of type $(p,q)$.  We employ
  the geometric   model established in \cite{CRZ23} to  investigate  the geometric characterization of exceptional sequences in $\cohx$  and study their properties. 
As applications, we give a classification of complete exceptional sequences
in $\cohx$, reveal
the enlargement property  of  exceptional sequences  and  the  transitivity 
 of the braid group action on  complete exceptional sequences in $\cohx$ from combinatorial perspectives.
Notice that 
 the category  $\cohx$  is derived
equivalent to the finitely generated module category ${\rm mod}\tilde{A}_{p,q}$ of the canonical algebra $\tilde{A}_{p,q}$ (affine  type $A$) \cite{MR1318999}.  Thus, our work also     offers a new approach for studying  exceptional sequences in the category ${\rm mod}  \tilde{A}_{p,q}$.

 Let $A_{p,q}$  be an annulus with $p$ marked points on the inner boundary and $q$ marked points on the outer boundary. As shown in \cite{CRZ23}, there is a bijection $\phi$ between indecomposable sheaves
over $\cohx$ and certain homotopy classes of oriented curves in  $A_{p,q}$. 
Moreover, the dimension of extension space between two indecomposable coherent sheaves equals to the positive intersection number of the associated curves.  

Based on these  findings,  we give 
the unique decomposition of morphisms in   $\cohx$ into compositions of epimorphisms and monomorphisms through their corresponding oriented curves    in  Theorem~\ref{udecom}. Additionally, we explore
the kernels (resp. cokernels) of monomorphisms (resp. epimorphisms) between indecomposable
sheaves via oriented curves  in Theorem~\ref{thm:morphism}.

Furthermore, we introduce a family of ordered collections consisting of arcs in  $A_{p,q}$, called \emph{ordered exceptional collections} (see Definition~\ref{def:collection} and Definition~\ref{definition:ordered exceptional collection}). The  
  following theorem provides a geometric realization of exceptional sequences in $\cohx$.

\begin{theorem}[Theorem~\ref{collection and sequence}]
Let $\gamma_{1},\gamma_2,\ldots, \gamma_{s}$ be   arcs in  $A_{p,q}$. The following  are equivalent.
\begin{enumerate}
\item The ordered set  $(\gamma_1,\gamma_2,\ldots,\gamma_s)$
 is an ordered exceptional collection in $A_{p,q}$;
\item The sequence $(\phi( \gamma_{1}),\phi( \gamma_{2}),\ldots,\phi(\gamma_{s}))$ forms an exceptional sequence in $\cohx$.
\end{enumerate}
\end{theorem}

 As applications, we show  that any exceptional sequence in $\cohx$ can be enlarged into a complete
exceptional sequence (see Corollary~\ref{complete}). We also introduce a geometric method for constructing complete exceptional sequences,  thereby offering a classification of these sequences in   $\cohx$.

  As other applications, we define a braid group action  on ordered exceptional collections in $A_{p,q}$ (see Definition~\ref{mutation of arcs}) and demonstrate its compatibility with the braid group action on exceptional sequences in   $\cohx$ as follows.

\begin{theorem}[Theorem~\ref{mutation correspond}]
 Let   $(E,F)$ be an exceptional pair in $\cohx$ with  $\phi^{-1}(E)=\alpha$  and $\phi^{-1}(F)=\beta$, where $\alpha$ and  $\beta$ are  arcs in $A_{p,q}$. 
Then the following equations hold
\[\phi( L_{\alpha}\beta)=L_{E}F, \ \phi( R_{\beta}\alpha )=R_FE,\]
where $L_{\alpha}\beta$ (resp. $R_{\beta}\alpha$)  denotes left  mutation of $\beta$ at $\alpha$ (resp. right mutation of $\alpha$ at $\beta$), and  $L_{E}F$ (resp. $R_{F}E$) is left  mutation of $F$ at $E$ (resp. right mutation of $E$ at $F$).
\end{theorem}

By considering the  action of braid group on ordered maximal exceptional collections in $A_{p,q}$, we establish the transitivity of the braid group  $B_{p+q}$ on the set
of complete exceptional sequences in $\cohx$ (see Theorem~\ref{transitivity}).

As show in \cite[Theorem 4.3]{CRZ23}, tilting sheaves in  $\cohx$ are in bijective correspondence with triangulations of  $A_{p,q}$, which are a special class of maximal exceptional collections in $A_{p,q}$. For this  special case, we aim to  explore their combinatorial   characteristics. 
 
Two  tilting sheaves $T$ and $T'$ in $\cohx$ are said to be \emph{$\tau$-equivalent}, if there exists some  $k\in\mathbb{Z}$ such that
$T'=\tau^{-k}T$.  Our main result is as follows:

\begin{theorem}[Theorem~\ref{number  tilting bundles}]\label{l to l}
There is a one-to-one correspondence between  the   $\tau$-equivalence classes of tilting bundles in  $\cohx$ 
and the  lattice paths from $(0,0)$ to $(p,q)$, as defined in Definition~\ref{lattice and dyck}.  Consequently, 
the number of $\tau$-equivalence classes of tilting bundles in  $\cohx$ is given by
  $$\binom{p+q}{p}.$$ 
 
\end{theorem}

This theorem provides a combinatorial description of tilting bundles in $\cohx$ via lattice paths.  Lattice paths, as important  combinatorial structures, have broad applications across mathematics, chemistry, physics, and computer science. They can model various constructs such as trees, Young tableaux, continued fractions, and integer partitions \cite{MR4543742}. Moreover, their deep connections to symmetric function theory and group representation theory further highlight their significance.
 $(m, n)$-Dyck paths are a specific type of lattice paths. When  $m=n$, these paths are enumerated by the famous Catalan numbers, which have over 200 combinatorial interpretations \cite{MR3467982}. In algebra, $(m, n)$-Dyck paths are used to describe solutions to the consistency equations of twisted generalized Weyl algebras \cite{MR3514777} and to study the conjugation of core partitions in symmetric group representation theory \cite{MR3682397}. 
By Theorem~\ref{l to l}, we identify a class of tilting bundles in $\cohx$ that correspond to $(p, q)$-Dyck paths  (see Corollary~\ref{Dyck 1}). These connections enrich our understanding of tilting bundles in $\cohx$  and offer a novel perspective  on the importance of lattice paths in mathematics.

We also
derive the quantitative characteristic of  tilting sheaves in $\cohx$ as follows.

\begin{theorem}[Theorem~\ref{numbertilting sheaves}]
The number of tilting sheaves  in $\cohx$, up to $\tau$-equivalence, is 
$$
  \sum_{k=1}^{q} k   C_{p+q-k}  C_{k-1} +\sum_{l=1}^{p} l   C_{p+q-l} C_{l-1},
$$
where    $C_n$ denotes  the $n$-th Catalan number. 
\end{theorem}
The paper is organized as follows. Section~\ref{sec.2} provides   background material on   the category $\cohx$ of coherent sheaves over weighted projective line of type $(p,q)$ and exceptional sequences.
 In Section~\ref{sec.3}, we use the geometric model $A_{p,q}$ for $\cohx$ to compute the  epic-monic factorisation of a morphism and the kernel (resp. cokernel) of a monomorphism (resp. an epimorphism) in  $\cohx$. Section~\ref{sec.4} offers a geometric characterization of exceptional sequences in $\cohx$. Based on it, we classify complete exceptional sequences 
 and reveal the enlargement property of exceptional sequences   in  $\cohx$ from
combinatorial perspectives. 
In Section~\ref{sec.5}, we define a braid group action on ordered exceptional collections in $A_{p,q}$ and demonstrate its compatibility with the braid group action on exceptional sequences in $\cohx$. As an application, we establish the transitivity of the braid group action on the set of complete exceptional sequences in $\cohx$. 
In Section~\ref{sec.6},  we give a combinatorial description of tilting bundles and count the number of tilting sheaves in
  $\cohx$, up to the Auslander-Reiten translation.

\section{Preliminaries}\label{sec.2}

\subsection{Weighted projective lines of type $(p,q)$} 
We recall  from \cite{AMM,CRZ23,Geigle:Lenzing:1987} some definitions and fundamental properties   about weighted projective lines of type $(p, q)$, where $p$ and $q$ are positive integers.

Denote by $\mathbb{L}:=\mathbb{L}(p,q)$ the abelian group with generators $\vec{x}_1,\vec{x}_2$ and relations
$p\vec{x}_1= q \vec{x}_2 :=\vec{c},$ where the element $\vec{c}$ is called the \emph{canonical element}.  Consequently, each element $\vec{x} \in  \mathbb{L}$ can be uniquely written in \emph{normal form}  
$$\vec{x}= l_1 \vec{x}_1+l_2 \vec{x}_2+l\vec{c}\quad {\rm with}\ 0 \leq l_1 < p,\ 0 \leq l_2 < q\ {\rm and}\ l \in \mathbb{Z}.$$  The \emph{dualizing element} of  $\mathbb{L}$  is defined as   $\vec{\omega}=-(\vec{x}_1+\vec{x}_2)$.
Besides, the   group  $\mathbb{L}$ is ordered by defining the positive cone $\{\vec{x} \in  \mathbb{L} \mid \vec{x} \geq \vec{0}\}$ to consist of the elements of the form $l_1 \vec{x}_1+l_2 \vec{x}_2$, where $l_1, l_2 \geq 0$.
 A homomorphism $\delta:\mathbb{L}\to \mathbb{Z}$  defined on generators by $\delta(\vec{x}_1)=lcm(p,q)/p$ and  $\delta(\vec{x}_2)=lcm(p,q)/q$ is called the  \emph{degree map}.

Let $\Bbbk$ be an algebraically closed field and $\boldsymbol{\lambda}=(\lambda_1, \lambda_2)$ be a sequence of pairwise distinct closed points on the projective lines $\mathbb{P}_{\Bbbk}^1$. A \emph{weighted projective line}  $\mathbb{X}(p,q)$ of weight type $(p,q)$ and parameter sequence $\boldsymbol{\lambda}$ is obtained from the projective line  $\mathbb{P}_{\Bbbk}^1$ by attaching the weight $p,q$ to $\lambda_1, \lambda_2$, respectively. The parameter sequence can be normalized into $\lambda_1=\infty,\lambda_2=0$.

  The \emph{homogeneous coordinate algebra}   $S:=S(p,q)$ of the weighted projective line  $\mathbb{X}(p,q)$ is given by $\Bbbk[x_1, x_2]$,  
  which is $\mathbb{L}$-graded by setting $\mathrm{deg\ }x_i =\vec{x}_i$ for $i=1,2$. Hence, 
$S=\bigoplus_{\vec{x} \in \mathbb{L}} S_{\vec{x}},$
where $S_{\vec{x}}$ is the homogeneous component of degree $\vec{x}$. In addition, if $\vec{x}= l_1\vec{x}_1+l_2\vec{x}_2+l \vec{c} $ is
in normal form, 
{then ${\rm dim\ } S_{\vec{x}}=l+1$ with $l \geq -1$.}  
 We now revisit  the definition of the category $\cohx$ of coherent sheaves over $\mathbb{X}(p,q)$,  as   introduced in  \cite[Section 1]{Geigle:Lenzing:1987}.  
Let ${\rm mod}^{\mathbb{L}}\ S$ be the abelian category of finitely generated $\mathbb{L}$-graded $S$-modules, and 
${\rm mod}_0^{\mathbb{L}}\ S$ be its Serre subcategory formed by finite dimensional modules. Denote by $${\rm qmod}^{\mathbb{L}}\ S:={\rm mod}^{\mathbb{L}}\ S/{\rm mod}_0^{\mathbb{L}}\ S$$ the quotient abelian category. By \cite[Theorem 1.8]{Geigle:Lenzing:1987}, the sheafification functor yields an equivalence
\[
{\rm qmod}^{\mathbb{L}}\ S\xrightarrow{\sim} \cohx.
\] 
From now on, we will identify these two categories.  
The category $\cohx$ can be expressed as $$\cohx={\rm vect}\mbox{-}\mathbb{X}(p,q)\bigvee{\rm coh}_{0}\mbox{-}\mathbb{X}(p,q),$$ where ${\rm vect}\mbox{-}\mathbb{X}(p,q)$ (resp.  ${\rm coh}_{0}\mbox{-}\mathbb{X}(p,q)$) denotes the full subcategory of $\cohx$ consisting of coherent sheaves without any simple subobjects (resp. coherent sheaves of finite length), the symbol $\bigvee$ indicates that each indecomposable object of $\cohx$ is either in ${\rm vect}\mbox{-}\mathbb{X}(p,q)$ or in ${\rm coh}_{0}\mbox{-}\mathbb{X}(p,q)$, and there are no non-zero morphisms from ${\rm coh}_{0}\mbox{-}\mathbb{X}(p,q)$ to ${\rm vect}\mbox{-}\mathbb{X}(p,q).$
The objects in ${\rm vect}\mbox{-}\mathbb{X}(p,q)$ are called \emph{vector bundles}. There is a specific type of vector bundles  called \emph{line bundles}. Up to isomorphism, each line bundle has the form $\co(\vec{x})$ for a uniquely determined $\vec{x}\in\mathbb{L}$.   
Additionally, the homomorphism space between any two line bundles $\mathcal{O}(\vec{x})$ and $\mathcal{O}(\vec{y})$  in  $\cohx$ is given by:
 $${\rm Hom_{\cohx}}(\co(\vec{x}), \co(\vec{y}))\cong S_{\vec{y}-\vec{x}},\ \forall \vec{x}, \vec{y}\in\mathbb{L}.$$

For simplification of the notations, we will abbreviate ${\rm Hom_{\cohx}}$ and ${\rm Ext_{\cohx}}$ as ${\rm Hom}$ and ${\rm Ext}$. Now we give a list of fundamental properties of the category $\cohx$.
\begin{proposition}[\cite{{Geigle:Lenzing:1987}, Len2011,MR2232010}]\label{prop of coh}
The category $\cohx$ is connected, Hom-finite, hereditary, abelian  and  $\Bbbk$-linear with the following properties:
\begin{itemize}
\item[(1)] The category $\cohx$ admits Serre duality in the form 
${\rm Ext^{1}}(X,Y)  \cong {\rm DHom}(Y, \tau X),$
where $ { \mathrm D}={\rm Hom}_\Bbbk(-,\Bbbk)$ and  the $\Bbbk$-equivalence $\tau:\cohx\to\cohx$ is the shift $X \mapsto X(\vec\omega)$.
 \item[(2)]   Each indecomposable bundle over $\mathbb{X}(p,q)$ is a line bundle.
\item[(3)] 
 The rank of Grothendieck group $K_{0}(\cohx)$ is equal to $p+q$.
\end{itemize}
\end{proposition}

 By \cite[Proposition 1.1]{MR2232010}, the torsion subcategory ${\rm coh}_0\mbox{-}\mathbb{X}(p,q)$ of $\cohx$ decomposes into a coproduct
$\coprod_{\lambda\in \mathbb{P}_{\Bbbk}^1}\mathcal{U}_{\lambda}$, where $\mathcal{U}_{\lambda}$ is a connected uniserial length category, whose associated Auslander-Reiten quiver is a stable tube $\mathbb{ZA}_{\infty}/(\tau^r)$ for some $r\in \mathbb{Z}_{\geq 1}$ (c.f.\cite{SS07}). Here, the integral $r$ is called the rank of the stable tube $\mathbb{ZA}_{\infty}/(\tau^r)$ and depends on $\lambda$. Precisely, $r=p,\; q$ for $\lambda=\infty,\;0$, respectively and $r=1$ for $\lambda \in \Bbbk^{*}=\Bbbk\setminus \{ 0\}$. A stable tube of rank $1$ is called a \emph{homogeneous stable tube}. Objects that lie at the bottom of the stable tubes are all simple objects of $\cohx$. Each $\lambda \in  \Bbbk^{*}$ is associated with a unique simple sheaf $S_{\lambda}$, called \emph{ordinary simple}; while $\lambda=\infty$ (resp, $\lambda=0$) is associated with $p$ (resp, $q$) simple objects $S_{\infty,i}\, (i\in \mathbb{Z}/p\mathbb{Z}) $ (resp, $S_{0,i}\, (i\in \mathbb{Z}/q\mathbb{Z})$) called \emph{exceptional simples}.  For $\lambda \in \{\infty,0\}$ and each $j\in\mathbb{Z}_{>0}$, let  $S_{\lambda,i}^{(j)}$ denote the indecomposable object in $\mathcal{U}_{\lambda}$ of length $j$ with top $S_{\lambda,i}$. More precisely, the  composition series of $S_{\lambda,i}^{(j)}$ has the following form 
$$S_{\lambda,i-j+1}\hookrightarrow S_{\lambda,i-j+2}^{(2)}\hookrightarrow \ldots \hookrightarrow S_{\lambda,i-2}^{(j-2)} \hookrightarrow S_{\lambda,i-1}^{(j-1)} \hookrightarrow S_{\lambda,i}^{(j)}$$
with $S_{\lambda,i-k}^{(j-k)}/S_{\lambda,i-k-1}^{(j-k-1)}\cong S_{\lambda,i-k}$ for $0\leq k \leq j-2$. 

As is well known, for each $\vec{x}=l_1\vec{x}_1+l_2\vec{x}_2\in \mathbb{L}$, the twists act on the simple sheaves by  
\begin{equation}\label{shift}
S_{\infty,i}(\vec{x})=S_{\infty, i+l_{1}},\;\; S_{0,i}(\vec{x})=S_{0, i+l_{2}},\ S_{\lambda}(\vec{x})=S_{\lambda}\;\;{\rm for}\;\, \lambda \in \Bbbk^{*}.    
\end{equation}
Besides, for each ordinary simple sheaf $S_{\lambda}$, there is an
 exact sequence
$$0\longrightarrow \mathcal{O}\stackrel{u_{\lambda}}{\longrightarrow}\mathcal{O}(\vec{c})\longrightarrow S_{\lambda}\longrightarrow 0,$$
where the homomorphism $u_{\lambda}: \mathcal{O}\longrightarrow \mathcal{O}(\vec{c})$ is given by
multiplication with $x_2^{q}-\lambda x_1^{p}$.
In contrast, if $\lambda\in \{\infty, 0\}$, there are exact sequences
\begin{equation}\label{esp}
    0\longrightarrow \mathcal{O}((i-1)\vec{x}_1)\stackrel{u_{\infty}}{\longrightarrow} \mathcal{O}(i\vec{x}_1)\longrightarrow S_{\infty,i}\longrightarrow 0,\;\;i\in\mathbb{Z}/p\mathbb{Z};
\end{equation}
\begin{equation}\label{esq}
   0\longrightarrow \mathcal{O}((i-1)\vec{x}_2)\stackrel{u_{0}}{\longrightarrow} \mathcal{O}(i\vec{x}_2)\longrightarrow S_{0,i}\longrightarrow 0,\;\;i\in\mathbb{Z}/q\mathbb{Z}, 
\end{equation}
where $u_{\infty}$ (resp, $u_{0}$) is given by multiplication with $x_1$ (resp, $x_2$).

Lastly, let us recall that in a hereditary $\Bbbk$-category $\mathcal{H}$, an object $E$   is called \emph{exceptional} if  it satisfies the conditions ${\rm End}_{\mathcal{H}}(E)=\Bbbk$ and ${\rm Ext}_{\mathcal{H}}^1(E,E)=0$.
Furthermore, a sequence of exceptional objects $\epsilon =(E_1, \ldots, E_r)$ is said to be an \emph{exceptional sequence}  if both ${\rm Hom}_{\mathcal{H}}(E_t, E_s)=0$ and ${\rm Ext}_{\mathcal{H}}^1(E_t, E_s)=0$  hold for all $t>s$. Particularly,  if $r=2$ then $ \epsilon$ is referred to as an \emph{exceptional pair}, and if $r$ equals the rank of the Grothendieck group $K_0(\mathcal{H})$ then $ \epsilon$ is called a \emph{complete exceptional sequence}. 

Note that  the 
category $\cohx$ is hereditary. Herein, we revisit 
the braid group action on the set of (isomorphism classes)  exceptional sequences in $\cohx$, as outlined in \cite{AMM,M}.
 For  an exceptional pair  $(E,F)$ over $\mathbb{X}(p,q)$,  the trace map ${\rm can}:{\rm Hom}(E,F)\otimes_{\Bbbk} E \to F $ is defined conventionally by $\mathrm{can}(f\otimes a)= f(a)$. The literature  often refers to the image of ${\mathrm
{can}}$ as the trace of $E$ in $F$.  Additionally, if  
${\rm Hom}(E,F)\ne0$,  then the canonical map ${\mathrm {can}}$  is either surjective or injective, but not bijective.
The \emph{left mutation} of $(E,F)$  is the exceptional pair $(L_E F, E)$, where $ L_E F$  is determined as follows:
\begin{itemize}
\item
If ${\rm Hom}(E,F)= 0={\rm Ext}^{1}(E,F)$, then $L_E F=F$;
 \item
If $ {\rm Hom}(E,F)\neq 0$, then  $L_E F$ is given by one of
the following  exact sequences:
$$   0 \to L_E F  \to  {\rm Hom}(E,F) \otimes_{\Bbbk} E \stackrel{{\mathrm {can}}}{\to} F \to 0,$$
$$   0 \to  {\rm Hom}(E,F) \otimes_{\Bbbk} E \stackrel{{\mathrm {can}}}{\to} F\to L_E F  \to 0;$$
 \item If $ {\rm Ext}^{1}(E,F)\neq 0$, then $L_E F$ is given by
the following  exact sequence
$$   0 \to  F  \to L_E F   \to  {\rm Ext}^{1}(E,F) \otimes_{\Bbbk}  E \to 0, $$  which is the universal extension.
\end{itemize} 

Dually, the \emph{right mutation} of $(E,F)$  is the exceptional pair $(F, R_F E)$, where $ R_F E=E$ when  ${\rm Hom}(E,F)= 0={\rm Ext}^{1}(E,F)$;  otherwise  $R_F E$
is determined by one of the following  exact sequences
$$   0 \to E   \stackrel{{\mathrm {cocan}}}{\to} { \mathrm {DHom}}(E,F) \otimes_{\Bbbk} F \to  R_F E\to
0,$$
$$   0 \to R_F E  \to  E   \stackrel{{\mathrm {cocan}}}{\to} { \mathrm {DHom}}(E,F) \otimes_{\Bbbk} F \to
0,$$
$$   0 \to   { \mathrm{ DExt}}^1(E,F) \otimes_{\Bbbk}  F    \to  R_F E  \to   E \to 0.$$
Here ${\mathrm {cocan }}$ denotes the co-canonical map and  the third sequence is
the universal extension.
 
Let $B_r$ be the \emph{braid group} on $r$ strings, defined by
$$B_r= \langle {\rm \sigma}_1, \ldots , {\rm \sigma}_{r-1} | {\rm \sigma}_s {\rm \sigma}_t ={\rm \sigma}_t{\rm \sigma}_s \;
 {\mathrm {for} } \; s-t\geq 2 \, \; {\mathrm {and} }\; {\rm \sigma}_s {\rm \sigma}_{s+1}{\rm \sigma}_s={\rm \sigma}_{s+1}{\rm \sigma}_s{\rm \sigma}_{s+1} \rangle.$$
The group $B_r$ acts on the set  of exceptional sequences of length $r$ in $\cohx$ by
$${\rm \sigma }_s \cdot (E_1, \ldots, E_{s-1}, E_{s},  E_{s+1},\ldots,  E_{r})=
(E_1, \ldots, E_{s-1}, L_{E_s} E_{s+1},E_s,  E_{s+2} ,\ldots,  E_{r}),$$
 $${\rm \sigma }_s^{-1} \cdot (E_1, \ldots, E_{s-1},  E_{s},  E_{s+1},\ldots,  E_{r})=
(E_1, \ldots, E_{s-1},E_{s+1},  R_{E_{s+1}} E_{s},  E_{s+2} ,\ldots,  E_{r}).$$


\subsection{The geometric model of $\cohx$}\label{geometric modle}
  We recall a geometric model of $\cohx$  
 from \cite{CRZ23}.  
Let $A_{p,q}$ be an annulus with $p$ marked points on the inner boundary labeled as  $0_{\partial},\,(\frac{1}{p})_{\partial},\,\ldots,\, (\frac{p-1}{p})_{\partial}$  in counterclockwise order, and $q$ marked points on the outer boundary labeled as $0_{\partial^{\prime}},\, (\frac{q-1}{q})_{\partial^{\prime}},\,\ldots,\, (\frac{1}{q})_{\partial^{\prime}}$ in clockwise order
(see Figure~\ref{marked surface}). Without
 loss of generality, assume that $1\leq p \leq q$ and the marked points are distributed in equidistance on the two boundaries of $A_{p,q}$. 
\begin{figure}[H]
\begin{tikzpicture}[scale=0.7]
\draw[](0,0)arc(0:-540:2);
\draw[](-1,0)arc(0:360:1);
\draw[dashed,red,thick,->](-0.8,0)arc(0:50:1.2);
\draw[dashed,red,thick,->](0.2,0)arc(0:-30:2.2);
\node()at(-2,-1){$\bullet$};
\node()at(-2,-1.2){\tiny{$0_{\partial}$}};
\node()at(-1.3,-0.7){$\bullet$};
\node()at(-0.8,-0.7){\tiny{$(\frac{1}{p})_{\partial}$}};
\node()at(-2.7,-0.7){$\bullet$};
\node()at(-3.3,-0.9){\tiny{$(\frac{p-1}{p})_{\partial}$}};
\node()at(-2,-2){$\bullet$};
\node()at(-2,-2.3){\tiny{$0_{\partial^{\prime}}$}};
\node()at(-2.58,-1.92){$\bullet$};
\node()at(-3.2,-2.1){\tiny{$(\frac{q-1}{q})_{\partial^{\prime}}$}};
\node()at(-1.42,-1.92){$\bullet$};
\node()at(-0.8,-2){\tiny{$(\frac{1}{q})_{\partial^{\prime}}$}};
\end{tikzpicture}
\caption{Marked surface $A_{p,q}$}\label{marked surface}
\end{figure}
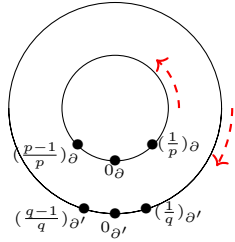
Denote by ${\rm Cyl_{p,q}}$  a rectangle of height 1 and width $1$ in $\mathbb{R}^{2}$ with $p$ marked points on the upper boundary $\partial$ and $q$ marked points on the lower boundary $\partial^{\prime}$, identifying its two vertical sides. 
Then  $A_{p,q}$ can be identified with ${\rm Cyl_{p,q}}$ in the sense of viewing the upper (resp. lower) boundary of ${\rm Cyl_{p,q}}$ as the inner (resp. outer) boundary of $A_{p,q}$, where the marked points of ${\rm Cyl_{p,q}}$ are oriented  from left to right on the upper boundary $\partial$ and from right to left on the lower boundary $\partial^{\prime}$. 

Let $(\mathbb{U}, \pi)$ be the universal cover  of ${\rm Cyl_{p,q}}$, where $\mathbb{U}=\{(x,y)\in\mathbb{R}^{2}|0\leq y\leq 1\}$ inherits the orientation of ${\rm Cyl_{p,q}}$  and the covering map 
$$\pi:=\pi_{p,q}: \mathbb{U}\to {\rm Cyl_{p,q}},\;\;(x,y)\mapsto(x-\lfloor x\rfloor, y).$$
Here, $\lfloor x\rfloor$ denotes the integer part of $x$. Naturally,  $\pi$ is also a covering map of $A_{p,q}$. Denote the marked point $(i, 1)$ on the upper boundary $\partial$ of $\mathbb{U}$ by $i_{\partial}$ and the marked point $(j, 0)$ on the lower boundary $\partial^{\prime}$ of $\mathbb{U}$ by $j_{\partial^{\prime}}$, for $ i\in \frac{\mathbb{Z}}{p}$ and $j\in \frac{\mathbb{Z}}{q}$.

Now we recall some definitions of curves and arcs in  $\mathbb{U}$ and $A_{p,q}$.  
In this paper, we consider curves  up to homotopy and all curves are always assumed to be in minimal position.

\begin{definition}\cite{BZ2011,CRZ23,V2018}
A {\emph {curve}} in $\mathbb{U}$ is a continuous function $c:[0,1]\longrightarrow \mathbb{U}$ such that $c(t)\in \mathbb{U}^0:=\mathbb{U}\setminus\{\partial, \partial^{\prime}\}$ for any $t\in (0,1)$. An {\emph {arc}} in $\mathbb{U}$ is a curve whose endpoints are marked points of $\mathbb{U}$, satisfying that it does not intersect itself in the interior of $\mathbb{U}$, the interior of the arc is disjoint from the boundary of $\mathbb{U}$, and it does not cut out a monogon or digon.
\end{definition}

Denoted by $[x_{b_1}, y_{b_2}]$ the 
  arc in $\mathbb{U}$  starts at a marked point $x_{b_1}$ and ends at a marked point $y_{b_2}$ with  $x, \,y\in\{ \frac{\mathbb{Z}}{p},\frac{\mathbb{Z}}{q}\}$, $b_{1},b_{2}\in \{\partial,\,\partial^{\prime}\}$.
 There are two main types oriented arcs in $\mathbb{U}$.

\begin{definition}[\cite{BT2020,CRZ23}] \label{peripheral arc and bridging arc}
Let $\gamma=[x_{b_1}, y_{b_2}]$ be an arc in $\mathbb{U}$. If $b_1\neq b_2,$ $\gamma$ is called a \emph{bridging arc}.
 If $b_1=b_2=\partial$ and $y-x\geq \frac{2}{p}$ or $b_1=b_2=\partial^{\prime}$ and $y-x\geq\frac{2}{q}$, then $\gamma$ is called a \emph{peripheral arc}. In particular,   $[x_{\partial^{\prime}}, y_{\partial}]$ is called a \emph{positive bridging arc}.
\end{definition}

Similarly,  in $A_{p,q}$, a \emph{bridging (\emph{resp.} peripheral) curve} is defined as the image   $\pi(\gamma)$ for a  bridging (resp.  peripheral) arc  $\gamma$ in $\mathbb{U}$. If $\pi(\gamma)$  itself is an arc, it is specifically termed a \emph{bridging} (resp. \emph{peripheral}) \emph{arc}  in $A_{p,q}$.   
For the sake of simplicity, from now on,  denote by
$$D^{\frac{i}{p}}_{\frac{j}{q}}:=[(\frac{j}{q})_{\partial^{\prime}}, (\frac{i}{p})_{\partial}],\;\;D^{\frac{i}{p}, \frac{j}{p}}:=[(\frac{i}{p})_{\partial}, (\frac{j}{p})_{\partial}]\;\;{\rm and}\;\;D_{\frac{i}{q}, \frac{j}{q}}:=[(\frac{i}{q})_{\partial^{\prime}}, (\frac{j}{q})_{\partial^{\prime}}]$$
for positive bridging arcs and peripheral arcs in $\mathbb{U}$,
and denote their image under $\pi: \mathbb{U}\to A_{p,q}$ as $[D^{\frac{i}{p}}_{\frac{j}{q}}],\ [D^{\frac{i}{p}, \frac{j}{p}}]$ and $[D_{\frac{i}{q}, \frac{j}{q}}]$ respectively.  
\begin{definition}[\cite{CRZ23}]\label{geometric}
For $n\in \mathbb{Z}_{\geq1}$, an $n$-cycle in $A_{p,q}$ is a curve $\pi(c)$, where $c$ is a curve in $\mathbb{U}^{0}$ with $c(1)-c(0)=(n,0).$ In particular, the 1-cycle will be called a \emph{loop}. For the notion of $\Bbbk^{*}$-parameterized $n$-cycles, we refer to the set $\{(\lambda, L^{n})\,|\,\lambda\in \Bbbk^{*}\}$, where $L$ is a loop in $A_{p,q}$ with the orientation in an anti-clockwise direction.
\end{definition}

Let $\mathcal{C}$ be the set consisting of the following elements with $i, j\in\mathbb{Z}$ and $n\geq 1$:
\begin{itemize}
\item[-]  
$[D^{\frac{i}{p}}_{\frac{j}{q}}]$;
    \item[-]  
$[D^{\frac{i}{p}, \frac{j}{p}}]$ and $[D_{\frac{i}{q}, \frac{j}{q}}]$, with $j-i\geq 2;$
 \item[-]   
$\{(\lambda, L^{n})\,|\,\lambda\in \Bbbk^{*}\}$.
\end{itemize}
Then it is shown in \cite{CRZ23} that the following assignments:
\begin{align}\label{map}\nonumber
&[D^{\frac{i}{p}}_{\frac{j}{q}}]\mapsto\co(i\vec{x}_1-j\vec{x}_2),\ [D^{\frac{i-j-1}{p}, \frac{i}{p}}]\mapsto S_{\infty,i}^{(j)},  \  [D_{-\frac{i}{q}, \frac{j-i+1}{q}}]\mapsto S_{0,i}^{(j)},\
 (\lambda, L^{j}) \mapsto S_{\lambda}^{(j)},\nonumber
\end{align}
define a bijective map $\phi:\mathcal{C}\longrightarrow{\rm ind}(\cohx)$, where ${\rm ind}(\cohx)$ denotes the full subcategory of $\cohx$ consisting of all indecomposable objects.

Based on this bijection, we recall  the geometric interpretation for the Auslander-Reiten sequences and the dimension of extension groups in $\cohx$. 
\begin{definition}[\cite{BT2020,BZ2011,CRZ23}]
For any oriented curve $\gamma$ in $A_{p,q}$, denote by $_{s}\gamma$ the \emph{elementary move of $\gamma$ on its starting point}, meaning that the curve $_{s}\gamma$ is obtained from $\gamma$ moving its starting point
to the next marked point on the same boundary, such that $_{s}\gamma$ is rotated in clockwise direction around the ending point of $\gamma.$
Similarly, denote by $\gamma_{e}$ the \emph{elementary move of $\gamma$ on its ending point}. Iterated elementary moves are denoted by $_{s}\gamma_{e}=_{s}(\gamma_{e})=(_{s}\gamma)_{e}$,  $_{s^2}\gamma=_{s}(_{s}\gamma)$ and $ \gamma_{e^2}=(\gamma_{e})_{e}$, respectively.
\end{definition}

\begin{proposition}[\cite{CRZ23}]\label{prop:tau}
Let  $X\in \cohx$ be a line bundle or an indecomposable torsion sheaf supported at an exceptional point.   Assume $\phi^{-1}(X)=\gamma$. Then $\tau^{-1}X=\phi(_{s}\gamma_{e})$ and  
the Auslander-Reiten sequence  starting at $X$ has the form
$$0\longrightarrow \phi(\gamma)\longrightarrow \phi(_{s}\gamma)\oplus\phi(\gamma_{e})\longrightarrow\phi(_{s}\gamma_{e})\longrightarrow 0.$$
\end{proposition}
\begin{remark}
For any oriented curve $\gamma$ in $A_{p,q}$, we denote by  $_{s^{-1}}\gamma_{e^{-1}}$ the curve obtained by $\gamma$ moving both the starting point and ending point simultaneously to the  previous  marked points on the same boundary. Then we have
$\phi( _{s^{-1}}\gamma_{e^{-1}})=\tau \phi( \gamma)$.
\end{remark}
\begin{definition}[\cite{BM2012,W2008}]\label{positive intersection}
A point of intersection of two oriented curves $\gamma_{1}$ and $\gamma_{2}$ in $A_{p,q}$ is called \emph{positive intersection} of $\gamma_{1}$ and $\gamma_{2}$, if $\gamma_{2}$ intersects $\gamma_{1}$ from the right, that is, the picture looks like this:
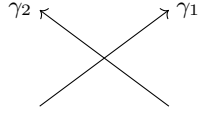
\begin{figure}[H]
\begin{tikzpicture}[scale=0.85]
\draw[->](0,0)--(2,1.5);
\draw[->](2,0)--(0,1.5);
\node()at(-0.3,1.5){\small$\gamma_{2}$};
\node()at(2.3,1.5){\small$\gamma_{1}$};
\end{tikzpicture}
\caption{Positive intersection}
\label{fig:Positive intersection.}
\end{figure}
\end{definition}
\noindent Denoted by
$$I^{+}(\gamma_{1},\gamma_{2})=\mathop{{\rm min}}\limits_{\gamma_{1}^{\prime}\sim\gamma_{1},\ \gamma_{2}^{\prime}\sim\gamma_{2}}|\gamma_{1}^{\prime}\cap^{+}\gamma_{2}^{\prime}| $$
where $\gamma_{1}^{\prime}\cap^{+}\gamma_{2}^{\prime}$  is the set of the positive intersections of $\gamma_{1}^{\prime}$ and $\gamma_{2}^{\prime}$, excluding their endpoints, and the sign $a\sim b$ means that the  curves $a$ and $b$ are homotopy. 

\begin{theorem}[\cite{CRZ23}]\label{dimension and positive intersection}
Let $X, Y$ be two indecomposable objects in $\cohx$. Then \[ {\rm dim}_{\Bbbk}{\rm Ext}^{1}(X,Y) =I^{+}(\phi^{-1}(X), \phi^{-1}(Y)).\]
 
\end{theorem}

\begin{proposition}[\cite{CRZ23}]\label{prop:ext}
Let $\gamma_1,\gamma_2$ be positive bridging arcs or peripheral arcs in $\mathbb{U}$. If $I^{+}(\gamma_1,\gamma_2)=1$, then there is a natural short exact sequence in $ \cohx$ associated to the intersection in geometric terms. Precisely, consider the following figure of the chosen intersection, we have a short exact sequence 
\begin{equation}\label{3.2}
0\to \phi(\pi(\gamma_{2})) \to \phi(\pi(\gamma_{3}))\oplus\phi(\pi(\gamma_{4})) \to \phi(\pi(\gamma_{1}))\to 0.
\end{equation} 
\begin{figure}[H]
\begin{tikzpicture}[scale=0.85]
\node(1)at(-0.1,-0.06){\tiny{}};
\node(2)at(2.05,-0.06){\tiny{}};
\node(3)at(0,2){\tiny{}};
\node(4)at(2,2){\tiny{}};
\draw[->](1)--(4);
\draw[->](2)--(3);
\draw[->,green](0.05,0.1) arc(330:390:1.7);
\draw[->,green](1.95,0.1) arc(210:150:1.7);
\node at(1.7,2){\small{$\gamma_1$}};
\node at(0.3,2){\small{$\gamma_2$}};
\node[ green] at(2,0.95){\small{${ \gamma_4}$}};
\node[ green] at(0,0.95){\small{${ \gamma_3}$}};
\end{tikzpicture}
\caption{Short exact sequence}\label{labelx}
\end{figure}
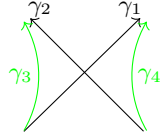
\end{proposition}

\section{The geometric interpretation of morphisms in $\cohx$} \label{sec.3}
In this section, we study the  morphisms in $\cohx$  via the geometric model established in~\cite{CRZ23}. Since $\cohx$ is an abelian category, we consider the epic-monic factorisation of a  morphism and  the kernel (resp. cokernel) of a monomorphism (resp. an epimorphism) in $\cohx$.

First, the dimension of morphism space between two  indecomposable  sheaves over $\cohx$   can be calculated 
via the common points of the corresponding  curves in   $A_{p,q}$.

\begin{proposition}\label{hom condition} 
 Let   $X$ and $Y$ be
 indecomposable   bundles or  torsion sheaves supported at exceptional points.
Assume   that $ \phi  ^{-1}(X)=\gamma_1$ and   $ \phi ^{-1}(Y)=\gamma_2$. 
The dimension of the morphism space ${\rm Hom}(X, Y)$ equals $k$ if and only if there are exactly $k$ common points between $\gamma_1$ and $\gamma_2$ that satisfy the following conditions: 
  
\begin{itemize}
    \item[(1)]    At each common point, $\gamma_2$ follows $\gamma_1$ in the clockwise order.
  \item[(2)]      These common points are not the   intersections of $\gamma_1$ and $\gamma_2$ depicted in Figure~\ref{fig:morphism point}.
    \end{itemize} 
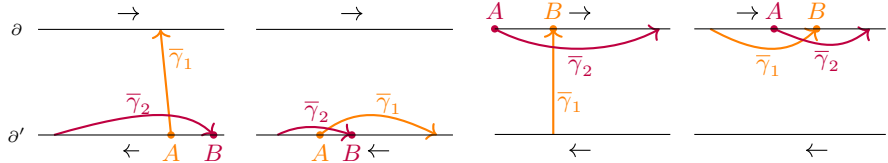
\begin{figure}[h]
    \centering
\begin{tikzpicture}[scale=0.7]
\coordinate (3) at (1.7,-0.4);
\coordinate (5) at (0.9,-0.1);
\coordinate (7) at (0.1-0.3,0.35);
\coordinate (8) at (3,-0.25);
\coordinate (9) at (2.5,0.45);
\coordinate (orangeStart) at (2.3, -1);  
\coordinate (purpleEnd) at (2.1+1,-1);   

\node[orange] at (9){\small$ {\overline{\gamma}_1}$};
\node[purple] at (3){\small$ {\overline{\gamma}_2}$};
 
\draw [->,orange,thick](orangeStart) -- (2.6-0.5,1);
\fill[orange]  (orangeStart) circle (2pt) node[below  ]{\small$A$}; 
\draw [->,purple,thick](-0.4+0.5,-1) .. controls (1.3+0.5,-.5) and (2.6,-.5) .. (purpleEnd);
\fill[purple] (purpleEnd) circle (2pt) node[below]{\small$B$};  
\draw (-0.2,1) -- (3.3,1);
\draw (-0.2,-1) -- (3.3,-1);
 
\draw[->] (1.3,1.3) -- (1.7,1.3);
\draw[<-] (1.4,-1.3) -- (1.7,-1.3);
\node()at(-0.6,1){\tiny{${\partial}$}};
\node()at(-0.6,-1){\tiny{${\partial'}$}};
\end{tikzpicture}
\hspace{0.05cm}
\begin{tikzpicture}[scale=0.7]
\coordinate (3) at (1.7,-0.4);
\coordinate (5) at (1.1,0.1);
\coordinate (6) at (3.7,-0.1);
\coordinate (9) at (0.35,-0.55);
\coordinate (orangeStart) at (-0.6+1,-1);  
\coordinate (purpleEnd) at (1,-1);        

\node[orange] at (3){\small$ {\overline{\gamma}_1}$};
\node[purple] at (9){\small$ {\overline{\gamma}_2}$};
 
\draw [->,orange,thick] (orangeStart) .. controls
(1,-.5) and (2.6-1,-.5) ..(2.1+0.5,-1);
\fill[orange] (orangeStart) circle (2pt) node[below]{\small$A$}  
;
 
\draw [->,purple,thick] (-0.4,-1) .. controls
(0,-0.8) and (0.4,-0.8) .. (purpleEnd);
\fill[purple] (purpleEnd) circle (2pt) node[below]{\small$B$} 
;
 
\draw (-.8,1) -- (2.9,1);
\draw (-.8,-1) -- (2.9,-1);

\draw[->] (.8,1.3) -- (1.2,1.3);
\draw[<-] (1.3,-1.3) -- (1.7,-1.3);

\end{tikzpicture}
\hspace{0.05cm}
\begin{tikzpicture}[scale=0.7]
\coordinate (2) at (0.2,0.3);
\coordinate (5) at (1,0);
\coordinate (7) at (0.1,0);
\coordinate (9) at (2.4,0);
\coordinate (8) at (-.9,0.25);
\node[below right,orange] at (0,0){\small$ {\overline{\gamma}_1}$};
\node[right,purple] at (2){\small$ {\overline{\gamma}_2}$};
 
\fill[purple] (-1,1) circle (2pt) node[above]{\small$A$}; 
 
\fill[orange] (0.1,1) circle (2pt) node[above]{\small$B$}; 

\draw[->,orange,thick] (-0.4+0.5,-1) -- (0.1,1);
\draw [->,purple,thick](-1,1) .. controls
(1.3-1.4,0.5) and (2.6-1.4,0.5) ..(2.6-0.5,1);
\draw (-1,1) -- (2.3,1);
\draw (-1,-1) -- (2.3,-1);
\draw[->] (0.4,1.3) -- (0.8,1.3);
\draw[<-] (0.4,-1.3) -- (0.8,-1.3);
\node at (0.45,-1.5){};
\end{tikzpicture}
\hspace{0.05cm}
\begin{tikzpicture}[scale=0.7]
\coordinate (2) at (0.45,0.55);
\coordinate (5) at (1,0);
\coordinate (7) at (0.1-0.4,0.7);
\coordinate (9) at (2.6,0);
\coordinate (8) at (-.9,0.25);
\node[below,orange] at (7){\small$ {\overline{\gamma}_1}$};
\node[purple] at  (0.8,0.4){\small$ {\overline{\gamma}_2}$};
 
\fill[purple] (-0.2 ,1) circle (2pt) node[above]{\small$A$}; 
 
\fill[orange] (2.6-2,1) circle (2pt) node[above]{\small$B$}; 

\draw [->,orange,thick](0.1-0.5-1,1) .. controls
(1.3-1.8,0.5) and (2.6-2.5,0.5) ..(2.6-2,1);
\draw [->,purple,thick](-0.2 ,1) .. controls
(0.7,0.6) and (1.1,0.6) ..(1.6,1);
\draw (-1.7,1) -- (1.9,1);
\draw (-1.7,-1) -- (1.9,-1);
\draw[->] (-0.9,1.3) -- (-0.5,1.3);
\draw[<-] (0.4,-1.3) -- (0.8,-1.3);
\coordinate (22) at (0.45,-1.5);
\node at (0.45,-1.5){};
\end{tikzpicture}
    \caption{$\overline{\gamma}_1$  and $\overline{\gamma}_2$  are arcs in $\mathbb{U}$  such that $\pi(\overline{\gamma}_1)=\gamma_1$ and $\pi(\overline{\gamma}_2)=\gamma_2$, with no marked points between points $A$ and $B$
    }
    \label{fig:morphism point}
\end{figure} 
\end{proposition}
\begin{proof}
It is a direct consequence of Proposition~\ref{prop of coh}(1), Theorem~\ref{dimension and positive intersection} and Proposition~\ref{prop:ext}.
\end{proof}

 Before  further exploring the morphisms  in $\cohx$  via  the geometric model, we show some properties  of  morphisms in $\cohx$ that will be used subsequently. The following proposition gives a simple observation  regarding monomorphisms and epimorphisms of indecomposable sheaves  in $\cohx$.

\begin{proposition}\label{hom:mon,epi}
Let   $X$ and $Y$ be two  indecomposable  sheaves over $\cohx$. 
\begin{enumerate}
\item
 The morphism $f:X\to Y$ is a proper monomorphism if and only if   one of the following holds: 
\begin{enumerate}
\item[(a)] $X=\mathcal{O}(\vec{x})$ and $Y=\mathcal{O}(\vec{y})$, where $\vec{x},\vec{y} \in\mathbb{L}$ and $\vec{x}<\vec{y}$.  
\item[(b)] $X=S_{\lambda}^{(j)}$ and $Y=S_{\lambda}^{(j+m)}$, where  $\lambda\in\Bbbk^{*}$ and $j,m\in\mathbb{Z}_{>0}$.
\item[(c)] $X=S_{\lambda,i}^{(j)}$ and $Y=S_{\lambda,i+m}^{(j+m)}$, where $\lambda\in\{\infty,0\}$, $i\geq0$ and $j,m\in\mathbb{Z}_{>0}$. 
\end{enumerate}
\item The morphism  $g:X\to Y$  is a proper epimorphism if and only if   one of the following holds:
\begin{enumerate}
\item 
$X=\mathcal{O}(\vec{x})$ and $Y=S_{\lambda}^{(j)}$, where  $\vec{x} \in\mathbb{L}$,  $\lambda\in\Bbbk^{*}$  and $j\in \mathbb{Z}_{>0}$.
\item $X=\mathcal{O}(\vec{x})$ and $Y=S_{\lambda,0}^{(j)}(\vec{x})$, where  $\vec{x} \in\mathbb{L}$,  $\lambda\in\{\infty,\;0\}$  and $j\in \mathbb{Z}_{>0}$. 
\item  $X=S_{\lambda}^{(j)}$ and $Y=S_{\lambda}^{(j-m)}$, where  $\lambda\in\Bbbk^{*}$ and  $m,j-m\in \mathbb{Z}_{>0}$.
\item $X=S_{\lambda,i}^{(j)}$ and $Y=S_{\lambda,i}^{(j-m)}$, where $\lambda\in\{\infty,0\}$, $i\geq0$ and $m,j-m\in \mathbb{Z}_{>0}$.
\end{enumerate}
\end{enumerate}
\end{proposition}

\begin{proposition}\label{cro:inj,epi}
 Let   $X$ and $Y$ be two  indecomposable  sheaves over $\cohx$.
 Suppose    $ \phi^{-1}(X)=\gamma_1$,  $\phi^{-1}(Y)=\gamma_2$, 
and $Y$ is an indecomposable torsion sheaf supported at an exceptional point. Then
  \begin{itemize}
\item[(1)]  $\gamma_1$ and $\gamma_2$ share the starting (resp. ending) point $M$ on the outer (resp. inner)   boundary of  $A_{p,q}$ with $\gamma_2$ following $\gamma_1$ in the clockwise order at $M$ 
    if and only if there is  an epimorphism $f_M : X\to Y $    associated with 
    $M$(see  Figure~\ref{epi+mon}(a)-(d)); 
\item [(2)]  
$\gamma_1$ and $\gamma_2$ share the  ending (resp. starting)   point  $M$ on    the outer (resp. inner)   boundary of  $A_{p,q}$, with $\gamma_2$ following $\gamma_1$ in the clockwise order at $M$ if and only if there is  a    monomorphism $f_M : X\to Y $  associated with  $M$(see  Figure~\ref{epi+mon}(e)(f)).
\end{itemize}
\begin{figure} [H]
\begin{tikzpicture}[scale=0.7]
\coordinate (3) at (2,-0.35);
\coordinate (5) at (0.9-0.5,-0.1);
\coordinate (6) at (3.1,-0.1);
\coordinate (7) at (0.1-0.5,0.35);
\coordinate (8) at (0,-0.35);
\coordinate (9) at (2.1,0.25);
\node[left,purple]  at (5){\small$\overline{\gamma}_1$};
\node[orange] at (3){\small$\overline{\gamma}_2$};
\draw [->,purple,thick](-0.4+0.5,-1) --(2-1,1);
\draw [->,orange,thick](-0.4+0.5,-1) .. controls
(1.3+0.5,-.5) and (2.4,-.5) ..(2.5,-1);

\draw (-0.1,1) -- (3.2,1);
\draw (-0.1,-1) -- (3.2,-1);
\node()at(-0.5,1){\tiny{${\partial}$}};
\node()at(-0.5,-1){\tiny{${\partial'}$}};
\draw[->] (1.3,1.3) -- (1.7,1.3);
\draw[<-] (1.3,-1.3) -- (1.7,-1.3);
\node at (1.5,-1.8){\small$(a)$};
\end{tikzpicture}
\hspace{0.01cm}
\begin{tikzpicture}[scale=0.7]
\coordinate (3) at (1.7,-0.85);
\coordinate (5) at (1.1,0.1);
\coordinate (6) at (4.2-0.5,-0.1);
\coordinate (9) at (0.5,-0.7);
\node[right,purple]  at (5){\small$\overline{\gamma}_1$};
\node[orange] at (9){\small$\overline{\gamma}_2$};
\draw [->,purple,thick](-0.4,-1) .. controls
(1.3+0.5-1,0.05) and (2.6-1,0.05) ..(2.1+0.5,-1);
\draw (-.5,1) -- (2.8,1);
\draw (-.5,-1) -- (2.8,-1);
\draw[->] (1,1.3) -- (1.4,1.3);
\draw[<-] (1,-1.3) -- (1.4,-1.3);
\draw [->,orange,thick](-0.4,-1) .. controls
(0,-0.8) and (0.4,-0.8) ..(1,-1);
\node at (1.2,-1.8){\small$(b)$};
\end{tikzpicture}
\hspace{0.01cm}
\begin{tikzpicture}[scale=0.7]
\coordinate (2) at (0.6,0.3);
\coordinate (3) at (2.7,0.5);
\coordinate (5) at (1.8,0);
\coordinate (6) at (-2+.75,0);
\coordinate (7) at (0.1+0.5,0);
\coordinate (9) at (2.6,0);
\node[left,purple]  at (5){\small$\overline{\gamma}_1$};
\node[orange] at (2){\small$\overline{\gamma}_2$};
\draw [->,purple,thick](1.5,-1) --(2.6-0.5,1);
\draw [->,orange,thick](0.1-0.5,1) .. controls
(1.3-1.4,0.5) and (2.6-1.4,0.5) ..(2.6-0.5,1);
\draw (-0.7,1) -- (2.3,1);
\draw (-0.7,-1) -- (2.3,-1);
\draw[->] (0.5,1.3) -- (0.9,1.3);
\draw[<-] (0.5,-1.3) -- (0.9,-1.3);
\node at (0.7,-1.8){\small$(c)$};
\end{tikzpicture}
\hspace{0.05cm}
\begin{tikzpicture}[scale=0.7]
\coordinate (2) at (0.45,0.55);
\coordinate (5) at (1,-0.1);
\coordinate (7) at (0.1-0.5,0.75);
\coordinate (9) at (2.6,0);
\coordinate (8) at (-.9,0.25);
\node[left,purple]  at (5){\small$\overline{\gamma}_1$};
\node[orange] at (2){\small$\overline{\gamma}_2$};
\draw [->,purple,thick](0.1-0.5-1,1) .. controls
(1.3-1.8,-0.1) and (2.6-1.8,-0.1) ..(2.6-1,1);
\draw [->,orange,thick](0.2,1) .. controls
(0.7,0.6) and (1.1,0.6) ..(1.6,1);
\draw (-1.6,1) -- (1.7,1);
\draw (-1.5,-1) -- (1.7,-1);
\draw[->] (0,1.3) -- (0.4,1.3);
\draw[<-] (0,-1.3) -- (0.4,-1.3);
\node at (0.1,-1.8){\small$(d)$};
\end{tikzpicture}
\hspace{0.01cm}
\begin{tikzpicture}[scale=0.7]
\coordinate (3) at (1.7,-0.9);
\coordinate (5) at (1.1,0.1);
\coordinate (6) at (4.2-0.5,-0.1);
\coordinate (9) at (0.5,-0.7);
\node[right,orange]  at (5){\small$\overline{\gamma}_2$};
\node[purple] at (3){\small$\overline{\gamma}_1$};
\draw [->,orange,thick](-0.4,-1) .. controls
(1.3+0.5-1,0.05) and (2.6-1,0.05) ..(2.1+0.5,-1);
\draw [->,purple,thick](-0.4+1,-1) .. controls
(1,-.5) and (2.6-1,-.5) ..(2.1+0.5,-1);
\draw (-.5,1) -- (2.7,1);
\draw (-.6,-1) -- (2.7,-1);
\draw[->] (1,1.3) -- (1.4,1.3);
\draw[<-] (1,-1.3) -- (1.4,-1.3);
\node at (1.2,-1.8){\small$(e)$};
\end{tikzpicture}
\hspace{0.01cm}
\begin{tikzpicture}[scale=0.7]
\coordinate (2) at (0.45,0.55);
\coordinate (5) at (1,-0.1);
\coordinate (7) at (0.1-0.4,0.7);
\coordinate (9) at (2.6,0);
\coordinate (8) at (-.9,0.25);
\node[below,purple] at (7){\small$\overline{\gamma}_1$};
\node[left,orange]  at (5){\small$\overline{\gamma}_2$};
\draw [->,orange,thick](0.1-0.5-1,1) .. controls
(1.3-1.8,-0.1) and (2.6-1.8,-0.1) ..(2.6-1,1);
\draw [->,purple,thick](0.1-0.5-1,1) .. controls
(1.3-1.8,0.5) and (2.6-2.5,0.5) ..(2.6-2,1);
\draw (-1.5,1) -- (1.7,1);
\draw (-1.5,-1) -- (1.7,-1);
\draw[->] (0,1.3) -- (0.4,1.3);
\draw[<-] (0,-1.3) -- (0.4,-1.3);
\node at (0.2,-1.8){\small$(f)$};
 \end{tikzpicture}
\caption{ $\overline{\gamma}_1$  and $\overline{\gamma}_2$  are arcs in $\mathbb{U}$  such that $\pi(\overline{\gamma}_1)=\gamma_1$ and $\pi(\overline{\gamma}_2)=\gamma_2$ 
}
\label{epi+mon}
\end{figure}
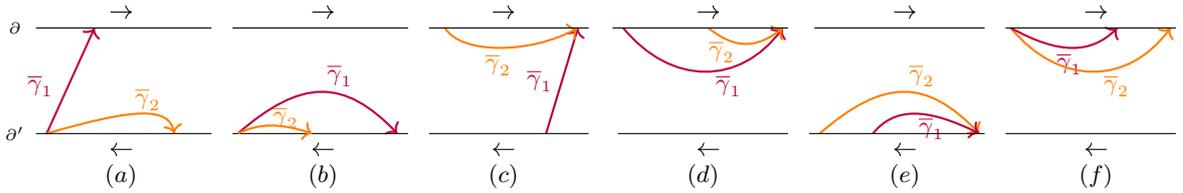
\end{proposition}
\begin{proof}
 The necessity of (1) and (2) follows directly from the bijection $\phi$ and Proposition~\ref{hom:mon,epi}. It remains to show  their  sufficiency. 
Since $Y$ is an indecomposable torsion sheaf supported at an exceptional point,  it is either  $Y=S_{\infty,i}^{(j)}$ or  $Y=S_{0,i'}^{(j)}$, where   $i\in \mathbb{Z}/p\mathbb{Z}$, $i'\in \mathbb{Z}/q\mathbb{Z}$ and $j\in \mathbb{Z}_{>0}$.

(1) If  $f:X\to Y$  is an epimorphism and  $Y=S_{\infty,i}^{(j)}$,  then  Proposition \ref{hom:mon,epi}(1) implies that  
$X$ is $S_{\infty,i}^{(j+m)}$ or $\mathcal{O}(i\vec{x}_1+l'\vec{x}_2+l\vec{c})$  with $m>0$,  $l\in \mathbb{Z}$ and $0\leq l'<q$. Using the bijection $\phi$, we get that $\gamma_1$ and $\gamma_2$ share the same   ending  point  on the  inner  boundary of  $A_{p,q}$ with $\gamma_2$ following $\gamma_1$ in the clockwise order at  this common point. By the same token, we can
  prove that if   $f  :X\to Y$ is an epimorphism and  $Y\in \mathcal{U}_{0}$, then $\gamma_1$ and $\gamma_2$ must   share the same starting point on the outer boundary.

(2) If  $f :X\to Y$  is a monomorphism and  $Y=S_{\infty,i}^{(j)}$, then it follows from  Proposition~\ref{hom:mon,epi}(2) that $X$ is $S_{\infty,i-m}^{(j-m)}$ with $0<m<j$. According to the bijection  $\phi$,  $\gamma_1$ and $\gamma_2$ share the same   starting    point  on the  inner  boundary of  $A_{p,q}$ with $\gamma_2$ following $\gamma_1$ in the clockwise order at  this common point. A similar argument to the one used for the case $Y = S_{\infty,i}^{(j)}$ shows that statement (2) also holds  when $Y \in \mathcal{U}_0$.
\end{proof}

\begin{remark}\label{rek:mon or epi} Let   $X$ and $Y$ be
 indecomposable   bundles or  torsion sheaves supported at exceptional points.
By integrating Theorem \ref{dimension and positive intersection}, Proposition~\ref{hom:mon,epi}, and Proposition~\ref{cro:inj,epi}, we establish the  result, originally presented in \cite{HR,M}:   If ${\rm Ext^{1}}(Y,X)=0$, then any nonzero morphism $X\to Y$ is an epimorphism or a monomorphism. 
In particular, for an indecomposable sheaf $X$ in $\cohx$, the condition ${\rm Ext^{1}}(X,X)=0$ implies that ${\rm End}(X)\cong\Bbbk$.
\end{remark}

We now proceed to  characterize the unique  epic-monic factorisation  of morphisms in   $\cohx$ into compositions of epimorphisms and monomorphisms through their corresponding oriented arcs in   $\mathbb{U}$. 
Note that  if $X,Y$ are line bundles  
 then  
 the nonzero morphism $f:X\to Y$   is a monomorphism. Hence, Proposition~\ref{hom condition}  and  Proposition~\ref{cro:inj,epi} imply that  we only  need to consider the remaining cases below.

\begin{theorem} \label{udecom}
Let $\gamma_1$ be a positive bridging arc  or peripheral arc,    and $\gamma_2$ be a peripheral arc  in $\mathbb{U}$.  Suppose   $I^{+}(\gamma_2,\gamma_1)=1$ and this  positive intersection $M$  meets the conditions in Proposition~\ref{hom condition}. 
Then there is a
 nonzero morphism $f_{M}:\phi(\pi(\gamma_1))\to \phi(\pi(\gamma_2))$ associated with the positive intersection  $M$ of $\gamma_2$ and $\gamma_1$. Furthermore,   $f_{M}$ can be uniquely decomposed,  
up to  isomorphism, into an epimorphism followed by a monomorphism through $\phi(\pi(\gamma))$, where 
$\gamma$ is a  peripheral arc  contained in $\gamma_2$ obtained by
smoothing  the crossing    at $M$ (see  Figure \ref{gamma}).

\begin{figure}[H]
\begin{tikzpicture}[scale=0.85]
\coordinate (1) at (1.2,0.7);
\node[blue] at (1){\small$\gamma_1$};
\draw[->,thick,blue] (0.7,0) -- (1.3,2);
\draw  [thick,->](-1,2) .. controls
(-0,1) and (1,1) .. (2,2);
\coordinate (2) at (1.5,1.5);
\node[right] at (2){\small$\gamma_2 $};
\draw  [->,red](-1,2) .. controls
(-0.2,1.3) and (0.4,1.3) .. (1.3,2);
\node[red]  at (0.8,1.85){\small$\gamma$};
\end{tikzpicture}
\hspace{1cm}
\begin{tikzpicture}[scale=0.85]
\coordinate (3) at (0,-.6);
\node  at (3){\small$\gamma_2$};
\draw [->,blue,thick](-0.4-1,-1) --(2-1.7-1,1);
\coordinate (1) at (-0.4,-1);
\draw [->,thick](0.1-3,-1) .. controls
(0.1-2,-0.1) and (0.1-1,-0.1) ..(0.1,-1);
\coordinate (4) at (0.4,0.2);
\draw [->,red](-1.4,-1) .. controls
(-0.9,-.5) and (-0.4,-.5) ..(0.1,-1);
\coordinate (4) at (0.4-1.1,0.2);
\node[blue] at (4){\small$\gamma_1$};
\node[red]  at (-0.35,-0.9){\small$\gamma$};
\end{tikzpicture}
\caption{Peripheral arc $\gamma$ contained in $\gamma_2$}
\label{gamma}
\end{figure}
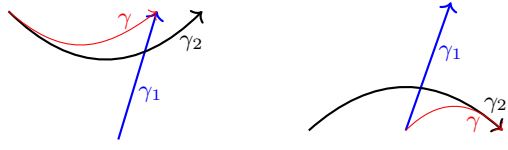
\end{theorem}

\begin{proof}
The proof proceeds via case analysis  according to the types of $\gamma_1$ and $\gamma_2$. 
For simplicity, we focus on the case where both endpoints of $\gamma_2$ are located on the upper boundary of $\mathbb{U}$, as the case with endpoints on the lower boundary follows similarly.

(1) $\gamma_1$ is a positive bridging arc. By assumption,   $\gamma_1=D^{\frac{i}{p}}_{\frac{j}{q} }$ and $\gamma_2=D^{\frac{k-l-1}{p}, \frac{k}{p}}$ for $i,j,k,l\in \mathbb{Z}$ with  
 $k-l-1<i<k$. From this,   we have
 \[
 \phi(\pi(\gamma_1))=\co(i\vec{x}_{1}-j\vec{x}_{2}), \quad   \phi(\pi(\gamma_2))=S_{\infty,k}^{(l)}, \quad   \phi(\pi(\gamma))=S_{\infty,i}^{(i-k+l)}.
 \]
 
We now show by induction on $m$ that
 there exists a short exact sequence:
\begin{equation}\label{exs:long}
    \begin{tikzcd}
   0\ar[r] & \co(-m\vec{x}_{1}) \ar[r,"{X}_{1}^{m}"] & \co\ar[r] & S_{\infty,0}^{(m)}\ar[r]  & 0,
\end{tikzcd} \end{equation}
for all $m\in\mathbb{Z}_{>0}$ 
in $\cohx$.
 For $m=1$, the existence follows immediately from  \eqref{esp}.
 Assume that the statement is valid for    $m$, 
we get
the following commutative diagram  with exact rows and columns.
\[\begin{tikzcd}[ampersand replacement=\&,cramped]
	\& 0 \& 0 \\
	\& {\mathcal{O}(-m\vec{x}_{1})} \& {\mathcal{O}(-m\vec{x}_{1})} \\
	0 \& {\mathcal{O}} \& {\mathcal{O}(\vec{x}_{1})} \& {S_{\infty,1}} \& 0 \\
	0 \& {S_{\infty,0}^{(m)}} \& {S_{\infty,1}^{(m+1)}} \& {S_{\infty,1}} \& 0 \\
	\& 0 \& 0
	\arrow[equal,from=2-2, to=2-3]
	\arrow["{X_{1}^{m}}", from=2-2, to=3-2]
	\arrow[  from=3-2, to=4-2]
	\arrow[from=4-2, to=5-2]
	\arrow[from=4-3, to=5-3]
	\arrow[from=4-2, to=4-3]
	\arrow[from=4-3, to=4-4]
	\arrow[from=4-4, to=4-5]
	\arrow[from=4-1, to=4-2]
	\arrow[from=3-1, to=3-2]
	\arrow[  from=3-2, to=3-3]
	\arrow[from=2-3, to=3-3]
	\arrow[from=3-3, to=3-4]
	\arrow[from=3-4, to=3-5]
	\arrow[equal, from=3-4, to=4-4]
	\arrow[from=3-3, to=4-3]
	\arrow[from=1-2, to=2-2]
	\arrow[from=1-3, to=2-3]
	\arrow[from=1-3, to=2-3]
\end{tikzcd}\]
From this setup, the proof of the required existence is complete. 
Consequently, it follows from \eqref{exs:long}   that  $f_{M}=g_{M}\circ h_{M}$, where $g_{M}$ is an epimorphism from $\phi(\pi(\gamma_1))$ to $\phi(\pi(\gamma))$ and  $h_{M}$ is a monomorphism from  $\phi(\pi(\gamma))$ to $\phi(\pi(\gamma_2))$.
  
(2)   $\gamma_1$ is a peripheral arc. 
In this case, we assume that  $\gamma_1=D^{\frac{i-j-1}{p}, \frac{i}{p}}$ and $\gamma_2=D^{\frac{k-l-1}{p}, \frac{k}{p}}$ 
for $i,j,k,l\in \mathbb{Z}$ with  
 $k-l-1<i<k$. Then  
 \[
 \phi(\pi(\gamma_1))=S_{\infty,i}^{(j)}, \quad   \phi(\pi(\gamma_2))=S_{\infty,k}^{(l)}, \quad   \phi(\pi(\gamma))=S_{\infty,i}^{(i-k+l)}
. \]
The required factorization follows immediately from Proposition~\ref{hom:mon,epi}, completing the proof.
\end{proof}
  
We now present another key result in this section. This theorem provides an efficient method for computing the kernels of monomorphisms and the cokernels of epimorphisms between indecomposable sheaves via oriented arcs in $\mathbb{U}$. Let $\mathcal{D}$ be the set defined by
\[
\mathcal{D} := \mathcal{C} \cup \left\{ [D^{\frac{i}{p}, \frac{j}{p}}], [D_{\frac{i}{q}, \frac{j}{q}}] \ \bigg| \ i,j\in \mathbb{Z},\ j-i=1\right\},
\]
where the   bijection $\phi: \mathcal{C} \to {\rm ind}(\cohx)$ canonically extends to a map $\Phi: \mathcal{D} \to {\rm ind}(\cohx) \cup \{0\}$ given by
\[
\Phi(\gamma) = \begin{cases} 
\phi(\gamma), & \gamma \in \mathcal{C}, \\
0, & \gamma \notin \mathcal{C}.
\end{cases}
\]

\begin{theorem}\label{thm:morphism}
Let $\gamma_1,\gamma_2$ be positive bridging arcs or peripheral arcs in $\mathbb{U}$.
\begin{itemize}
    \item [(1)]
If $I^{+}(\gamma_2,{_{s^{-1}}{\gamma_1}_{e^{-1}}})=1$ and the morphism $f_M : \phi(\pi(\gamma_1)) \to \phi(\pi(\gamma_2)) $ associated with this positive intersection $M$ is a monomorphism, then  there exists  a short exact sequence  given by
\begin{equation}\label{e:mon}
\eta_{M}: 0\to\phi(\pi(\gamma_1)) \to \phi(\pi(\gamma_2)) \to \Phi(\pi(\alpha_{1}))\oplus\Phi(\pi(\alpha_{2}))\to 0,
\end{equation}
where $\pi(\alpha_{1})$ and $\pi(\alpha_{2})$  are 
elements in $\mathcal{D}$ such that  $\alpha_{1}$ and $\alpha_{2}$ with endpoints at 
the starting points and ending points  of ${_{s^{-1}}{\gamma_1}_{e^{-1}}}$ and  $ \gamma_2$, respectively  (see Figure~\ref{fig:mon,phi}(a)).

\item [(2)] If $I^{+}( {_{s}{\gamma_2}_{e}},{\gamma_1})=1$ and the morphism $f_N : \phi(\pi(\gamma_1)) \to \phi(\pi(\gamma_2)) $ associated with this positive intersection $N$ is an epimorphism, then there exists a short exact sequence 
\begin{equation}\label{epi,short}
\eta_{N}:0\to \Phi(\pi(\beta_{1}))\oplus\Phi(\pi(\beta_{2})) \to \phi(\pi(\gamma_1)) \to \phi(\pi(\gamma_2)) \to 0,
\end{equation}
where  $\pi(\beta_{1})$ and $\pi(\beta_{2})$  are 
elements in $\mathcal{D}$ such that  $\beta_{1}$  and $\beta_{2}$    with endpoints at the starting points and ending points  of $\gamma_1$ and $_{s}{\gamma_2}_{e}$, respectively  (see  Figure~\ref{fig:mon,phi}(b)).
\end{itemize}
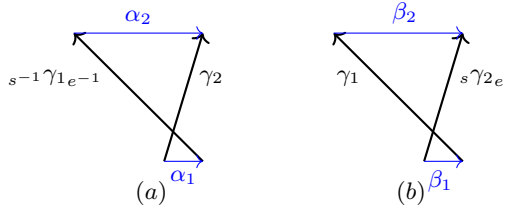
\begin{figure}[H]
\begin{tikzpicture}[scale=0.85]
\draw[->,thick] (1.3,0) -- (-0.7,2);
\coordinate (1) at (-0.1,1.3);
\node[left] at (1){\small${_{s^{-1}}{\gamma_1}_{e^{-1}}}$};
\draw[->,thick] (0.7,0) -- (1.3,2);
\coordinate (2) at (1.1,1.3);
\node[right] at (2){\small ${\gamma_2}$};
\draw[->,blue] (0.7,0) -- (1.3,0);
\draw[->,blue] (-0.7,2) -- (1.3,2);
\coordinate (3) at (0.3,2);
\coordinate (4) at (1,0);
\node[above,blue] at (3){\small $\alpha_{2}$};
\node[below,blue] at (4){\small $\alpha_{1}$};
\node at (0.5,-0.5) {\small $(a)$};
\end{tikzpicture}
\hspace{1cm}
\begin{tikzpicture}[scale=0.85]
\coordinate (1) at (1.6,1.3);
\node at (1){\small $_{s}{\gamma_2}_{e}$};
\draw[->,thick] (0.7,0) -- (1.3,2);
\draw[->,thick] (1.3,0) -- (-0.7,2);
\coordinate (1) at (-0.1,1.3);
\coordinate (2) at (-0.8,1.3);
\node[right] at (2){\small ${\gamma_1}$};
\draw[->,blue] (0.7,0) -- (1.3,0);
\draw[->,blue] (-0.7,2) -- (1.3,2);
\coordinate (3) at (0.4,2);
\coordinate (4) at (0.95,0);
\node[above,blue] at (3){\small $\beta_{2}$};
\node[below,blue] at (4){\small $\beta_{1}$};
\node at (0.5,-0.5) {\small $(b)$};
\end{tikzpicture}
\caption{The cutting of $({\gamma_2},{_{s^{-1}}{\gamma_1}_{e^{-1}}})$ and $(_{s}{\gamma_2}_{e},{\gamma_1})$ 
}
\label{fig:mon,phi}
\end{figure}
\end{theorem}
\begin{proof} 
(1)  
Since $f_M : \phi(\pi(\gamma_1)) \to \phi(\pi(\gamma_2))$ is a monomorphism,   it follows from Proposition~\ref{hom:mon,epi} that both $\phi(\pi(\gamma_1))$ and $\phi(\pi(\gamma_2))$ must lie either in ${\rm vect}\mbox{-}\mathbb{X}(p,q)$ or in $\mathcal{U}_{\lambda}$ for $\lambda \in \{\infty, 0\}$.

First, suppose that $\phi(\pi(\gamma_1)),\phi(\pi(\gamma_2))\in {\rm vect}\mbox{-}\mathbb{X}(p,q)$.
Then   $\gamma_1=D^{\frac{i}{p}}_{\frac{j}{q} }$ 
and $\gamma_2=D^{\frac{k}{p}}_{\frac{l}{q} }$ for some integers $i-1<k$ and $l<j+1$. This leads to $\alpha_1 = D^{\frac{i-1}{q}, \frac{k}{q}}$ and $\alpha_2 = D_{\frac{l}{q}, \frac{j+1}{q}}$. Based on \eqref{exs:long}, we get the following commutative diagram with   exact rows and columns in $\cohx$:
\[\begin{tikzcd}
	&&  0 & 0 \\
	0 & {\mathcal{O}(i\vec{x}_{1}-j\vec{x}_{2})} & {\mathcal{O}(k\vec{x}_{1}-j\vec{x}_{2})} & {S_{\infty,k}^{(k-i)}} & 0 \\
	0 & {\mathcal{O} (i\vec{x}_{1}-j\vec{x}_{2}) } & {\mathcal{O}(k\vec{x}_{1}-l\vec{x}_{2})} & {S_{\infty,k}^{(k-i)}\oplus S_{0,-l}^{(j-l)}} & 0 \\
	&& {S_{0,-l}^{(j-l)}} & {S_{0,-l}^{(j-l)}} \\
	&& 0 & 0
	\arrow[from=1-3, to=2-3]
	\arrow[from=1-4, to=2-4]
	\arrow[from=2-1, to=2-2]
	\arrow[  from=2-2, to=2-3]
	\arrow[equal,from=2-2, to=3-2]
	\arrow[from=2-3, to=2-4]
	\arrow[from=2-3, to=3-3]
	\arrow[from=2-4, to=2-5]
	\arrow[from=2-4, to=3-4]
	\arrow[from=3-1, to=3-2]
	\arrow[  from=3-2, to=3-3]
	\arrow[from=3-3, to=3-4]
	\arrow[from=3-3, to=4-3]
	\arrow[from=3-4, to=3-5]
	\arrow[from=3-4, to=4-4]
	\arrow[equal,from=4-3, to=4-4]
	\arrow[from=4-3, to=5-3]
	\arrow[from=4-4, to=5-4]
\end{tikzcd}\]
Thus the associated sheaves form  a short exact sequence in $\cohx$:
\[ 0 \to \mathcal{O}(i\vec{x}_1 - j\vec{x}_2) \to \mathcal{O}(k\vec{x}_1 - l\vec{x}_2) \to S_{\infty,k}^{(k-i)} \oplus S_{0,-l}^{(j-l)} \to 0. \]
Here, $\Phi(\pi(\alpha_1)) = S_{\infty,k}^{(k-i)}$ and $\Phi(\pi(\alpha_2)) = S_{0,-l}^{(j-l)}$, confirming that equation \eqref{e:mon} holds. 

Next, if $\phi(\pi(\gamma_1))$ and $\phi(\pi(\gamma_2))$ are in $\mathcal{U}_{\lambda}$ with $\lambda \in \{\infty, 0\}$, we focus on $\lambda = \infty$ (the case for $\lambda = 0$ is analogous).  In this case, $\gamma_1 = D^{\frac{i}{p}, \frac{j}{p}}$ and $\gamma_2 = D^{\frac{i}{p}, \frac{k}{p}}$ with $i < j < k$, leading to $\alpha_1 = D^{\frac{j-1}{p}, \frac{k}{p}}$ and $\alpha_2 = D^{\frac{i-1}{p}, \frac{i}{p}}$.  Note that \[
 \phi(\pi(\gamma_1))=S_{\infty,j}^{(j-i-1)}, \quad \phi(\pi(\gamma_2))=S_{\infty,k}^{(k-i-1)}.
 \]
Since there   exists a short exact sequence
\[
0\to S_{\infty,j}^{(j-i-1)}\to S_{\infty,k}^{(k-i-1)} \to S_{\infty,k}^{(k-j)}\to 0
\]
in $\cohx$, the proof  
is thus complete. Therefore, the statement (1) holds.

(2) 
For $f_N : \phi(\pi(\gamma_1)) \to \phi(\pi(\gamma_2))$ being an epimorphism, Proposition~\ref{hom:mon,epi} gives two cases: either $\phi(\pi(\gamma_1)) \in {\rm vect}\mbox{-}\mathbb{X}(p,q)$ and $\phi(\pi(\gamma_2)) \in \mathcal{U}_{\lambda}$, or both are in $\mathcal{U}_{\lambda}$ for $\lambda \in \{\infty, 0\}$. Given the similarity between the cases $\lambda = 0$ and $\lambda = \infty$, we   focus on    $\lambda = \infty$.

First, suppose $\phi(\pi(\gamma_1))\in {\rm vect}\mbox{-}\mathbb{X}(p,q) $. 
In this setting,  
$\gamma_1=D_{\frac{j}{q}}^{\frac{i}{p}}$ and ${\gamma_2}=D^{\frac{k}{p}, \frac{i}{p}}$ for $i,j,k\in \mathbb{Z}$ with  
 $k+1<i$. 
Thus, $\Phi(\pi(\beta_{1}))=\mathcal{O}((k+1)\vec{x}_{1}-j\vec{x}_{2})$  and $\Phi(\pi(\beta_{2}))=0$. 
 The sequence \eqref{epi,short}  is obtained by applying the  twisting
 functor $\mathcal{O}(i\vec{x}_1-j\vec{x}_2) $ to the exact sequence \eqref{exs:long} with parameter $ m = i-k -1$.

Next, suppose $\phi(\pi(\gamma_1))\in \mathcal{U_{\lambda}}$ for $\lambda\in\{\infty, 0\}$. By assumption,  $\gamma_1=D^{\frac{i}{p},\frac{j}{p} }$ 
and $\gamma_2=D^{\frac{k}{p},\frac{j}{p} }$ for some integers $i<k<j$.
Then 
we obtain  $\Phi(\pi(\beta_{1}))=0$ and $\Phi(\pi(\beta_{2}))=S_{\infty,k+1}^{(k-i)}$. The exact sequence:  
\[
0\to S_{\infty,k+1}^{(k-i)}\to  S_{\infty,j}^{(j-i-1)} \to  S_{\infty,j}^{(j-k-1)} \to 0,
\] completes the proof.
\end{proof}
 

\section{The geometric realization of   exceptional sequences in $\cohx$} \label{sec.4}

This section provides a geometric characterization of exceptional sequences in $\cohx$. We start by establishing a combinatorial criterion, defined  in terms of arcs in $A_{p,q}$, to determine when two exceptional objects form an exceptional pair. Based on this characterization,    we define a family of arc collections in $A_{p,q}$ that offer  a graphical interpretation of  exceptional sequences in $\cohx$. 
  As  applications, we classify complete exceptional sequences, and reveal the enlargement property of exceptional sequences  in $\cohx$ from
combinatorial perspectives.

\subsection{The position relation of arcs with respect to an exceptional pair} 
Recall that
an indecomposable sheaf $E$  in $\cohx$ is    exceptional   if and only if $\phi^{-1}(E)$ is an   arc in $A_{p,q}$ \cite{CRZ23}. In order to give  a geometric characterization of exceptional pairs in $\cohx$, the following   definition is needed.

\begin{definition}\label{def:exceptional intersection}
Let $\gamma_{1}$ and $\gamma_{2}$ be    arcs in  $A_{p,q}$. A positive intersection of  $\gamma_{1}$ and $\gamma_{2}$  is called an \emph{exceptional intersection} of $\gamma_{1}$ and $\gamma_{2}$, if  it satisfies   condition   $I^{+}( \gamma_1, {_{s^{-1}}{\gamma_2}_{e^{-1}}} )=0$ (or 
$I^{+}( _{s}{\gamma_1}_{e},  {\gamma_{2}}  )=0$).
\end{definition} 
\begin{remark}\label{exceptional intersection=1}
For any  two  arcs
 $\gamma_{1},\gamma_{2}$  in   $A_{p,q}$,  the number of exceptional intersections is at most one. Moreover, if there exists an   exceptional intersection  of $\gamma_{1}$ and $\gamma_{2}$, then $\gamma_{1}$ must be   a peripheral arc.
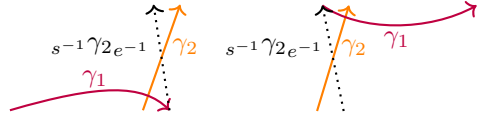
\begin{figure}[H]
\begin{tikzpicture}[scale=0.7]
\coordinate (3) at (-1.3,-.4);
\coordinate (5) at (-0.1,0.2);
\node[left]  at (5){${_{s^{-1}}{\gamma_2}_{e^{-1}}}$};
\node[purple] at (3){$\gamma_{1}$};
\draw [->,orange,thick](-0.4,-1) --(2-1.7,1);
\coordinate (1) at (-0.4,-1);
\draw [->,purple,thick](-0.4+0.5-3,-1) .. controls
(1.3+0.5-3,-.5) and (0.6-1,-.5) ..(0.1,-1);
\draw [->,thick,dotted](-0.4+0.5,-1) --(2-1.7-0.5,1);
\coordinate (4) at (0.4,0.2);
\node[orange] at (4){$\gamma_{2}$};
\end{tikzpicture}
\begin{tikzpicture}[scale=0.7]
\coordinate (2) at (3.5-.10,0.35);
\coordinate (6) at (2.7-.10,0.2);
\node[orange]  at (6){$\gamma_{2}$};
\node[purple] at (2){$\gamma_{1}$};
\draw [->,orange,thick](2-.10,-1) --(2.6-.10,1);
\draw [->,purple,thick](2.1-.10,1) .. controls
(1.3-1.4+1+1.9-.10,0.5) and (2.6-1.4+1+1.9-.10,0.5) ..(2.6-0.5+1+1.9-.10,1);
\draw [->,thick,dotted](2+0.5-.10,-1) --(2.6-0.5-.10,1);
\coordinate (4) at (2.3-.10,0.2);
\node[left] at (4){${_{s^{-1}}{\gamma_2}_{e^{-1}}}$};
\end{tikzpicture}
\caption{Exceptional intersection of $\gamma_{1}$ and $\gamma_{2}$ }
\label{EI}
\end{figure}
\end{remark}

\begin{theorem}\label{position}
 Let $E$ and $F$   be exceptional sheaves in $\cohx$ with $\phi^{-1}(E)=\alpha$  and $\phi^{-1}(F)=\beta$. Then   $(E,F)$  is an exceptional pair if and only if the arcs $\alpha$ and $\beta$ satisfy  one of the following:

\begin{itemize}
\item The intersection  of  $\alpha$ and $\beta$ in the interior  of  $A_{p,q}$ is an exceptional intersection of $\alpha$ and $\beta$;  
\item  
$\alpha$ and $\beta$ share either a starting point or an ending point, possibly both, and $\beta$ follows $\alpha$ in the clockwise order at the common  point;  
\item   
 $ \alpha $ and $ \beta $ do not intersect in the interior  of $A_{p,q}$,  and also do not share a starting point or an ending point.
 \end{itemize}
\end{theorem}

\begin{proof} 
Because the sufficiency is  evident, we only show the necessity.
Assume that $(E,F)$ is an exceptional pair. According to \cite[Lemma 8.1]{CRZ23}, both $\alpha$ and $\beta$ are  arcs. Furthermore,
it follows from Proposition~\ref{prop of coh}(1) and Theorem~\ref{dimension and positive intersection} that $I^{+}(\alpha, {_{s^{-1}}{\beta}_{e^{-1}}}  )=0=I^{+}(\beta,\alpha)$. 
Consequently, if there is an intersection  between  $\alpha$ and $\beta$ in the interior  of $A_{p,q}$, it must be an exceptional intersection of $\alpha$ and $\beta$. In addition,   Remark~\ref{exceptional intersection=1} guarantees that $I^{+}(\alpha,\beta)=1$.

Next, we consider the case, where there are no intersections between  $\alpha$ and $\beta$ in the interior  of $A_{p,q}$. By   Proposition~\ref{cro:inj,epi},  if  $\alpha$ and $\beta$ share the starting point or ending point, then the condition ${\rm Hom}(F,E)= 0$ indicates that  $\beta$ follows $\alpha$ in the clockwise order at this common point (see   Figure~\ref{fig:exceptional collection}). Besides, for all other positional relationships between arcs $\alpha$ and $\beta$, apart from the aforementioned cases,  ${\rm Hom}(F,E)= 0$ also holds. Now we have finished the proof.
\end{proof}

\begin{figure}[H]
\begin{tikzpicture}[scale=0.8]
\coordinate (5) at (0.9,-0.1);
\coordinate (9) at (2.6-0.3,0);
\node[right,purple] at (9){$\alpha$};
\node[left,orange]  at (5){$\beta$};
\draw [->,purple,thick](2.1+0.5,-1) -- (2.6-0.5,1);
\draw [->,orange,thick](-0.4+0.5,-1) --(2.6-0.5,1);
\draw[->] (2.225,0.5) to [out=200,in=-20] (1.6,0.5);

\fill (2.6-0.5,1) circle[radius=1.5pt];

\coordinate (4) at (1.8+4,0);
\coordinate (7) at (-0.5+4,0);
\node[purple] at (7){$\alpha$};
\node[left,orange]  at (4){$\beta$};
\draw[->,purple,thick] (-0.4+0.5+4,-1) -- (0.1-0.5+4,1);
\draw [->,orange,thick](-0.4+0.5+4,-1) --(2.6-0.5+4,1);
\draw[->] (4,-0.5) to [out=20,in=150] (4.6,-0.5);

\fill (-0.4+0.5+4,-1) circle[radius=1.5pt];
\coordinate (1) at (3.2,-1.5);
\end{tikzpicture}
\caption{ $\alpha$ and $\beta$ share the ending point (left) or starting point (right)
}
\label{fig:exceptional collection}
\end{figure}
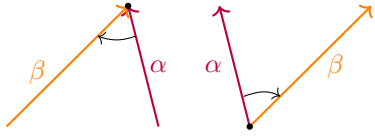

\begin{corollary}\label{iif}
Let   $(E,F)$ is an exceptional pair in $\cohx$ with  $\phi^{-1}(E)=\alpha$  and $\phi^{-1}(F)=\beta$.
\begin{itemize}
    \item [(1)]
 If  $I^{+}(_{s}\beta_{e},\alpha)\ne 0$, then  $I^{+}(_{s}\beta_{e},\alpha)=1$ or $I^{+}(_{s}\beta_{e},\alpha)=2$. Specifically,
 $I^{+}(_{s}\beta_{e},\alpha)=1$ if and only if $\alpha$ and $\beta$  only share the starting point or the ending point, with $\beta$ following  $\alpha$ in the clockwise order at  this common point. And   $I^{+}(_{s}\beta_{e},\alpha)=2$  if and only if 
 $\alpha$ and $\beta$   share both the starting   and the ending points, that is,  \[
        \alpha =  [D^{\frac{i}{p}}_{\frac{j}{q}}], \quad \beta = [D^{\frac{i+p}{p}}_{\frac{j}{q}}]\quad \text{for some } i,j\in \mathbb{Z}.
        \]
   \item[(2)]  If $I^{+}(\alpha,\beta)\ne 0$, then $I^{+}(\alpha,\beta)=1$. More precisely, $I^{+}(\alpha,\beta)=1$
 if and only if  
there exists an exceptional intersection of
 $\alpha$ and 
$\beta$.  
\end{itemize}
\end{corollary}

This corollary implies that   for an exceptional pair $(E,F)$ in $\cohx$,   either ${\rm Hom}(E,F)=  0$ or  ${\rm Ext}^{1}(E,F)= 0$,  which   has been  previously stated in \cite[Lemma 3.2.4]{M}.  
Additionally, we can derive an important property of exceptional pairs in  $\cohx$  as following. 
\begin{corollary}\label{cor:class}
Let $(E,F)$ be an exceptional pair in $\cohx$.   
\begin{enumerate}
\item If  ${\rm Hom}(E,F)\ne 0$, then  ${\rm Hom}(E,F)\cong\Bbbk$ or  ${\rm Hom}(E,F)\cong \Bbbk^2$. More precisely, ${\rm Hom}(E,F)\cong \Bbbk^2$ if and only if $E=\mathcal{O}(\vec{x})$ and $F=\mathcal{O}(\vec{x}+\vec{c})$ for some $\vec{x} \in\mathbb{L}$.
\item If  ${\rm Ext}^{1}(E,F)\ne 0$, then  ${\rm Ext}^{1}(E,F)\cong\Bbbk$.
\end{enumerate}
\end{corollary}
\begin{proof}
 In the following proof, we assume   that $\phi^{-1}(E)=\alpha$  and $\phi^{-1}(F)=\beta$.

(1)  Since ${\rm dim_{\Bbbk}}{\rm Hom}(E,F)=I^{+}(_{s}\beta_{e}, \alpha)$, it follows from Corollary~\ref{iif} that ${\rm dim_{\Bbbk}}{\rm Hom}(E,F)\leq 2$. 
According to the definition of $\phi$, we have that
$I^{+}(_{s}\beta_{e}, \alpha)=2$ if and only if  $E,F$ are line bundles and $F=E(\vec{c})$. Hence the statement (1) is proved.

(2) Also by Corollary~\ref{iif}, we have  $I^{+}( \alpha,\beta)= 0$ or $1$.  Consequently, statement (2) follows directly from Theorem \ref{dimension and positive intersection}.
\end{proof}

\subsection{The 
graphical characterization  of  exceptional sequences}
We now introduce a family of collections consisting of   arcs in $A_{p,q}$,  which  will used to provide a geometric  realization of  exceptional sequences in $\cohx$. To this end, the following definitions are necessary.

\begin{definition}
Let $\gamma_{1},\gamma_{2},\ldots, \gamma_{s}$ be distinct curves. If the collection $\{\gamma_{i}|i=1,2,\ldots,s\}$ forms a closed figure $\tilde{P}_s$, with vertices
at the intersections of these curves, 
then $\tilde{P}_s$ is called a \emph{$s$-curved polygon} and each  $\gamma_{i}$ is referred to as a \emph{curved edge} of $\tilde{P}_s$ for $i=1,2,\ldots,s$.
\end{definition}

\begin{definition}
    Let $\gamma:=\pi([x_{b}, y_{b}])$ be a  peripheral arc in $A_{p,q}$, where $x, \,y\in \{\frac{\mathbb{Z}}{p},\,\frac{\mathbb{Z}}{q}\}$ and $b\in \{\partial,\,\partial^{\prime}\}$. A marked point $z_b$ is said to be  
\emph{contained} in $\gamma$ if $x<z<y$. 

Given a collection $\Theta$ of arcs in $A_{p,q}$, a marked point $M$ of $A_{p,q}$ is called an  \emph{external point} of $\Theta$ if there exists no peripheral arc $\gamma \in \Theta$ containing $M$. In particular, if $\Theta$ consists entirely of positive bridging arcs, then every marked point of $A_{p,q}$  is an external point of $\Theta$.
\end{definition}

\begin{definition}\label{def:collection}
Let $\gamma_{1},\gamma_{2},\ldots, \gamma_{s}$ be distinct arcs in $A_{p,q}$. 
  The collection $\Theta=\{\gamma_{i}|i=1,2,\ldots,s\}$ is called    an \emph{exceptional collection} if it satisfies the following conditions: 
\begin{itemize}
\item [(E1)] The arcs from $\Theta$ do not intersect in the interior  of $A_{p,q}$, except at exceptional intersections.
 \item[(E2)]  
There is at least one external point     of the collection $\Theta$ on the inner boundary of $A_{p,q}$,
and similarly, there exists at least one external point on the outer boundary  of $A_{p,q}$.
 \item[(E3)]  For any curved polygon $\tilde{P}$ formed by  arcs from  $\Theta$,   $\tilde{P}$ is either  a $3$-curved polygon with two vertices in the interior  of $A_{p,q}$, or there exists a vertex of $\tilde{P}$ that is both the ending point of $\gamma_i$ and  the starting point of $\gamma_j$, where  $\gamma_i$ and $\gamma_j$ are the   curved edges of $\tilde{P}$. 
\end{itemize}

  An exceptional collection $\Theta$ is called 
\emph {maximal  exceptional collection},
 if for every  arc $\gamma$ in $A_{p,q}$, the collection $\Theta\cup \{\gamma\}$ is not an exceptional collection.
\end{definition} 
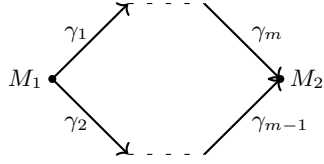
\begin{figure}[H]   
\begin{tikzpicture}
\fill(0,0) circle[radius=1.5pt];
\draw [->,thick](0,0) --(1,1);
\draw [->,thick](0,0) --(1,-1);
\draw [->,thick](2,1) --(3,0);
\draw [->,thick](2,-1) --(3,0);
\fill(3,0) circle[radius=1.5pt];
\node[left] at (0,0){\small$M_1$};
\node[right] at (3,0){\small$M_2$};
\node[left] at (0.6,0.6){\small$\gamma_1$};
\node[right] at (2.5,0.6){\small$\gamma_m$};
\node[left] at (0.6,-0.6){\small$\gamma_2$};
\node[right] at (2.5,-0.6){\small$\gamma_{m-1}$};
\draw[dash pattern=on 2pt off 5pt on 2pt off 5pt] (1,1)--(2,1);
\draw[dash pattern=on 2pt off 5pt on 2pt off 5pt] (1,-1)--(2,-1);
\end{tikzpicture}
\caption{Curved polygon with vertices in the interior of $A_{p,q}$, except for $M_1$ and $M_2$, which may coincide 
 ($M_2$ might be the ending point of  $\gamma_1$   or $\gamma_2$)}
\label{collection}
\end{figure}

\begin{remark}
Condition (E3) in Definition~\ref{def:collection} holds if and only if the arcs in $\Theta$  do not form any curved polygon  as depicted in Figure~\ref{collection}. In such a curved polygon, one vertex is the common starting point of $\gamma_1$ and $\gamma_2$, another vertex is the common ending point of $\gamma_{m-1}$ and $\gamma_m$, and the remaining vertices lie in the interior of $A_{p,q}$.
\end{remark}

\begin{example}
For the case of $p=2$ and $q=3$, let $\gamma_1, \gamma_2, \ldots,\gamma_6$ be six arcs in  $\mathbb{U}$, as illustrated in Figure~\ref{example}. Then the collection $\{\pi(\gamma_{i})|i=1,2,\ldots,5\}$ is a maximal exceptional  collection. However, the collection  $\{\pi({_{s^{-1}}{\gamma_1}_{e^{-1}}}),\pi(\gamma_5)\}$ and $\{\pi(\gamma_3),  \pi(\gamma_4) ,\pi(\gamma_6)\}$ are not   exceptional  collections, 
contradicting conditions (E2) and (E3) from Definition~\ref{def:collection}, respectively.
 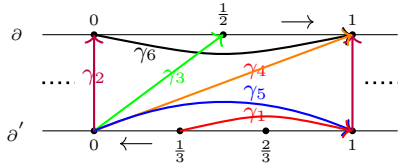
\begin{figure}[H] 
\begin{tikzpicture}[scale=0.85]
\coordinate (1) at (2.1+3,-0.5);
\coordinate (2) at (2.6-0.5+3,1);
\coordinate (5) at(1.1,-0.5);
\coordinate (0) at (3.1,1);
\fill(0) circle[radius=1.5pt];
\node[above] at (0){\tiny$\frac{1}{2}$};
\coordinate (8) at(2.43,-0.5);
\coordinate (3) at(3.76,-0.5);
\fill(8) circle[radius=1.5pt];
\node[below] at (8){\tiny$\frac{1}{3}$};
\fill(3) circle[radius=1.5pt];
\node[below] at (3){\tiny$\frac{2}{3}$};
\coordinate (10) at (3.6,0.35);
\coordinate (11) at (2.1,-0.65);
\coordinate (12) at (4.6,0.95);
\coordinate (14) at (3.6,0.9);
\coordinate (9) at (1.1,1);
\coordinate (6) at (3-0.5,1);
\coordinate (7) at (3.5-1,1);
\draw[->,thick] (1.1,1) .. controls  (3.1,0.6) .. (5.1,1);

\node[below,red] at (3.6,-0){\small$\gamma_{1}$};
\node[below] at (1.9 ,0.95){\small$\gamma_{6}$};
\node[below,blue] at (10){\small$\gamma_{5}$};

\node[red,below] at (3.6,0.7){\small$\gamma_{4}$};

\node[below] at (5){\tiny$0$};
\node[above] at (9){\tiny$0$};

\node[below ] at (1){\tiny$1$};
\node[above] at (2){\tiny$1$};
\node[below,green] at (2.35,0.6){\small$\gamma_{3}$};
\node[below, purple] at (1.1,0.6){\small$\gamma_{2}$};
\fill (9) circle[radius=1.5pt];
\fill (5) circle[radius=1.5pt];
\fill (1) circle[radius=1.5pt];
\fill (2) circle[radius=1.5pt];
\draw[line width=1pt,dotted](0.3,0.25)--(.8,0.25);
\draw[line width=1pt,dotted](9.7-4.4,0.25)--(10.2-4.4,0.25);
\draw (0.3,1) -- (5.8,1);
\draw (0.3,-0.5) -- (5.8,-0.5);
\draw[->] (4,1.2) -- (4.5,1.2);
\draw[<-] (1.5,-0.7) -- (2,-0.7);
\node (a) at (-0.1,1) {\tiny$\partial$};
\node (b) at (-0.1,-0.5) {\tiny{$\partial^{'}$}};
\draw [->,,purple,thick](2.1-1,-0.5) -- (2.6-0.5-1,1);
\draw [->,,purple,thick](2.1+3,-0.5) -- (2.6-0.5+3,1);
\draw [->,thick,orange](1.1,-0.5) -- (5.1,1);
\draw [->,thick,green](1.1,-0.5) -- (3.1,1);
\draw[->,blue,thick] (1.1,-0.5) .. controls
(2.5,0.1) and (3.7,0.1) .. (5.1,-0.5);
\draw[->,red,thick] (2.43,-0.5) .. controls  (3.76,-0.2) .. (5.1,-0.5);
 
\end{tikzpicture}
\caption{$\gamma_1,\gamma_2,\ldots,\gamma_6$  in  $\mathbb{U}$}
\label{example}
\end{figure}
\end{example}

\begin{definition}
     \label{prop:order}
Let $\Theta$ be an exceptional collection in $A_{p,q}$. A binary relation $``\preceq"$ on $\Theta$  is defined such that for   $\gamma_{i},\gamma_{j}\in \Theta$,    $\gamma_i\preceq \gamma_j$  
if   any of the following conditions is satisfied:
\begin{itemize}
\item $\gamma_i=\gamma_j$;
\item
there exists an exceptional intersection of  $\gamma_{i}$ and $\gamma_{j}$;
 \item  
$\gamma_{i}$,   $\gamma_{j}$ share either a starting point or an ending point,   possibly both,  and  $\gamma_j$ follows $\gamma_i$ in the clockwise order at the common point. 
\end{itemize}
\end{definition}

Taking the transitive closure of this relation,  still denoted  by $``\preceq"$, establishes a partial order on $\Theta$. Specifically,
this relation $``\preceq"$ is reflexive and antisymmetric by definition. Moreover, the transitivity of $``\preceq"$ is guaranteed by conditions (E2) and (E3) in Definition~\ref{def:collection}.

\begin{definition}[\cite{MR4578471}]\label{definition:ordered exceptional collection}
An ordered set of arcs $(\gamma_1,\gamma_2,\ldots,\gamma_m)$ in $A_{p,q}$ is called  an \emph{ordered exceptional collection} if it constitutes an exceptional collection and the ordering of the arcs is consistent with the partial order $``\preceq"$  previously defined, that is, $\gamma_{i}\preceq \gamma_{j}$ implies $i\leq j$.
\end{definition}

Now we are ready to determine the collection of   arcs which correspond to an  exceptional sequence.

\begin{theorem}\label{collection and sequence}
Let $\gamma_{1},\gamma_2,\ldots, \gamma_{s}$ be    arcs in  $A_{p,q}$. The following statements are equivalent.
\begin{enumerate}
\item The ordered set  $(\gamma_1,\gamma_2,\ldots,\gamma_s)$
 is an ordered exceptional collection in $A_{p,q}$;
\item The sequence $(\phi( \gamma_{1}),\phi( \gamma_{2}),\ldots,\phi(\gamma_{s}))$ forms an exceptional sequence in $\cohx$.
\end{enumerate}
\end{theorem}
\begin{proof}
(1) implies (2). Assume that $(\gamma_{1},\gamma_{2},\ldots, \gamma_{s})$ is  an ordered exceptional collection in  $A_{p,q}$. If  $i\leq j$, then either $\gamma_{i}\preceq \gamma_{j}$ or $\gamma_{i},\gamma_{j}$ are  not comparable by $``\preceq"$. According to the definition of  $``\preceq"$ on 
$\Theta=\{\gamma_{1},\gamma_{2},\ldots, \gamma_{s}\}$,  both situations imply that $(\phi (\gamma_{i}) ,\phi( \gamma_{j} ))$ is an exceptional pair. Consequently,  $(\phi( \gamma_{1}),\phi( \gamma_{2}),\ldots,\phi( \gamma_{s} ))$ is an exceptional sequence in $\cohx$.

(2) implies (1). Now suppose that $(\phi( \gamma_{1}),\phi( \gamma_{2}), \ldots,\phi( \gamma_{s} ))$ is an exceptional sequence.  
If there exist $\gamma_i$ and $\gamma_j$ in $\Theta$ such that the positive intersection of $\gamma_i$
and $\gamma_j$ is not an exceptional intersection, then
\[
    {\rm Ext}^{1}(\phi( \gamma_{i}) ,\phi( \gamma_{j}) )\ne 0,\
    {\rm Hom}(\phi( \gamma_{j}) ,\phi( \gamma_{i}) )\ne 0.
\] This is contrary to assumption.

We claim that  $\Theta$ is    an exceptional collection.  
In fact, if (E2)  is invalid,
without loss of generality, we may assume that the   peripheral arcs  $\gamma_{1},\gamma_{2},\ldots, \gamma_{m}$ with  ${\rm Int}^{+}(\gamma_m,\gamma_{1})\ne 0$ and ${\rm Int}^{+}(\gamma_i,\gamma_{i+1})\ne 0$  for all $1\leq i\leq m-1$, where $1< m\leq s$.  
Then we have 
\[ {\rm Ext}^{1}(\phi( \gamma_{m}) ,\phi( \gamma_{1})) )\ne 0,\
      {\rm Ext}^{1}(\phi( \gamma_{i}) ,\phi( \gamma_{i+1}) )\ne 0,
\]
for all $1\leq i\leq m-1$. 
This contradicts the fact that  $(\phi(\gamma_{1}),\phi(\gamma_{2}),\ldots,\phi(\gamma_{s}))$ is an exceptional sequence. On the other hand, if (E3)   fails to be satisfied, then we   may assume that the arcs $\gamma_{1},\gamma_{2},\ldots, \gamma_{n}$ form a   curved polygon, where $(\gamma_{n},\gamma_{1})$ and $(\gamma_{i},\gamma_{i+1})$ satisfy one of the conditions outlined in Definition~\ref{prop:order} for all $1\leq i\leq n-1$,   where $1< n\leq s$. 
Therefore,  we obtain:
\[
      {\rm Hom}(\phi( \gamma_{n}) ,\phi(\gamma_{1}))\ne 0,
\ {\rm or} \
      {\rm Ext}^{1}(\phi (\gamma_{n}) ,\phi( \gamma_{1}) )\ne 0.
\] 
This also leads to a contradiction.

Furthermore, if $\gamma_{i}\preceq \gamma_{j}$ with $\gamma_{i}\ne\gamma_{j}$, then  there exists  a subset $\{\gamma_{0}',\gamma_{1}',\ldots,\gamma_{m}'\}\subseteq\{\gamma_{1},\gamma_{2},\ldots, \gamma_{s}\}$ such that 
$$\gamma_{i}=\gamma_{0}'\preceq\gamma_{1}'\preceq \ldots \preceq \gamma_{m-1}'\preceq\gamma_{m}'=\gamma_{j}, $$ where the pair $(\gamma_{t-1}',\gamma_{t}')$ satisfies one of the conditions outlined in Definition~\ref{prop:order} for all $1\leq t \leq m$. 
Consequently,   the  sequence 
$$(\phi( \gamma_{i}) ,\phi( \gamma_{1}') ,\ldots,\phi( \gamma_{m-1}') ,\phi (\gamma_{j}) )$$ is a subsequence of  the original exceptional sequence. 
Thus, we conclude that $i\leq j$.  This completes the proof that  an exceptional sequence in  $\cohx$ leads to  an ordered exceptional collection.  
 \end{proof}

\subsection{The geometric realization of complete exceptional sequences}\label{complete exceptional} 
We now classify complete exceptional sequences in $\cohx$ based on the  following property of maximal exceptional collections.

\begin{proposition}\label{collection es} 
Let $\Theta$ be  a maximal  exceptional collection in  $A_{p,q}$.  Suppose   there are   $k$ (resp. $l$)   external points  of   $\Theta$ on the  inner (resp. outer)  boundary of   $A_{p,q}$ with $1\leq k\leq p$ and  $1\leq l\leq q$.
Then  $\Theta$ consists of $p-k$  (resp. $q-l$) peripheral arcs  with endpoints   on the inner  (resp. outer) boundary of    $A_{p,q}$,
and $k+l$ positive bridging arcs.  
Consequently, the  total number of arcs in $\Theta$  is  $p+q$. 
\end{proposition}
\begin{proof} 
Assume that the points $\pi((\frac{h}{p}, 1))$ and $\pi((\frac{h+d}{p}, 1))$ are external points of $\Theta$, with no other external points of $\Theta$ located between them, where $1 \leq d \leq p$.
For  convenience, we define the exceptional collection $\Theta (e, f)$ as follows:
\[
\Theta (e, f) = \{ [D^{\frac{u}{p}, \frac{v}{p}}] \in \Theta \mid e \leq u < v \leq f \}, 
\]
with $h \leq e < f \leq h + d$, and we denote by $|S|$ the cardinality of the set  $S$. These notations will be used  throughout the remainder of this proof.

  We claim that there are $d-1$  peripheral arcs in  $  \Theta (h,h+d)$, which    
  will be proved by induction  on $d$.
 For the base case where  $d=1$,   it is clear that there are no peripheral arcs in  $\Theta (h,h+d)$, thereby validating our claim. For  $d=2$,  the unique peripheral arc $[D^{\frac{h}{p}, \frac{h+2}{p}}]$ exists in $\Theta (h,h+d)$.   
Now, let us assume that   $d>2$.  We analyze two cases based on whether the arcs in $\Theta (h,h+d)$ intersect in the interior of $A_{p,q}$ or not.

If the  arcs  in $\Theta (h,h+d)$ do not intersect in the interior of $A_{p,q}$, then $\Theta (h,h+d)$ forms a  triangulation of marked surface with marked points $$\{\pi((\frac{h}{p},1)), \pi((\frac{h+1}{p},1)), \ldots,\pi((\frac{h+d}{p},1))\}$$ (refer to Figure~\ref{ep arc}).
Hence, there are $d-1$ arcs in $\Theta(h,h+d)$, by \cite[Proposition 2.10]{FST07}. 
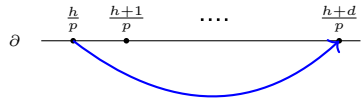
\begin{figure}[H]
\begin{tikzpicture}[scale=0.7]
\coordinate (7) at (0.6,1);
\coordinate (6) at (1.6,1);

\coordinate (0) at (2.6+3,1);
\node[above] at (7){\tiny$\frac{h}{p}$};
\node[above] at (6){\tiny$\frac{h+1}{p}$};
\node[above ] at (0){\tiny$\frac{h+d}{p}$};
\fill (0) circle[radius=1.5pt];
\node (a) at (-0.5,1) {\tiny{$\partial$}};
 
\fill (6) circle[radius=1.5pt];
\fill (7) circle[radius=1.5pt];
\draw (0,1) -- (6,1);
 
 
\draw[->,thick,blue] (0.1+0.5,1) .. controls
(0.6+1.8,-0.4) and (1.7+2.4,-0.4) .. (2.6+3,1);
\draw[line width=1pt,dotted](3,1.4)--(3+0.5,1.4);

\coordinate (1) at (0.1+0.5-0.6,-1);
\coordinate (2) at (2.6+3-0.6,-1);
 
\end{tikzpicture}
\caption{The marked surface with mark points $\{\pi((\frac{h}{p},1)),   \ldots,\pi((\frac{h+d}{p},1))\}$}
\label{ep arc}
\end{figure}
Conversely,  consider the arcs
  $[D^{\frac{a}{p}, \frac{b}{p}}]$
  and $[D^{\frac{b-1}{p}, \frac{c}{p}}]$  in $\Theta (h,h+d)$  such that $h\leq a<b<c\leq h+d$. 
 If $a=h$ and $c=h+d$, then the inductive hypothesis asserts that 
$$|\Theta (h, b)|=b-h-1, \ |\Theta (b-1, h+d)|=h+d-b.$$
Consequently,  the total number of arcs in $\Theta (h,h+d)$ remains  $d-1$.
If $a>h$ or $c<h+d$, then the inductive hypothesis yields that $|\Theta (a, c)|=c-a-1$.
Let $$\mathcal{N} :=\{N_i:=\pi((\frac{h_i}{p},1))|h\leq h_i\leq a,1\leq i\leq m\}, \ \mathcal{M} :=\{M_i:=\pi((\frac{h'_i}{p},1))|c\leq h'_i\leq h+d,1\leq i\leq n\},$$ denote the sets of external points of the collections $\Theta (h, a+1)$ and $\Theta (c-1, h+d)$, respectively.
 Due to the maximality property of $\Theta$, 
if  $\pi((\frac{u}{p},1))\in \mathcal{N} $  and
  $\pi((\frac{v}{p},1))\in \mathcal{M} $, then
 $h\leq u< a$  and  $ c < v \leq h + d $.
By the inductive hypothesis, we have:
\[
|\Theta(h, a+1)|+m=a+1-h,\ |\Theta(c-1, h+d)|+n=h+d+1-c.
\]
Note that the peripheral arcs $[D^{\frac{u}{p}, \frac{v}{p}}]\in  \Theta (h,h+d)$ with endpoints $\pi((\frac{u}{p}, 1))\in \mathcal{N} $ and $\pi ((\frac{v}{p}, 1))\in \mathcal{M} $ do not intersect in the interior of $A_{p,q}$. Consequently, these arcs correspond to a triangulation of the marked polygon illustrated in Figure~\ref{fig:theta1}.
 Moreover, the condition (E3) from Definition~\ref{def:collection}  implies that the peripheral arc ending with 
 $N_{m}\in \mathcal{N} $ and
 $M_1 \in \mathcal{M} $ is not in $\Theta (h,h+d)$ (see Figure~\ref{fig:theta1}).
Thus, the number of peripheral arcs $[D^{\frac{u}{p}, \frac{v}{p}}]\in\Theta (h,h+d)$   with $\pi((\frac{u}{p},1))\in \mathcal{N} $ and $\pi((\frac{v}{p},1))\in \mathcal{M} $ that do not intersect in the interior of $A_{p,q}$ is    $m+n$.
Therefore,  we can compute the total number of peripheral arcs in $\Theta (h,h+d)$ as follows: 
\[
|\Theta (a, c)|+|\Theta (h, a+1)|+|\Theta (c-1, h+d)|+m+n=d-1.
\]
 This establishes the claim.

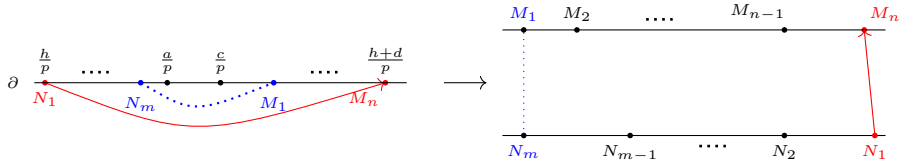
\begin{figure}[H]
    \begin{tikzpicture}[scale=0.7]
\fill[red] (-9+-0.8,1-1) circle[radius=1.5pt];
\node[above] at (-9+-0.8,1-1){\tiny$\frac{h}{p}$};
\node[below,red ] at (-9+-0.8,1-1){\tiny$N_{1}$};
\fill[red] (-9+2.6+3,1-1) circle[radius=1.5pt];
\node[above] at  (-9+2.6+3,1-1){\tiny$\frac{h+d}{p}$};
\node[below,red ] at (-9+2.2+3,1-1){\tiny$M_{n}$};
\draw[red,->] (-9+-0.8,1-1) ..  controls
(-9+1.5+0.5,0.4-1.5)  .. (-9+2.6+3,1-1);

 \node[above] at(-9+1+0.5,1-1){\tiny$\frac{a}{p}$};
 \node[above ] at (-9+2+0.5,1-1){\tiny$\frac{c}{p}$};
 \fill (-9+1+0.5,1-1) circle[radius=1.5pt];
\node (-9+a-1) at (-9+-1.4,1-1) {\tiny{$\partial$}};
 \fill (-9+2+0.5,1-1) circle[radius=1.5pt];
\draw (-9+-1,1-1) -- (-9+6,1-1);
\draw[line width=1pt,dotted](-9+3+1.2,1.2-1)--(-9+3+1.7,1.2-1);
\draw[line width=1pt,dotted](-9+-0.1,1.2-1)--(-9+0.4,1.2-1);
\draw[blue,thick,dotted] (-9+0.5+0.5,1-1) .. controls
(-9+1.5+0.5,0.4-1)  .. (-9+3+0.5,1-1);
\fill[blue] (-9+0.5+0.5,1-1) circle[radius=1.5pt];
\fill[blue] (-9+3+0.5,1-1) circle[radius=1.5pt];
\node[below,blue ] at (-9+0.5+0.5,0.9-1){\tiny$N_{m}$};
\node[below,blue ] at (-9+3+0.5,0.9-1) {\tiny$M_1$};

\draw[<-] (-1.5,0) -- (-2.3,0);

\coordinate (7) at (-0.8,1);
\coordinate (6) at (-0.8+1,1);
\coordinate (1) at (2.6+1.5,1);
\coordinate (0) at (2.6+3,1);
\node[above,red] at (3+3,1){\tiny$M_{n}$};
\node[above] at (2.6+1,1){\tiny$M_{n-1}$};
\node[above ] at (6){\tiny$M_{2}$};
\node[above,blue] at (7){\tiny$M_{1}$};
\fill[red] (0) circle[radius=1.5pt];
\fill (7) circle[radius=1.5pt];
\fill (1) circle[radius=1.5pt];
\fill (6) circle[radius=1.5pt];
\draw[line width=1pt,dotted](1.5,1.2)--(2,1.2);
\draw[line width=1pt,dotted](2.5,-1.2)--(3,-1.2);
\draw (-1.2,1) -- (6,1);
\draw (-1.2,-1) -- (6,-1);
\coordinate (2) at (-0.8,-1);
\coordinate (3) at (-0.8+2,-1);
\coordinate (4) at (2.6+1.5,-1);
\coordinate[red] (5) at (2.8+3,-1);
\node[below,red] at (5){\tiny$N_{1}$};
\node[below] at (4){\tiny$N_{2}$};
\node[below  ] at (3){\tiny$N_{m-1}$};
\node[below,blue] at (2){\tiny$N_{m}$};
\fill (2) circle[radius=1.5pt];
\fill (3) circle[radius=1.5pt];
\fill (4) circle[radius=1.5pt];
\fill[red] (5) circle[radius=1.5pt];
\draw[<-,red]   (2.6+3,1)--(2.8+3,-1);
\draw[dotted,blue]   (-0.8,-1)--(-0.8,1);
\end{tikzpicture}
    \caption{$N_i\in \mathcal{N} $ for all $1\leq i \leq n$ and $M_i\in \mathcal{M} $ for all $1\leq i \leq n$}
    \label{fig:theta1} 
\end{figure}

Based on the above claim and  the maximality property of $\Theta$, there are $p - k$ peripheral arcs with endpoints on the inner boundary of $A_{p,q}$.  
By a similar argument, we can also conclude that $\Theta$ contains $q - l$ peripheral arcs with endpoints on the outer boundary of $A_{p,q}$.

 Next, we will prove that there are $k+l$ positive bridging arcs  in $\Theta$.
Assume that the positive bridging arcs   in $\Theta$ are
\[ [D_{\frac{b_1}{q}}^{\frac{   a_1}{p}}],[D_{\frac{b_2}{q}}^{\frac{   a_2}{p}}], \dots,[D_{\frac{b_{r}}{q}}^{\frac{   a_{r}}{p}}] \]
where $  a_1 \leq a_2 \leq \cdots \leq a_{r} \leq a_1+p$ and $b_1 \leq b_2 \leq \cdots \leq b_r \leq b_1+q$.  Given the maximality of $\Theta$ and condition (E3) from Definition~\ref{def:collection}, we get that each set of marked points 
\[
\{\pi((\frac{x}{p},1)),\pi((\frac{y}{q},0))|a_i\leq x<a_{i+1},\  b_i<y\leq b_{i+1}\},
\] 
contains exactly one external point of $\Theta$.  
This holds for all $i = 1, 2, \ldots, r$  with $a_{r+1}=a_1+p$ and $b_{r+1}=b_1+q$ .  Thus  $r=k+l$,   
thereby completing the proof of this proposition.
\end{proof}
 
\begin{remark} Let $ \Theta$ be a maximal exceptional collection in $A_{p,q}$.
   Suppose the points $\pi((\frac{h}{p}, 1))$ and $\pi((\frac{h+d}{p}, 1))$ are external points of  $\Theta$. Then the set $ \{\phi([D^{\frac{a}{p}, \frac{b}{p}}]) | [D^{\frac{a}{p}, \frac{b}{p}}]\in\Theta, h\leq a < b \leq h+d\}$   forms an exceptional  set of type $A_{d-1}$. The number of such  sets is given by the generalized Catalan number
    $$\frac{1}{3d-2}\binom{3d-2}{d-1}=\frac{(3d-3)!}{({d}-1)!(2d-1)!},$$ as established in \cite[Theorem 5.4]{IR}.
  
 \end{remark}

In the light of   Proposition~\ref{collection es}, we present a
method for constructing
   maximal exceptional collections. For this, we give the following definitions.

\begin{definition}\label{def:extended-boundary}
     Let $\Theta$ be an  exceptional collection in $A_{p,q}$. The \emph{extended boundary point sets  of $\Theta$} are defined as follows:
 \begin{itemize}
    \item The extended boundary point set on the inner boundary: $$\overline{S}_1 :=\{ \pi ( (\frac{a}{p},1  ) \, |\, \text{every peripheral arc  in $ \Theta $ containing } \pi ( (\frac{a}{p},1 ) ) \text{ starts at } \pi ( (\frac{a-1}{p},1 ) ) \}.$$ 
    \item The extended boundary point set on the inner boundary:$$\overline{S}_2 := \{\pi ( (\frac{b}{q},0 )) \, |\, \text{every peripheral arc in $ \Theta $ containing } \pi ( (\frac{b}{q},0 ) ) \text{ ends at } \pi ( (\frac{b+1}{q},0 ) ) \}.$$
 \end{itemize} 
\end{definition}
\begin{definition}\label{adjustment map}
  Let $\Theta$ be an  exceptional collection in $A_{p,q}$ with extended boundary point sets $\overline{S}_1, \overline{S}_2$ as in Definition~\ref{def:extended-boundary}.  
   Denote  by $S_1$  (resp. $S_2$ )   the set of external points  of   $\Theta$ on the  inner (resp. outer)  boundary of   $A_{p,q}$.  
    The \emph{basic bridging arcs set  of $\Theta$}  is 
 $$\Theta_B := \{\gamma \in \Theta \mid \gamma \text{ is positive bridging with endpoints in } S_1 \cup S_2\},$$
  and   
    \emph{extended bridging arc set  of $\Theta$} is
$$\overline{\Theta}_B := \{\gamma \in \Theta \mid \gamma \text{ is positive bridging with endpoints in } S_1 \cup S_2 \cup \overline{S}_1 \cup \overline{S}_2\}.$$ 
 The \emph{endpoint adjustment map}
 $\Psi_\Theta: \overline{\Theta}_B \to \Theta_B$   is defined by
\[
\Psi_\Theta( [D_{\frac{b}{q}}^{\frac{a}{p}}]) = [D_{\frac{\overline{b}}{q}}^{\frac{\overline{a}}{p}}],
\]
where $\overline{a}=\min \{ c \mid\pi((\frac{c}{p}, 1))\in S_1,   a \leq c \}$ and  $\overline{b}=\max \{ d \mid\pi((\frac{d}{q}, 0))\in S_2,   d \leq b \}$.
\end{definition}
Based on Definition~\ref{def:extended-boundary} and Definition~\ref{adjustment map},
a maximal exceptional collection   of    $A_{p,q}$  is constructed as follows:  Fix integers $k,l$  such that $1\leq k\leq p$ and  $1\leq l\leq q$. 
\begin{enumerate}[leftmargin=*,label=Step \arabic*.,ref=Step~\arabic*]
 \item  Select $k$ marked points  on the inner  boundary to form the set $S_1$,  and  $l$ marked points   on the outer  boundary  to form the set $S_2$.
    \item \label{step:init} 
Include $p - k$ peripheral arcs on the inner boundary and $q - l$ peripheral arcs on the outer boundary.  These arcs, which form the set $\Theta$,  must satisfy conditions (E1) and (E3) from Definition~\ref{def:collection}, and    not contain any points from
$S_1\cup S_2$.    
    \item \label{step:quotient}     
  Consider $\overline{A}_{k,l}$ as the marked annulus with marked points $S_1 \cup S_2$. Select $k+l$  positive bridging arcs $\{\gamma_1,\gamma_2,\dots,  \gamma_{k+l} \} \subseteq \overline{\Theta}_B$ such that no two arcs intersect in the interior of  $A_{p,q}$ and  their adjusted images $\{\Psi_\Theta(\gamma_1),\Psi_\Theta(\gamma_2),\dots,  \Psi_\Theta(\gamma_{k+l}) \}$ form a triangulation of  $\overline{A}_{k,l}$. 
\end{enumerate}

The  resulting  collection  can be ordered into an ordered exceptional collection $(\overline{\gamma}_1, \ldots,\overline{\gamma}_{p+q})$. 
Moreover, by Theorem~\ref{collection and sequence}, the sequence $ ( \phi(\overline{\gamma}_1), \dots, \phi(\overline{\gamma}_{p+q}) )$
forms a complete exceptional sequence.

The proof of Proposition~\ref{collection es} also implies
that any exceptional collection can be
augmented with a finite number of arcs to construct a maximal exceptional collection.
 Combining this result with Theorem~\ref{collection and sequence} leads to the following conclusion, which is also stated in \cite[Lemma~2.6]{MR1318999}. 
 \begin{corollary}\label{complete}
   Any exceptional sequence in  $\cohx$  can be  enlarged into a complete exceptional sequence. 
\end{corollary}

\section{The  mutation of   exceptional sequences in $\cohx$} \label{sec.5}
In this section, we 
use the geometric model to describe the action of braid group 
on  exceptional sequences in $\cohx$ and
give a new approach to proving the transitivity of this action.
We  begin by introducing  the action of braid group on  ordered  exceptional collections in $A_{p,q}$. This is based on Corollary~\ref{iif} and  the following convention:
If $\gamma_{1}$ and $\gamma_{2}$  only share the starting point or the ending point at $M$, we denote by  $\delta$ the arc in $\mathcal{C}$ that   terminates   at the remaining endpoints  $M'$ and $M''$ of $\gamma_{1}$ and $\gamma_{2}$ (see  Figure~\ref{fig:mutation1}(a)(b)).

\begin{definition}\label{mutation of arcs}
(1) Let $(\gamma_{1}, \gamma_{2})$ be an ordered  exceptional collection in  $A_{p,q}$. We  define the left  mutation of $\gamma_{2}$ at $\gamma_{1}$ (resp. right mutation of $\gamma_{1}$ at $\gamma_{2}$) denoted by $L_{\gamma_{1}}\gamma_{2}$ (resp. $R_{\gamma_{2}}\gamma_{1}$) as follows.
\begin{itemize}
\item
If $\gamma_{1}$ and $\gamma_{2}$ only share the ending point at $M$, then both $L_{\gamma_{1}}\gamma_{2} $ and $R_{\gamma_{2}}\gamma_{1}$ are defined as $\delta[e]$, where $\delta[e]$  is  obtained by moving the endpoint of $\delta$ to the immediately adjacent marked point on the right along the same boundary;
\item If $\gamma_{1}$ and $\gamma_{2}$ only share the starting point at $M$, then  both $L_{\gamma_{1}}\gamma_{2} $ and $R_{\gamma_{2}}\gamma_{1}$ are defined as $[s]\delta$, where $[s]\delta$  is  obtained by   moving the starting point of $\delta$  to the immediately adjacent marked point on the left along the same boundary;
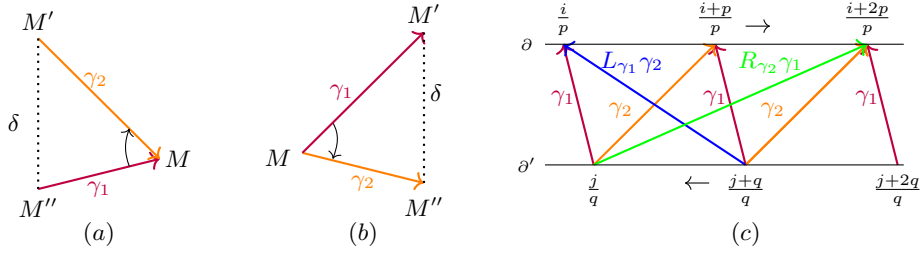
\begin{figure}[h]
\begin{tikzpicture}[scale=0.8]
 \begin{scope}[rotate=-90]
\coordinate (5) at (0.8,0.3);
\coordinate (9) at (2.6,0);
\node[purple] at (9){\small$\gamma_{1}$};
\node[left,orange]  at (5){\small$\gamma_{2}$};
\draw [->,purple,thick](2.1+0.5,-1) -- (2.6-0.5,1);
\draw [->,orange,thick](-0.4+0.5,-1) --(2.6-0.5,1);
\draw[thick, dotted](2.1+0.5,-1) --(-0.4+0.5,-1) ;
\node[above]  at (1.8,-1.4){\small$\delta$};
\node[ below]  at (-1+0.5,-1){\small$M'$};
\node[ below]  at(2.1+0.5,-1){\small$M''$};
\node at (2.6-0.5,1.3) {\small$M$};
\draw[->] (2.225,0.5) to [out=200,in=-20] (1.6,0.5);
\coordinate (1) at (1.4,-2);
 \end{scope}
\node[ below]  at (0,-3){\small$(a)$};
\end{tikzpicture}
\hspace{0.5cm}
\begin{tikzpicture}[scale=0.8]
\begin{scope}[rotate=270, xscale=-1]
\coordinate (4) at (1+4+0.1,0);
\coordinate (7) at (4-0.4,0);
\node[orange] at (7){\small$\gamma_{2}$};
\node[left,purple]  at (4){\small$\gamma_{1}$};
\draw[<-] (4,-0.5) to [out=20,in=150] (4.6,-0.5);
\draw[->,orange,thick] (-0.4+0.5+4,-1) -- (0.1-0.5+4,1);
\draw [->,purple,thick](-0.4+0.5+4,-1) --(2.6-0.5+4,1);
\draw[thick, dotted](0.1-0.5+4,1)--(2.6-0.5+4,1);
\node[above]  at (4.8,1.2){\small$\delta$};

\coordinate (1) at (3.2,-1.5);
\node[ above]  at (2.6-0.5+4,1){\small$M'$};
\node[ below]  at (-0.4+0.7+4,-1.4){\small$M$};
\node[above] at (0.1-0.7+4-0.5,1) {\small$M''$};
\coordinate (1) at (4.4,-2);
\end{scope} 
\node[ below]  at (0,4.2-1.1){\small$(b)$};
\end{tikzpicture}
\hspace{0.5cm}
\begin{tikzpicture}[scale=0.8]
\coordinate (3) at (1.4,0.7);
\coordinate (4) at (0.7+3,0.7);
\coordinate (2) at (0.9,-0.1);
\coordinate (5) at (0.9+2.5,-0.1);
\coordinate (7) at (0.1-0.5,0.1);
\coordinate (9) at (2.6-0.5,0.1);
\coordinate (1) at (2.6-0.5+2.5,0.1);
\node[purple] at (7){\small$\gamma_{1}$};
\node[purple] at (9){\small$\gamma_{1}$};
\node[left,orange]  at (5){\small$\gamma_{2}$};
\node[purple] at (1){\small$\gamma_{1}$};
\node[left,orange]  at (2){\small$\gamma_{2}$};
\node[left,blue]  at (3){\small$L_{\gamma_{1}}\gamma_{2}$};
\node[left,green]  at (4){\small$R_{\gamma_{2}}\gamma_{1}$};
\draw[->,purple,thick] (-0.4+0.5,-1) -- (0.1-1+0.5,1);
\draw [->,purple,thick](2.1+0.5,-1) -- (2.6-0.5,1);
\draw [->,orange,thick](-0.4+0.5,-1) --(2.6-0.5,1);
\draw [->,purple,thick](2.1+0.5+2.5,-1) -- (2.6-0.5+2.5,1);
\draw [->,orange,thick](-0.4+0.5+2.5,-1) --(2.6-0.5+2.5,1);
\draw[->,orange,thick] (2.1+0.5,-1)-- (2.6-0.5+2.5,1);
\draw[->,blue,thick] (2.1+0.5,-1)-- (0.1-1+0.5,1);
\draw[->,green,thick] (-0.4+0.5,-1)-- (2.6-0.5+2.5,1);
\node[ below]  at (-0.4+0.5,-1){\small$\frac{j}{q}$};
\node[above] at (0.1-1+0.5,1){\small$\frac{i}{p}$};
\node[ below]  at (-0.4+0.5+2.5,-1){\small$\frac{j+q}{q}$};
\node[above] at (0.1-1+0.5+2.5,1){\small$\frac{i+p}{p}$};
\node[ below]  at (-0.4+0.5+2.5+2.5,-1){\small$\frac{j+2q}{q}$};
\node[above] at (0.1-1+0.5+2.5+2.5,1){\small$\frac{i+2p}{p}$};
\draw (-0.7,1) -- (5.2,1);
\draw (-0.7,-1) -- (5.2,-1);
\node()at(-1,1){\tiny{${\partial}$}};
\node()at(-1,-1){\tiny{${\partial'}$}};
\draw[->] (2.6,1.3) -- (3,1.3);
\draw[<-] (1.6,-1.3) -- (2,-1.3);
\coordinate (6) at (2.2,-2);
\node[ below]  at (2.6,-1.8){\small$(c)$};
\end{tikzpicture}
\caption{ $\gamma_{1}$ and $ \gamma_{2}$ meet at their endpoints}
\label{fig:mutation1}
\end{figure}
\item If $\gamma_{1}=[D^{\frac{i}{p}}_{\frac{j}{q}}] $ and $\gamma_{2}=[D^{\frac{i+p}{p}}_{\frac{j}{q}}] $  
 for  $ i,j\in \mathbb{Z}$, 
then we define     $L_{\gamma_{1}}\gamma_{2}=[D_{\frac{j+q}{q}}^{\frac{i}{p}}]$  and $R_{\gamma_{2}}\gamma_{1}=[D_{\frac{j}{q}}^{\frac{i+2p}{p}}]$ (see Figure~\ref{fig:mutation1}(c));
\item If there exists an exceptional intersection of  $\gamma_{1}$ and $\gamma_{2}$, then
$L_{\gamma_{1}}\gamma_{2}=R_{\gamma_{2}}\gamma_{1}$ defined as the result of smoothing the crossing between  $\gamma_{1}$ and $\gamma_{2}$ at the intersection (see  Figure~\ref{fig:mutation3});
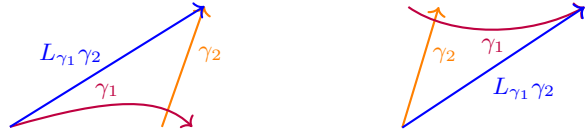
\begin{figure}[h]
\begin{tikzpicture}[scale=0.8]
\coordinate (3) at (-1.3,-.4);
\node[purple] at (3){\small$\gamma_{1}$};
\draw [->,orange,thick](-0.4,-1) --(2-1.7,1);
\coordinate (1) at (-0.4,-1);
\draw [->,purple,thick](-0.4+0.5-3,-1) .. controls
(1.3+0.5-3,-.5) and (0.6-1,-.5) ..(0.1,-1);
\coordinate (4) at (0.4,0.2);
\node[orange] at (4){\small$\gamma_{2}$};
\node[blue]  at (-1.9,0.2){\small$L_{\gamma_{1}}\gamma_{2}$};
\draw [->,blue,thick](-0.4+0.5-3,-1)--(2-1.7,1);
\end{tikzpicture}
\hspace{2cm}
\begin{tikzpicture}[scale=0.8]
\coordinate (2) at (3.5-.10,0.35);
\coordinate (6) at (2.7-.10,0.2);
\node[orange]  at (6){\small$\gamma_{2}$};
\node[purple] at (2){\small$\gamma_{1}$};

\draw [->,orange,thick](2-.10,-1) --(2.6-.10,1);
\draw [->,purple,thick](2.1-.10,1) .. controls
(1.3-1.4+1+1.9-.10,0.5) and (2.6-1.4+1+1.9-.10,0.5) ..(2.6-0.5+1+1.9-.10,1);
\draw [->,blue,thick](2-.10,-1)--(2.6-0.5+1+1.9-.10,1);
\node[blue]  at (3.5+.40,-0.35){\small$L_{\gamma_{1}}\gamma_{2}$};
\coordinate (7) at (1,-1.5);
\end{tikzpicture}
\caption{There exists an exceptional intersection of  $\gamma_{1}$ and $\gamma_{2}$}
\label{fig:mutation3}
\end{figure}
\item If
$\gamma_{1}$, $\gamma_{2}$ 
  do not intersect in the interior  of $A_{p,q}$,  and also do not share a starting point or an ending point, 
then we define
 $L_{\gamma_{1}}\gamma_{2}=\gamma_{2}$ and $R_{\gamma_{2}}\gamma_{1}=\gamma_{1}$.
\end{itemize}

(2) Let $(\gamma_{1}, \ldots, \gamma_{r})$ be an  ordered exceptional collection in $A_{p,q}$. For any $1\leq s< r$, we define
$${\rm \sigma }_s \cdot (\gamma_1, \ldots,  \gamma_{s},  \gamma_{s+1},\ldots,  \gamma_{r})=
(\gamma_1, \ldots, \gamma_{s-1}, L_{\gamma_s} \gamma_{s+1},\gamma_s,  \gamma_{s+2} ,\ldots,  \gamma_{r}),$$
 $${\rm \sigma }_s^{-1} \cdot (\gamma_1, \ldots,  \gamma_{s},  \gamma_{s+1},\ldots,  \gamma_{r})=
(\gamma_1, \ldots, \gamma_{s-1},\gamma_{s+1},  R_{\gamma_{s+1}} \gamma_{s},  \gamma_{s+2} ,\ldots,  \gamma_{r}).$$
\end{definition}
 
Now we give a theorem which shows that the left (resp. right) mutation of an ordered  exceptional collection in $A_{p,q}$ corresponds  the left (resp. right) mutation of an exceptional pair in $\cohx$. This theorem enables us to compute  the left or right mutation of an exceptional pair   by applying Definition~\ref{mutation of arcs} to the ordered exceptional collection associated with it.

\begin{theorem}\label{mutation correspond} Let   $(E,F)$ is an exceptional pair in $\cohx$ with  $\phi^{-1}(E)=\alpha$  and $\phi^{-1}(F)=\beta$.  
Then the following equations hold
\[\phi( L_{\alpha}\beta)=L_{E}F, \ \phi( R_{\beta}\alpha )=R_FE.\]
\end{theorem}
\begin{proof}
According to Corollary~\ref{cor:class}, we will divide this proof into the following four cases. 
\begin{enumerate}
\item  If ${\rm Hom}(E,F)\cong \Bbbk$, then  the nonzero morphism $f:E\to F$ is a monomorphism or an epimorphism by Remark~\ref{rek:mon or epi}. In this case,  $L_{E}F=R_{F}E$ and  we will  
suppose that $\phi^{-1}(L_{E}F)= {\gamma}$.
\begin{itemize}
 \item
If $f:E\to F$ is a monomorphism, then  the arcs $\alpha$, $\beta$ and ${\gamma}$ are illustrated  in Figure~\ref{fig:mutation mon}, following from  Proposition~\ref{hom:mon,epi}(1) and Theorem~\ref{thm:morphism}(1).
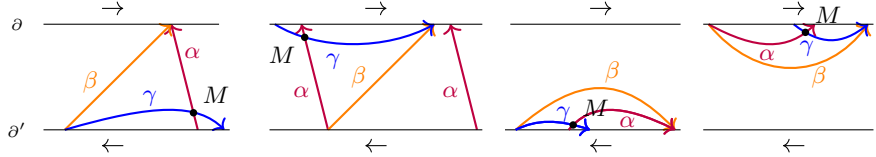
\begin{figure}[H]
\begin{tikzpicture}[scale=0.7]
\coordinate (3) at (1.7,-0.4);
\coordinate (5) at (0.9,-0.1);
\coordinate (7) at (0.1-0.3,0.35);
\coordinate (8) at (3,-0.25);
\coordinate (9) at (2.5,0.45);
\node[purple] at (9){\small$\alpha$};
\node[left,orange]  at (5){\small$\beta$};
\node[blue] at (3){\small$ {\gamma}$};
\draw [->,purple,thick](2.1+0.5,-1) -- (2.6-0.5,1);
\draw [->,orange,thick](-0.4+0.5,-1) --(2.6-0.5,1);
\draw [->,blue,thick](-0.4+0.5,-1) .. controls
(1.3+0.5,-.5) and (2.6,-.5) ..(2.1+1,-1);
\draw (-0.3,1) -- (3,1);
\draw (-0.3,-1) -- (3.2,-1);
\node()at(-0.8,1){\tiny{${\partial}$}};
\node()at(-0.8,-1){\tiny{${\partial'}$}};
\draw[->] (0.8,1.3) -- (1.2,1.3);
\draw[<-] (0.8,-1.3) -- (1.2,-1.3);

\path[name path=path1] (2.6, -1) -- (2.1, 1); 
    \path[name path=path2] (-0.4 + 0.5, -1) .. controls (1.3 + 0.5, -.5) and (2.6, -.5) .. (2.1 + 1, -1);  
   
    \path[name intersections={of=path1 and path2,by=intersection1}];   
    
    \fill[black,radius=2pt] (intersection1) circle;  
    \node[above right] at (intersection1) {\small$M$};  
\end{tikzpicture}
\hspace{0.05cm}
\begin{tikzpicture}[scale=0.7]
\coordinate (2) at (0.2,0.3);
\coordinate (5) at (1,0);
\coordinate (7) at (0.1-0.5,0);
\coordinate (9) at (2.4,0);
\coordinate (8) at (-.9,0.25);
\node[below,purple] at (7){\small$\alpha$};
\node[below,purple] at (9){\small$\alpha$};
\node[left,orange]  at (5){\small$\beta$};
\node[blue] at (2){\small$ {\gamma}$};
\draw[->,purple,thick] (-0.4+0.5,-1) -- (0.1-0.5,1);
\draw [->,purple,thick](2.1+0.8,-1) -- (2.4,1);
\draw [->,orange,thick](-0.4+0.5,-1) --(2.6-0.5,1);
\draw [->,blue,thick](0.1-1,1) .. controls
(1.3-1.4,0.5) and (2.6-1.4,0.5) ..(2.6-0.5,1);
\draw (-1,1) -- (3,1);
\draw (-1,-1) -- (3,-1);
\draw[->] (0.8,1.3) -- (1.2,1.3);
\draw[<-] (0.8,-1.3) -- (1.2,-1.3);

    \path[name path=path1] (-0.4+0.5,-1) -- (0.1-0.5,1);  
    \path[name path=path2] (0.1-1,1) .. controls
(1.3-1.4,0.5) and (2.6-1.4,0.5) ..(2.6-0.5,1);

    \path[name intersections={of=path1 and path2,by=intersection1}];

    \fill[black,radius=2pt] (intersection1) circle;  
    \node[below left] at (intersection1) {\small$M$};  
\end{tikzpicture}
\hspace{0.05cm}
\begin{tikzpicture}[scale=0.7]
\coordinate (3) at (1.7,-0.85);
\coordinate (5) at (1.1,0.1);
\coordinate (6) at (4.2-0.5,-0.1);
\coordinate (9) at (0.5,-0.7);
\node[right,orange]  at (5){\small$\beta$};
\node[purple] at (3){\small$\alpha$};
\node[blue] at (9){\small$ {\gamma}$};
\draw [->,orange,thick](-0.4,-1) .. controls
(1.3+0.5-1,0.05) and (2.6-1,0.05) ..(2.1+0.5,-1);
\draw [->,purple,thick](-0.4+1,-1) .. controls
(1,-.5) and (2.6-1,-.5) ..(2.1+0.5,-1);
\draw (-.5,1) -- (2.7,1);
\draw (-.6,-1) -- (2.7,-1);
\draw[->] (1,1.3) -- (1.4,1.3);
\draw[<-] (1,-1.3) -- (1.4,-1.3);
\draw [->,blue,thick](-0.4,-1) .. controls
(0,-0.8) and (0.4,-0.8) ..(1,-1);

    \path[name path=path1] (-0.4+1,-1) .. controls
(1,-.5) and (2.6-1,-.5) ..(2.1+0.5,-1);
    
    \path[name path=path2] (-0.4,-1) .. controls
(0,-0.8) and (0.4,-0.8) ..(1,-1);
   
    \path[name intersections={of=path1 and path2,by=intersection1}];  
   
    \draw[->,purple,thick] (-0.4+1,-1) .. controls
(1,-.5) and (2.6-1,-.5) ..(2.1+0.5,-1);    
    \draw[->,blue,thick] (-0.4,-1) .. controls
(0,-0.8) and (0.4,-0.8) ..(1,-1);  
   
    \fill[black,radius=2pt] (intersection1) circle;  
    \node[above right] at (intersection1) {\small$M$};  
  
\end{tikzpicture}
\hspace{0.05cm}
\begin{tikzpicture}[scale=0.7]
\coordinate (2) at (0.45,0.55);
\coordinate (5) at (1,0);
\coordinate (7) at (0.1-0.4,0.7);
\coordinate (9) at (2.6,0);
\coordinate (8) at (-.9,0.25);
\node[below,purple] at (7){\small$\alpha$};
\node[left,orange]  at (5){\small$\beta$};
\node[blue] at (2){\small$ {\gamma}$};
\draw [->,orange,thick](0.1-0.5-1,1) .. controls
(1.3-1.8,-0.1) and (2.6-1.8,-0.1) ..(2.6-1,1);
\draw [->,purple,thick](0.1-0.5-1,1) .. controls
(1.3-1.8,0.5) and (2.6-2.5,0.5) ..(2.6-2,1);
\draw [->,blue,thick](0.2,1) .. controls
(0.7,0.6) and (1.1,0.6) ..(1.6,1);
\draw (-1.5,1) -- (1.7,1);
\draw (-1.5,-1) -- (1.7,-1);
\draw[->] (0,1.3) -- (0.4,1.3);
\draw[<-] (0,-1.3) -- (0.4,-1.3);
\path[name path=path1] (0.1-0.5-1,1) .. controls
(1.3-1.8,0.5) and (2.6-2.5,0.5) ..(2.6-2,1);
     
    \path[name path=path2] (0.2,1) .. controls
(0.7,0.6) and (1.1,0.6) ..(1.6,1); 
   
    \path[name intersections={of=path1 and path2,by=intersection1}];  
   
    \fill[black,radius=2pt] (intersection1) circle;  
    \node[above right] at (intersection1) {\small$M$};  
 \end{tikzpicture}
\caption{$M$ is an exceptional intersection of  $({\gamma},\alpha)$}
\label{fig:mutation mon}
\end{figure}
\item If $f:E\to F$ is an epimorphism, then the arcs  $\alpha$, $\beta$ and $ {\gamma}$ are depicted in Figure~\ref{fig:mutation epi}, based on Proposition~\ref{cro:inj,epi}(1) and Theorem~\ref{thm:morphism}(2).  
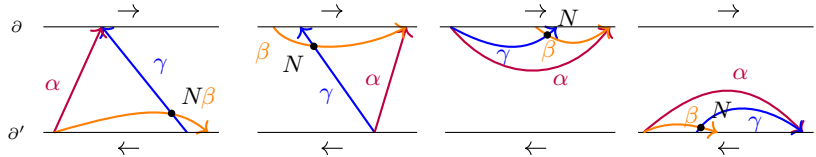
\begin{figure}[H]
\begin{tikzpicture}[scale=0.7]
\coordinate (3) at (3,-0.35);
\coordinate (5) at (0.9-0.5,-0.1);
\coordinate (6) at (3.1,-0.1);
\coordinate (7) at (0.1-0.5,0.35);
\coordinate (8) at (0,-0.35);
\coordinate (9) at (2.1,0.25);
\node[blue] at (9){\small$ {\gamma}$};
\node[left,purple]  at (5){\small$\alpha$};
\node[orange] at (3){\small$\beta$};
\draw [->,blue,thick](2.1+0.5,-1) -- (1,1);
\draw [->,purple,thick](-0.4+0.5,-1) --(2-1,1);

\draw [->,orange,thick](-0.4+0.5,-1) .. controls
(1.3+0.5,-.5) and (2.4,-.5) ..(2+1,-1);
\draw (-0.1,1) -- (3.2,1);
\draw (-0.1,-1) -- (3.2,-1);
\node()at(-0.6,1){\tiny{${\partial}$}};
\node()at(-.6,-1){\tiny{${\partial'}$}};
\draw[->] (1.3,1.3) -- (1.7,1.3);
\draw[<-] (1.3,-1.3) -- (1.7,-1.3);
\path[name path=path1] (2.1+0.5,-1) -- (1,1);  
    \path[name path=path2] (-0.4 + 0.5, -1) .. controls (1.3 + 0.5, -.5) and (2.6, -.5) .. (2.1 + 1, -1);  
    
    \path[name intersections={of=path1 and path2,by=intersection1}];   
      
    \fill[black,radius=2pt] (intersection1) circle;  
    \node[above right] at (intersection1) {\small$N$};  

\end{tikzpicture}
\hspace{0.05cm}
\begin{tikzpicture}[scale=0.7]
\coordinate (2) at (-0.6,0.5);
\coordinate (3) at (2.7,0.5);
\coordinate (5) at (1.8,0);
\coordinate (6) at (-2+.75,0);
\coordinate (7) at (0.1+0.5,0);
\coordinate (9) at (2.6,0);
\node[below,blue] at (7){\small$ {\gamma}$};
\node[left,purple]  at (5){\small$\alpha$};
\node[orange] at (2){\small$\beta$};
\draw[->,blue,thick] (1.5,-1) -- (0.1,1);
\draw [->,purple,thick](1.5,-1) --(2.6-0.5,1);
\draw [->,orange,thick](0.1-0.5,1) .. controls
(1.3-1.4,0.5) and (2.6-1.4,0.5) ..(2.6-0.5,1);
\draw (-0.7,1) -- (2.3,1);
\draw (-0.7,-1) -- (2.3,-1);
\draw[->] (0.5,1.3) -- (0.9,1.3);
\draw[<-] (0.5,-1.3) -- (0.9,-1.3);
\path[name path=path1] (1.5,-1) -- (0.1,1);  
    \path[name path=path2] (0.1-1,1) .. controls
(1.3-1.4,0.5) and (2.6-1.4,0.5) ..(2.6-0.5,1); 
   
    \path[name intersections={of=path1 and path2,by=intersection1}];

    \fill[black,radius=2pt] (intersection1) circle;  
    \node[below left] at (intersection1) {\small$N$};  
  
\end{tikzpicture}
\hspace{0.05cm}
\begin{tikzpicture}[scale=0.7]
\coordinate (2) at (0.45,0.55);
\coordinate (5) at (1,0);
\coordinate (7) at (0.1-0.5,0.75);
\coordinate (9) at (2.6,0);
\coordinate (8) at (-.9,0.25);
\node[below,blue] at (7){\small$ {\gamma}$};
\node[left,purple]  at (5){\small$\alpha$};
\node[orange] at (2){\small$\beta$};
\draw [->,purple,thick](0.1-0.5-1,1) .. controls
(1.3-1.8,-0.1) and (2.6-1.8,-0.1) ..(2.6-1,1);
\draw [->,blue,thick](0.1-0.5-1,1) .. controls
(1.3-1.8,0.5) and (2.6-2.5,0.5) ..(2.6-2,1);
\draw [->,orange,thick](0.2,1) .. controls
(0.7,0.6) and (1.1,0.6) ..(1.6,1);
\draw (-1.6,1) -- (1.7,1);
\draw (-1.5,-1) -- (1.7,-1);
\draw[->] (0,1.3) -- (0.4,1.3);
\draw[<-] (0,-1.3) -- (0.4,-1.3);
\path[name path=path1] (0.1-0.5-1,1) .. controls
(1.3-1.8,0.5) and (2.6-2.5,0.5) ..(2.6-2,1);

    \path[name path=path2] (0.2,1) .. controls
(0.7,0.6) and (1.1,0.6) ..(1.6,1); 
    
    \path[name intersections={of=path1 and path2,by=intersection1}];

    \fill[black,radius=2pt] (intersection1) circle;  
    \node[above right] at (intersection1) {\small$N$};  
\end{tikzpicture}
\hspace{0.05cm}
\begin{tikzpicture}[scale=0.7]
\coordinate (3) at (1.7,-0.85);
\coordinate (5) at (1.1,0.1);
\coordinate (6) at (4.2-0.5,-0.1);
\coordinate (9) at (0.5,-0.7);
\node[right,purple]  at (5){\small$\alpha$};
\node[blue] at (3){\small$ {\gamma}$};
\node[orange] at (9){\small$\beta$};
\draw [->,purple,thick](-0.4,-1) .. controls
(1.3+0.5-1,0.05) and (2.6-1,0.05) ..(2.1+0.5,-1);
\draw [->,blue,thick](-0.4+1,-1) .. controls
(1,-.4) and (2.6-.8,-.4) ..(2.1+0.5,-1);
\draw (-.5,1) -- (2.8,1);
\draw (-.5,-1) -- (2.8,-1);
\draw[->] (1,1.3) -- (1.4,1.3);
\draw[<-] (1,-1.3) -- (1.4,-1.3);
\draw [->,orange,thick](-0.4,-1) .. controls
(0,-0.8) and (0.4,-0.8) ..(1,-1);

    \path[name path=path1] (-0.4+1,-1) .. controls
(1,-.5) and (2.6-1,-.5) ..(2.1+0.5,-1);  
    \path[name path=path2] (-0.4,-1) .. controls
(0,-0.8) and (0.4,-0.8) ..(1,-1);
   
    \path[name intersections={of=path1 and path2,by=intersection1}];

    \fill[black,radius=2pt] (intersection1) circle;  
    \node[above right] at (intersection1) {\small$N$};  
  
\end{tikzpicture}
\caption{$N$ is an exceptional intersection of  $(\beta,{\gamma})$}
\label{fig:mutation epi}
\end{figure}
\end{itemize}
In conclusion, it follows from Definition~\ref{mutation of arcs}(1) that $L_{\alpha}\beta = {\gamma}$. Thus, we obtain the results.

\item  If  ${\rm Ext}^{1}(E,F)\cong\Bbbk$, then we have $L_{E}F=R_{F}E$.  Besides, the intersection of
 $\alpha$ and 
$\beta$ is an exceptional  intersection by Theorem~\ref{position}.  Therefore, the statement holds 
 by Proposition~\ref{prop:ext}. 
\item  If ${\rm Hom}(E,F)\cong \Bbbk^{2}$, then it follows from Corollary~\ref{cor:class}
that
$E=\mathcal{O}(\vec{x})$ and $F=\mathcal{O}(\vec{x}+\vec{c})$ for some $\vec{x} \in\mathbb{L}$. By considering the two  short exact sequences 
in $\cohx$: 
\[
\begin{tikzcd}
   0\ar[r] & \mathcal{O}(\vec{x})  \ar[r] & \mathcal{O}(\vec{x}+\vec{c})^{2} \ar[r] & \mathcal{O}(\vec{x}+2\vec{c})\ar[r]  & 0,
\end{tikzcd}
\]
\[
\begin{tikzcd}
    0\ar[r] & \mathcal{O}(\vec{x}-\vec{c})  \ar[r] & \mathcal{O}(\vec{x})^{2} \ar[r] & \mathcal{O}(\vec{x}+\vec{c})\ar[r]  & 0, 
    \end{tikzcd}
\]
we get that
 $L_{E}F=E(-\vec{c})$ and $R_{F}E=E(\vec{2c})$. 
 Hence,  we have 
\[\phi( L_{\alpha}\beta )=E(-\vec{c})=L_{E}F, \ \phi( R_{\beta}\alpha )=E(2\vec{c})=R_FE.\]
  
\item  If ${\rm Hom}(E,F)= 0={\rm Ext}^{1}(E,F)$, then 
$ \alpha $, $ \beta $ do not intersect in the interior  of $A_{p,q}$   and also do not share a starting point or an ending point. 
Therefore, this proof is completed by  Definition~ \ref{mutation of arcs}(1) and
  the definition of the   mutation of $(E,F)$. 
  \end{enumerate}
\end{proof}

From the proof of Theorem~\ref{mutation correspond}, we observe the following characteristics  of braid group action on exceptional sequences.

\begin{corollary}\label{property of braid action}
Let $(E_1, \ldots,  E_{s},  \ldots,  E_{r})$ be an exceptional sequence in $\cohx$.  Then, for any ${\rm \sigma }_s\in B_r$ and $k\in \mathbb{Z}_{>0}$,   the following results hold:
\begin{enumerate}
\item [(1)] If  ${\rm Hom}(E_{s},  E_{s+1})\cong\Bbbk^2$, then 
$$
{\rm \sigma }_s^k  \cdot (E_1, \ldots, E_{s},  E_{s}(\vec{c}),\ldots,  E_{r})=
(E_1, \ldots ,E_{s-1}, E_{s}(k\vec{c}),E_{s}((k+1)\vec{c}) , E_{s+1},\ldots,  E_{r}),$$
$$
{\rm \sigma }_s^{-k}  \cdot (E_1, \ldots, E_{s},E_{s}(\vec{c}),\ldots,  E_{r})=
(E_1, \ldots, E_{s-1},E_{s}(-k\vec{c}),E_{s}(-(k-1)\vec{c}),E_{s+1},\ldots,  E_{r}).$$

\item [(2)] If  ${\rm Hom}(E_{s},  E_{s+1})\cong\Bbbk$ or ${\rm Ext}^{1}(E_{s},  E_{s+1})\cong\Bbbk$, then 
$$
{\rm \sigma }_s^{3k}  \cdot (E_1, \ldots,  E_{s} ,\ldots,  E_{r})=
(E_1, \ldots,  E_{s}, \ldots,  E_{r})={\rm \sigma }_s^{-3k}  \cdot (E_1, \ldots,  E_{s},\ldots,  E_{r}).
$$

\item [(3)]
If  ${\rm Hom}(E_{s},  E_{s+1})=0={\rm Ext}^{1}(E_{s},  E_{s+1})$, then 
$$
{\rm \sigma }_s^{2k}  \cdot (E_1, \ldots,  E_{s} ,\ldots,  E_{r})=
(E_1, \ldots,  E_{s}, \ldots,  E_{r})={\rm \sigma }_s^{-2k}  \cdot (E_1, \ldots,  E_{s},\ldots,  E_{r}).
$$
\end{enumerate}
\end{corollary}
\begin{proof}
 Statements (1) and (3) can be verified using Theorem~\ref{mutation correspond} and Definition~\ref{mutation of arcs}. 
 For statement (2), by Theorem~\ref{mutation correspond}, we only need to
 compute the left and right mutation of order exceptional collection $(\xi,\eta)$, which  satisfies either $I^{+}(\xi,\eta)=1$ or $I^{+}(_{s}\eta_{e}, \xi)=1$.

There are   twelve cases of the order exceptional collection $(\xi,\eta)$ that satisfy 
either
  $I^{+}(\xi,\eta)=1$ or $I^{+}(_{s}\eta_{e}, \xi)=1$ (see  
  Figure~\ref{The fourth orbit}). Furthermore, 
 for any  
row  in Figure~\ref{The fourth orbit}, the order exceptional collection $(\xi_{i+1},\eta_{i+1})$ represents   $( L_{\xi_{i}}\eta_{i}, \xi_{i})$  for $i=1,2,3$, where $i$ is taken modulo  $3$.  
  Conversely, the order exceptional collection $(\xi_{i},\eta_{i})$ represents   $( \eta_{i+1},R_{\eta_{i+1}}\xi_{i+1})$, also for $i=1,2,3$, with $i$ again taken modulo  $3$.
Therefore, statement (2) holds.
\end{proof}
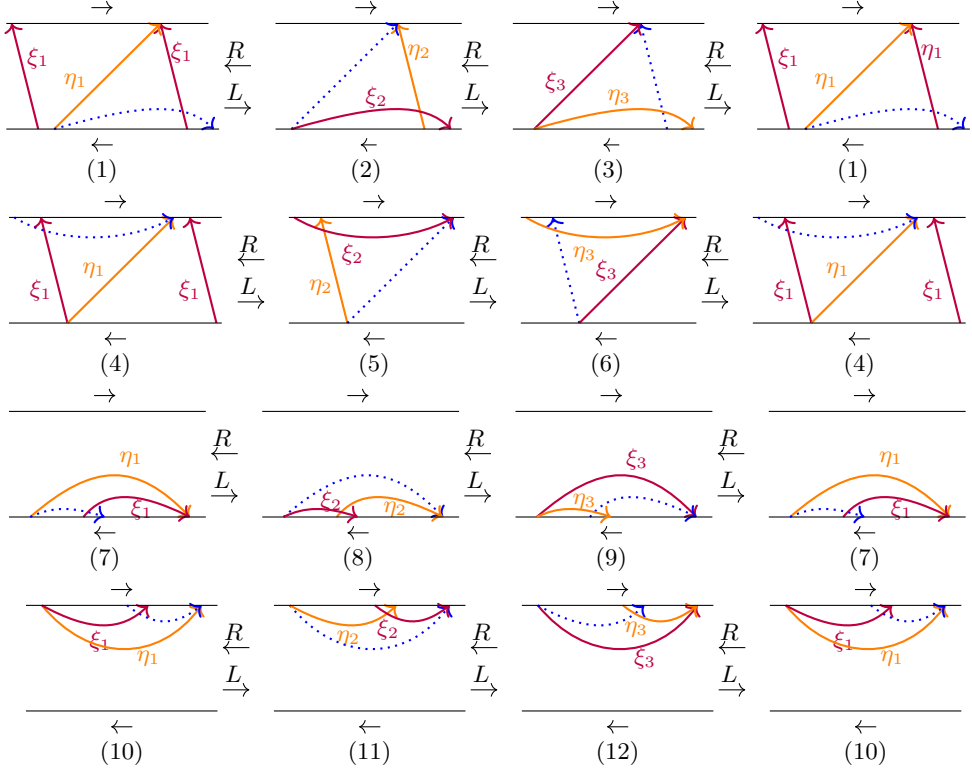
\begin{figure} [h]
\begin{tikzpicture}[scale=0.7]
\coordinate (3) at (1.7,-0.4);
\coordinate (5) at (0.9,-0.1);
\coordinate (7) at (0.1-0.3,0.35);
\coordinate (8) at (3,-0.25);
\coordinate (9) at (2.5,0.45);
\node[purple] at (7){\small$\xi_1$};
\node[purple] at (9){\small$\xi_1$};
\node[left,orange]  at (5){\small$\eta_1$};
\draw[->,purple,thick] (-0.2,-1) -- (0.3-1,1);
\draw [->,purple,thick](2.1+0.5,-1) -- (2.6-0.5,1);
\draw [->,orange,thick](-0.4+0.5,-1) --(2.6-0.5,1);
\draw [->,blue,thick,dotted](-0.4+0.5,-1) .. controls
(1.3+0.5,-.5) and (2.6,-.5) ..(2.1+1,-1);
\draw (-0.8,1) -- (3,1);
\draw (-0.8,-1) -- (3.2,-1);
\draw[->] (0.8,1.3) -- (1.2,1.3);
\draw[<-] (0.8,-1.3) -- (1.2,-1.3);
\node at (1,-1.8){\small$(1)$};
\draw[->] (3.3,-0.6) -- (3.8,-0.6);
\node  at (3.5,-0.3){\small$L$};
\draw[<-] (3.3,0.2) -- (3.8,0.2);
\node  at (3.5,0.5){\small$R$};
\end{tikzpicture}
\hspace{0.05cm}
\begin{tikzpicture}[scale=0.7]
\coordinate (3) at (1.7,-0.4);
\coordinate (5) at (0.9,-0.1);
\coordinate (7) at (0.1-0.3,0.35);
\coordinate (8) at (3,-0.25);
\coordinate (9) at (2.5,0.45);
\node[orange] at (9){\small$\eta_2$};
\node[purple] at (3){\small$\xi_2$};
\draw [->,orange,thick](2.1+0.5,-1) -- (2.6-0.5,1);
\draw [->,blue,thick,dotted](-0.4+0.5,-1) --(2.6-0.5,1);
\draw [->,purple,thick](-0.4+0.5,-1) .. controls
(1.3+0.5,-.5) and (2.6,-.5) ..(2.1+1,-1);
\draw (-0.2,1) -- (3.3,1);
\draw (-0.2,-1) -- (3.3,-1);
\draw[->] (1.3,1.3) -- (1.7,1.3);
\draw[<-] (1.4,-1.3) -- (1.7,-1.3);
\node at (1.5,-1.8){\small$(2)$};
\draw[->] (3.3,-0.6) -- (3.8,-0.6);
\node  at (3.5,-0.3){\small$L$};
\draw[<-] (3.3,0.2) -- (3.8,0.2);
\node  at (3.5,0.5){\small$R$};
\end{tikzpicture}
\hspace{0.05cm}
\begin{tikzpicture}[scale=0.7]
\coordinate (3) at (1.7,-0.4);
\coordinate (5) at (0.9,-0.1);
\coordinate (8) at (0.9+3.2,-0.1);
\coordinate (7) at (0.1-0.3,0.35);
\coordinate (9) at (2.5,0.45);
\node[left,purple]  at (5){\small$\xi_3$};
\node[orange] at (3){\small$\eta_3$};
\draw [->,blue,thick,dotted](2.1+0.5,-1) -- (2.6-0.5,1);
\draw [->,purple,thick](-0.4+0.5,-1) --(2.6-0.5,1);
\draw [->,orange,thick](-0.4+0.5,-1) .. controls
(1.3+0.5,-.5) and (2.6,-.5) ..(2.1+1,-1);
\draw (-0.3,1) -- (3.3,1);
\draw (-0.3,-1) -- (3.3,-1);
\draw[->] (1.3,1.3) -- (1.7,1.3);
\draw[<-] (1.4,-1.3) -- (1.7,-1.3);
\node at (1.5,-1.8){\small$(3)$};
\draw[->] (3.3,-0.6) -- (3.8,-0.6);
\node  at (3.5,-0.3){\small$L$};
\draw[<-] (3.3,0.2) -- (3.8,0.2);
\node  at (3.5,0.5){\small$R$};
\end{tikzpicture}
\hspace{0.05cm}
\begin{tikzpicture}[scale=0.7]
\coordinate (3) at (1.7,-0.4);
\coordinate (5) at (0.9,-0.1);
\coordinate (7) at (0.1-0.3,0.35);
\coordinate (8) at (3,-0.25);
\coordinate (9) at (2.5,0.45);
\node[purple] at (7){\small$\xi_1$};
\node[purple] at (9){\small$\eta_1$};
\node[left,orange]  at (5){\small$\eta_1$};
\draw[->,purple,thick] (-0.2,-1) -- (0.3-1,1);
\draw [->,purple,thick](2.1+0.5,-1) -- (2.6-0.5,1);
\draw [->,orange,thick](-0.4+0.5,-1) --(2.6-0.5,1);
\draw [->,blue,thick,dotted](-0.4+0.5,-1) .. controls
(1.3+0.5,-.5) and (2.6,-.5) ..(2.1+1,-1);
\draw (-0.8,1) -- (3,1);
\draw (-0.8,-1) -- (3.2,-1);
\draw[->] (0.8,1.3) -- (1.2,1.3);
\draw[<-] (0.8,-1.3) -- (1.2,-1.3);
\node at (1,-1.8){\small$(1)$};
\end{tikzpicture}
\hspace{0.05cm}
\begin{tikzpicture}[scale=0.7]
\draw[->] (3.3,-0.6) -- (3.8,-0.6);
\node  at (3.5,-0.3){\small$L$};
\draw[<-] (3.3,0.2) -- (3.8,0.2);
\node  at (3.5,0.5){\small$R$};
\coordinate (2) at (0.2,0.3);
\coordinate (5) at (1,0);
\coordinate (7) at (0.1-0.5,0);
\coordinate (9) at (2.4,0);
\coordinate (8) at (-.9,0.25);
\node[below,purple] at (7){\small$\xi_1$};
\node[below,purple] at (9){\small$\xi_1$};
\node[left,orange]  at (5){\small$\eta_1$};
\draw[->,purple,thick] (-0.4+0.5,-1) -- (0.1-0.5,1);
\draw [->,purple,thick](2.1+0.8,-1) -- (2.4,1);
\draw [->,orange,thick](-0.4+0.5,-1) --(2.6-0.5,1);
\draw [->,blue,thick,dotted](0.1-1,1) .. controls
(1.3-1.4,0.5) and (2.6-1.4,0.5) ..(2.6-0.5,1);
\draw (-1,1) -- (3,1);
\draw (-1,-1) -- (3,-1);
\draw[->] (0.8,1.3) -- (1.2,1.3);
\draw[<-] (0.8,-1.3) -- (1.2,-1.3);
\node at (1,-1.8){\small$(4)$};
\end{tikzpicture}
\hspace{0.05cm}
\begin{tikzpicture}[scale=0.7]
\draw[->] (2.4,-0.6) -- (2.9,-0.6);
\node  at (2.6,-0.3){\small$L$};
\draw[<-] (2.4,0.2) -- (2.9,0.2);
\node  at (2.6,0.5){\small$R$};
\coordinate (2) at (0.2,0.3);
\coordinate (5) at (1,0);
\coordinate (7) at (0.1-0.5,0);
\coordinate (9) at (2.4,0);
\coordinate (8) at (-.9,0.25);
\node[below,orange] at (7){\small$\eta_2$};
\node[purple] at (2){\small$\xi_2$};
\draw[->,orange,thick] (-0.4+0.5,-1) -- (0.1-0.5,1);
\draw [->,blue,thick,dotted](-0.4+0.5,-1) --(2.6-0.5,1);
\draw [->,purple,thick](0.1-1,1) .. controls
(1.3-1.4,0.5) and (2.6-1.4,0.5) ..(2.6-0.5,1);
\draw (-1,1) -- (2.3,1);
\draw (-1,-1) -- (2.3,-1);
\draw[->] (0.4,1.3) -- (0.8,1.3);
\draw[<-] (0.4,-1.3) -- (0.8,-1.3);
\node at (.6,-1.8){\small$(5)$};
\end{tikzpicture}
\hspace{0.05cm}
\begin{tikzpicture}[scale=0.7]
\coordinate (2) at (0.2,0.3);
\coordinate (4) at (1,0);
\coordinate (5) at (-2.3,0);
\coordinate (7) at (0.1-0.5,0);
\coordinate (9) at (2.4,0);
\coordinate (8) at (-.9,0.25);
\node[left,purple]  at (4){\small$\xi_3$};
\node[orange] at (2){\small$\eta_3$};
\draw[->,blue,thick,dotted] (-0.4+0.5,-1) -- (0.1-0.5,1);
\draw [->,purple,thick](-0.4+0.5,-1) --(2.6-0.5,1);
\draw [->,orange,thick](0.1-1,1) .. controls
(1.3-1.4,0.5) and (2.6-1.4,0.5) ..(2.6-0.5,1);
\draw (-1,1) -- (2.3,1);
\draw (-1,-1) -- (2.3,-1);
\draw[->] (2.4,-0.6) -- (2.9,-0.6);
\node  at (2.6,-0.3){\small$L$};
\draw[<-] (2.4,0.2) -- (2.9,0.2);
\node  at (2.6,0.5){\small$R$};

\draw[->] (0.4,1.3) -- (0.8,1.3);
\draw[<-] (0.4,-1.3) -- (0.8,-1.3);
\node at (0.6,-1.8){\small$(6)$};
\end{tikzpicture}
\hspace{0.05cm}
\begin{tikzpicture}[scale=0.7]

\coordinate (2) at (0.2,0.3);
\coordinate (5) at (1,0);
\coordinate (7) at (0.1-0.5,0);
\coordinate (9) at (2.4,0);
\coordinate (8) at (-.9,0.25);
\node[below,purple] at (7){\small$\xi_1$};
\node[below,purple] at (9){\small$\xi_1$};
\node[left,orange]  at (5){\small$\eta_1$};
\draw[->,purple,thick] (-0.4+0.5,-1) -- (0.1-0.5,1);
\draw [->,purple,thick](2.1+0.8,-1) -- (2.4,1);
\draw [->,orange,thick](-0.4+0.5,-1) --(2.6-0.5,1);
\draw [->,blue,thick,dotted](0.1-1,1) .. controls
(1.3-1.4,0.5) and (2.6-1.4,0.5) ..(2.6-0.5,1);
\draw (-1,1) -- (3,1);
\draw (-1,-1) -- (3,-1);
\draw[->] (0.8,1.3) -- (1.2,1.3);
\draw[<-] (0.8,-1.3) -- (1.2,-1.3);
\node at (1,-1.8){\small$(4)$};
\end{tikzpicture}
\hspace{0.05cm}
\begin{tikzpicture}[scale=0.7]
\coordinate (3) at (1.7,-0.85);
\coordinate (5) at (1.1,0.1);
\coordinate (6) at (4.2-0.5,-0.1);
\coordinate (9) at (0.5,-0.7);
\node[right,orange]  at (5){\small$\eta_1$};
\node[purple] at (3){\small$\xi_1$};
\draw [->,orange,thick](-0.4,-1) .. controls
(1.3+0.5-1,0.05) and (2.6-1,0.05) ..(2.1+0.5,-1);
\draw [->,purple,thick](-0.4+1,-1) .. controls
(1,-.5) and (2.6-1,-.5) ..(2.1+0.5,-1);
\draw (-.8,1) -- (2.9,1);
\draw (-.8,-1) -- (2.9,-1);
\draw[->] (3.3-0.3,-0.6) -- (3.8-0.3,-0.6);
\node  at (3.5-0.3,-0.3){\small$L$};
\draw[<-] (3.3-0.3,0.2) -- (3.8-0.3,0.2);
\node  at (3.5-0.3,0.5){\small$R$};

\draw[->] (.8,1.3) -- (1.2,1.3);
\draw[<-] (.8,-1.3) -- (1.2,-1.3);
\node at (1,-1.8){\small$(7)$};
\draw [->,blue,thick,dotted](-0.4,-1) .. controls
(0,-0.8) and (0.4,-0.8) ..(1,-1);
\end{tikzpicture}
\hspace{0.05cm}
\begin{tikzpicture}[scale=0.7]
\coordinate (3) at (1.7,-0.85);
\coordinate (5) at (1.1,0.1);
\coordinate (6) at (4.2-0.5,-0.1);
\coordinate (9) at (0.5,-0.7);
\node[orange] at (3){\small$\eta_2$};
\node[purple] at (9){\small$\xi_2$};
\draw [->,blue,thick,dotted](-0.4,-1) .. controls
(1.3+0.5-1,0.05) and (2.6-1,0.05) ..(2.1+0.5,-1);
\draw [->,orange,thick](-0.4+1,-1) .. controls
(1,-.5) and (2.6-1,-.5) ..(2.1+0.5,-1);
\draw (-.8,1) -- (2.9,1);
\draw (-.8,-1) -- (2.9,-1);
\draw[->] (3.3-0.3,-0.6) -- (3.8-0.3,-0.6);
\node  at (3.5-0.3,-0.3){\small$L$};
\draw[<-] (3.3-0.3,0.2) -- (3.8-0.3,0.2);
\node  at (3.5-0.3,0.5){\small$R$};
\node at (1,-1.8){\small$(8)$};
\draw[->] (.8,1.3) -- (1.2,1.3);
\draw[<-] (.8,-1.3) -- (1.2,-1.3);
\draw [->,purple,thick](-0.4,-1) .. controls
(0,-0.8) and (0.4,-0.8) ..(1,-1);
\end{tikzpicture}
\hspace{0.05cm}
\begin{tikzpicture}[scale=0.7]
\coordinate (3) at (1.7,-0.85);
\coordinate (5) at (1.1,0.1);
\coordinate (6) at (4.2-0.5,-0.1);
\coordinate (8) at (0.5,-0.7);
\coordinate (9) at (0.5+3.2,-0.7);
\node[right,purple]  at (5){\small$\xi_3$};
\node[orange] at (8){\small$\eta_3$};
\draw [->,purple,thick](-0.4,-1) .. controls
(1.3+0.5-1,0.05) and (2.6-1,0.05) ..(2.1+0.5,-1);
\draw [->,blue,thick,dotted](-0.4+1,-1) .. controls
(1,-.5) and (2.6-1,-.5) ..(2.1+0.5,-1);
\draw (-.8,1) -- (2.9,1);
\draw (-.8,-1) -- (2.9,-1);
\draw[->] (3.3-0.3,-0.6) -- (3.8-0.3,-0.6);
\node  at (3.5-0.3,-0.3){\small$L$};
\draw[<-] (3.3-0.3,0.2) -- (3.8-0.3,0.2);
\node  at (3.5-0.3,0.5){\small$R$};
\draw[->] (.8,1.3) -- (1.2,1.3);
\draw[<-] (.8,-1.3) -- (1.2,-1.3);
\draw [->,orange,thick](-0.4,-1) .. controls
(0,-0.8) and (0.4,-0.8) ..(1,-1);
\node at (1,-1.8){\small$(9)$};
\end{tikzpicture}
\hspace{0.05cm}
\begin{tikzpicture}[scale=0.7]
\coordinate (3) at (1.7,-0.85);
\coordinate (5) at (1.1,0.1);
\coordinate (6) at (4.2-0.5,-0.1);
\coordinate (9) at (0.5,-0.7);
\node[right,orange]  at (5){\small$\eta_1$};
\node[purple] at (3){\small$\xi_1$};
\draw [->,orange,thick](-0.4,-1) .. controls
(1.3+0.5-1,0.05) and (2.6-1,0.05) ..(2.1+0.5,-1);
\draw [->,purple,thick](-0.4+1,-1) .. controls
(1,-.5) and (2.6-1,-.5) ..(2.1+0.5,-1);
\draw (-.8,1) -- (2.9,1);
\draw (-.8,-1) -- (2.9,-1);
\draw[->] (.8,1.3) -- (1.2,1.3);
\draw[<-] (.8,-1.3) -- (1.2,-1.3);
\draw [->,blue,thick,dotted](-0.4,-1) .. controls
(0,-0.8) and (0.4,-0.8) ..(1,-1);
\node at (1,-1.8){\small$(7)$};
\end{tikzpicture}
  
\hspace{0.05cm}
\begin{tikzpicture}[scale=0.7]
\coordinate (2) at (0.45,0.55);
\coordinate (5) at (1,0);
\coordinate (7) at (0.1-0.4,0.7);
\coordinate (9) at (2.6,0);
\coordinate (8) at (-.9,0.25);
\node[below,purple] at (7){\small$\xi_1$};
\node[left,orange]  at (5){\small$\eta_1$};
\draw [->,orange,thick](0.1-0.5-1,1) .. controls
(1.3-1.8,-0.1) and (2.6-1.8,-0.1) ..(2.6-1,1);
\draw [->,purple,thick](0.1-0.5-1,1) .. controls
(1.3-1.8,0.5) and (2.6-2.5,0.5) ..(2.6-2,1);
\draw [->,blue,thick,dotted](0.2,1) .. controls
(0.7,0.6) and (1.1,0.6) ..(1.6,1);
\draw (-1.7,1) -- (1.9,1);
\draw (-1.7,-1) -- (1.9,-1);
\draw[->] (3.3-1.3,-0.6) -- (3.8-1.3,-0.6);
\node  at (3.5-1.3,-0.3){\small$L$};
\draw[<-] (3.3-1.3,0.2) -- (3.8-1.3,0.2);
\node  at (3.5-1.3,0.5){\small$R$};

\draw[->] (-0.1,1.3) -- (0.3,1.3);
\draw[<-] (-0.1,-1.3) -- (0.3,-1.3);
\node at (.1,-1.8){\small$(10)$};
\end{tikzpicture}
\hspace{0.05cm}
\begin{tikzpicture}[scale=0.7]
\coordinate (2) at (0.45,0.55);
\coordinate (5) at (1,0);
\coordinate (7) at (0.1-0.4,0.7);
\coordinate (9) at (2.6,0);
\coordinate (8) at (-.9,0.25);
\node[below,orange] at (7){\small$\eta_2$};
\node[purple] at (2){\small$\xi_2$};
\draw [->,blue,thick,dotted](0.1-0.5-1,1) .. controls
(1.3-1.8,-0.1) and (2.6-1.8,-0.1) ..(2.6-1,1);
\draw [->,orange,thick](0.1-0.5-1,1) .. controls
(1.3-1.8,0.5) and (2.6-2.5,0.5) ..(2.6-2,1);
\draw [->,purple,thick](0.2,1) .. controls
(0.7,0.6) and (1.1,0.6) ..(1.6,1);
\draw (-1.7,1) -- (1.9,1);
\draw (-1.7,-1) -- (1.9,-1);
\draw[->] (3.3-1.3,-0.6) -- (3.8-1.3,-0.6);
\node  at (3.5-1.3,-0.3){\small$L$};
\draw[<-] (3.3-1.3,0.2) -- (3.8-1.3,0.2);
\node  at (3.5-1.3,0.5){\small$R$};
\draw[->] (-0.1,1.3) -- (0.3,1.3);
\draw[<-] (-0.1,-1.3) -- (0.3,-1.3);
\node at (.1,-1.8){\small$(11)$};
\end{tikzpicture}
\hspace{0.05cm}
\begin{tikzpicture}[scale=0.7]
\coordinate (2) at (0.45,0.55);
\coordinate (3) at (0.45-3.2,0.55);
\coordinate (5) at (1,0);
\coordinate (7) at (0.1-0.4,0.7);
\coordinate (9) at (2.6,0);
\coordinate (8) at (-.9,0.25);
\node[left,purple]  at (5){\small$\xi_3$};
\node[orange] at (2){\small$\eta_3$};
\draw [->,purple,thick](0.1-0.5-1,1) .. controls
(1.3-1.8,-0.1) and (2.6-1.8,-0.1) ..(2.6-1,1);
\draw [->,blue,thick,dotted](0.1-0.5-1,1) .. controls
(1.3-1.8,0.5) and (2.6-2.5,0.5) ..(2.6-2,1);
\draw [->,orange,thick](0.2,1) .. controls
(0.7,0.6) and (1.1,0.6) ..(1.6,1);
\draw (-1.7,1) -- (1.9,1);
\draw (-1.7,-1) -- (1.9,-1);
\draw[->] (3.3-1.3,-0.6) -- (3.8-1.3,-0.6);
\node  at (3.5-1.3,-0.3){\small$L$};
\draw[<-] (3.3-1.3,0.2) -- (3.8-1.3,0.2);
\node  at (3.5-1.3,0.5){\small$R$};
\draw[->] (-0.1,1.3) -- (0.3,1.3);
\draw[<-] (-0.1,-1.3) -- (0.3,-1.3);
\node at (.1,-1.8){\small$(12)$};
\end{tikzpicture}
\hspace{0.05cm}
\begin{tikzpicture}[scale=0.7]
\coordinate (2) at (0.45,0.55);
\coordinate (5) at (1,0);
\coordinate (7) at (0.1-0.4,0.7);
\coordinate (9) at (2.6,0);
\coordinate (8) at (-.9,0.25);
\node[below,purple] at (7){\small$\xi_1$};
\node[left,orange]  at (5){\small$\eta_1$};
\draw [->,orange,thick](0.1-0.5-1,1) .. controls
(1.3-1.8,-0.1) and (2.6-1.8,-0.1) ..(2.6-1,1);
\draw [->,purple,thick](0.1-0.5-1,1) .. controls
(1.3-1.8,0.5) and (2.6-2.5,0.5) ..(2.6-2,1);
\draw [->,blue,thick,dotted](0.2,1) .. controls
(0.7,0.6) and (1.1,0.6) ..(1.6,1);
\draw (-1.7,1) -- (1.9,1);
\draw (-1.7,-1) -- (1.9,-1);
\draw[->] (-0.1,1.3) -- (0.3,1.3);
\draw[<-] (-0.1,-1.3) -- (0.3,-1.3);
\node at (.1,-1.8){\small$(10)$};
\end{tikzpicture}
\caption{The bule  dotted arc represent the arc  $L_{\xi_i}\eta_i=R_{\eta_i}\xi_i$ for $i=1,2,3$}
\label{The fourth orbit}
\end{figure}

\begin{corollary}
Let $\epsilon =(E_1,E_2, \ldots,  E_{r})$ be an exceptional sequence in $\cohx$.  
\begin{enumerate}
\item  For any
$\sigma\in B_r$, there exists a line bundle
 in $\sigma\epsilon$ if and only if   there exists     a line bundle
 in $\epsilon$. 
\item If $(E_1,E_2)$  satisfies  either ${\rm Hom}(E_1,E_2)=\Bbbk$ or ${\rm Ext}^{1}(E_1,E_2)=\Bbbk$, then  the following   hold:
$$L_{(L_{E_1}E_2)}E_1=E_2\quad {\rm and}\quad  L_{E_2}(L_{E_1}E_2)=E_1.$$ Dually, $$R_{(R_{E_2}E_1)}E_2=E_1\quad  {\rm and}\quad  R_{E_1}(R_{E_2}E_1)=E_2.$$
\end{enumerate}
\end{corollary}
\begin{proof}
This conclusion follows directly from the proof of Corollary~\ref{property of braid action}(2).
\end{proof}

Our discussion now turns to the transitivity of the braid group action on the set of ordered exceptional collections in $A_{p,q}$.  To explore this, we first establish several technical lemmas.

Note that different choices of order for a given exceptional collection in $A_{p,q}$ give rise to  distinct ordered exceptional collections. However, we show that   all these ordered collections are interconnected via the braid group action.

\begin{lemma}\label{same collection}
Assume that $\Gamma=(\gamma_1,\gamma_2,\ldots,\gamma_r)$ and $\Gamma'=(\gamma'_1,\gamma'_2,\ldots,\gamma'_r)$ are two ordered exceptional collections in $A_{p,q}$ derived from the same exceptional collection. Then there exists  an element $\sigma$ in $B_r$ such that $\sigma \Gamma=\Gamma'$.
\end{lemma}
\begin{proof}
Suppose there exists some $1<i \leq r$ such that $\gamma_i = \gamma'_1$. Consequently, the pair $(\gamma_j, \gamma_i)$ fails to meet   all  the conditions outlined in Definition~\ref{prop:order}, that is,  $L_{\gamma_j}\gamma_i = \gamma_i$ holds for $1 \leq j \leq i-1$.  

By applying the braid group generators $\sigma_1, \ldots, \sigma_{i-1}$ to $\Gamma$, we can move $\gamma_i$ (which is $\gamma'_1$) to the first position, shifting $\gamma_1, \ldots, \gamma_{i-1}$ one position to the right:
\[
\sigma_{i-1} \ldots \sigma_1 \Gamma = (\gamma_i = \gamma'_1, \gamma_1, \ldots, \gamma_{i-1}, \gamma_{i+1}, \ldots, \gamma_r).
\]

Next, locate $\gamma'_2$  in the  ordered collection $\sigma_{i-1} \ldots \sigma_1 \Gamma$, and perform a similar mutation to position it immediately after $\gamma'_1$. Repeating this process for all elements of $\Gamma'$, we can  mutate $\Gamma$ into $\Gamma'$. This inductive procedure constructs an element $\sigma$ of the braid group such that $\sigma \Gamma = \Gamma'$.
\end{proof}

 Any ordered exceptional collection can be transformed via mutations  to yield  one where bridging arcs and peripheral arcs do not intersect in the interior of $A_{p,q}$. 
 \begin{lemma}\label{lem:5.7}
     Let $\Gamma=(\gamma_1,\gamma_2,\ldots,\gamma_r)$ be an ordered exceptional collection in $A_{p,q}$. Then there exists an   element  $\sigma\in B_r$ such that  $\sigma \Gamma=(\gamma'_1,\gamma'_2,\ldots,\gamma'_r)$
 satisfies the following condition:  if there is an exceptional intersection between 
$\gamma'_i$  and $\gamma'_j$  then both   $\gamma'_i$  and $\gamma'_j$ are  peripheral arcs for $1\leq i<j\leq r$. 
\end{lemma}
  The following lemma
 allows us to give  an inductive proof of above assertion.

\begin{lemma}\label{no exceptional int}
Let $\Gamma=(\gamma_1,\gamma_2,\ldots,\gamma_r)$ be an ordered exceptional collection in $A_{p,q}$, where 
 \[
\gamma_i=\begin{cases}
[D_{ {a_i} ,0}],  &{\rm for}\ 1\leq i\leq k,\\
[D^{0, {b_i} }],  &{\rm for}\ k+1\leq i\leq r-1,\\
[D_{-\frac{1}{q}}^{\frac{1}{p}}],  &{\rm for}\ i=r,\\
\end{cases} 
    \] 
    with $ -1\leq a_k<\ldots<a_1<-\frac{1}{q}$ and $\frac{1}{p}<b_{k+1}<\ldots<b_{r-1}\leq1$. 
 Then 
$$\sigma_{k}\ldots\sigma_{r-1} \Gamma=(\gamma_1,\ldots,\gamma_{k-1},[D_{ {a_{k}} }^{ {b_{r-1}} }],\gamma_{k},\ldots,\gamma_{r-1}).$$
\end{lemma}
\begin{proof}
By Definition~\ref{mutation of arcs}, we have
 $L_{\gamma_{r-1}}\gamma_r=[D_{-\frac{1}{q}}^{b_{r-1}}]$,
  $L_{\gamma_{k}}[D_{-\frac{1}{q}}^{b_{r-1}}]=[D_{a_{k}}^{b_{r-1}}]$ and $L_{\gamma_{i}}[D_{-\frac{1}{q}}^{b_{r-1}}]=[D_{-\frac{1}{q}}^{b_{r-1}}]$ for all $k+1\leq i\leq r-2$. Thus this statement holds.  
\end{proof}

\begin{proof}[The proof of Lemma~\ref{lem:5.7}] 
Let $n$ represent the number of    bridging arcs $\gamma$ in   $\Gamma$ such that ${\rm Int}^{+}(\gamma', \gamma)=1$ for some $\gamma'\in\Gamma$.
We will prove the statement  by induction on $n$.
 
If $n=0$, then there is nothing to show. Assume that $n>0$.
According to Lemma~\ref{same collection}
  and   Lemma~\ref{no exceptional int}, there is an   element  $\sigma'\in B_r$ such that  the ordered exceptional collection $\sigma' \Gamma$   contains  $n-1$
  positive bridging arcs $\overline{\gamma}$ with the property  that ${\rm Int}^{+}(\overline{\gamma}', \overline{\gamma})=1$ for some $\overline{\gamma}'\in\sigma' \Gamma$.
By the inductive hypothesis, we conclude that the result holds for this case as well.   
\end{proof}

\begin{proposition}
\label{parallel}
 Let $\Gamma $ be a maximal ordered exceptional collection in $ A_{p,q}$.
Suppose that  if there is an exceptional intersection of $\gamma_i$  and $\gamma_j$, then  $\gamma_i$  and $\gamma_j$ are  peripheral arcs in $\Gamma$. 
Denote by  $S_1$ (resp. $S_2$)   the set of external points of $\Gamma$ on the inner (resp. outer) boundary of $A_{p,q}$.
Then, there exist $h$ marked points $\{M_1, \ldots, M_h\}$ in $S_1$ with at least two positive bridging arcs in $\Gamma$ ending at each if and only if there exist $h$ marked points $\{N_1, \ldots, N_h\}$ in $S_2$ with at least two positive bridging arcs in $\Gamma$ starting from each, where $1 \leq h \leq \min\{p, q\}$.
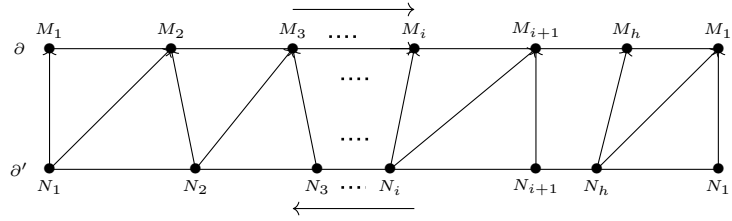
\[
\begin{tikzpicture}[scale=0.8]
\draw[->](-4,-0.5-0.5)--(-4,1);
\draw[->](4,-0.5-0.5)--(4,1);
\draw[->](4-2.4,-1)--(4,1);
\draw[->](-4,-1)--(-2,1);
\draw[->](4-2.4,1)--(2,1);
\draw[->](-4+2.4,-0.5-0.5)--(-2,1);
\draw[->](0.4,-1)--(0,1);
\draw[->](-4+2.4,-1)--(0,1);
\node()at(0,1){$\bullet$};
\node()at(0.4,-0.5-0.5){$\bullet$};
\draw[->](5,-1)--(5.5,1);
\draw[->](7,-1)--(7,1);
\draw[->](5,-1)--(7,1);
\node()at(7,1){$\bullet$};
\node()at(7,-0.5-0.5){$\bullet$};
\node()at(5,-1){$\bullet$};
\node()at(5.5,1){$\bullet$};
\node()at(5,-1.3){\tiny{$N_h$}};
\node()at(7,-1.3){\tiny{$N_1$}};
\node()at(5.5,1.3){\tiny{$M_h$}};
\node()at(7,1.3){\tiny{$M_1$}};
\draw[-](-4,-0.5-0.5)--(7,-0.5-0.5);
\draw[-](-4,1)--(7,1);
\draw[->](0,1.65)--(2,1.65);
\draw[->](2,-1.15-0.5)--(0,-1.15-0.5);

\node()at(-4,-0.5-0.5){$\bullet$};
\node()at(-4,1){$\bullet$};
\node()at(-2,1){$\bullet$};
\node()at(2,1){$\bullet$};
\node()at(-4+2.4,-0.5-0.5){$\bullet$};
\node()at(4-2.4,-0.5-0.5){$\bullet$};

\draw[->](4-2.4,-0.5-0.5)--(2,1);

\node()at(4,1){$\bullet$};
\node()at(4,-0.5-0.5){$\bullet$};

\node()at(-4,1.3){\tiny{$M_1$}};
\node()at(0,1.3){\tiny{$M_3$}};
\node()at(0.4,-1.3){\tiny{$N_3$}};
\node()at(-4.5 ,1){\tiny{${\partial}$}};
\node()at(-4.5 ,-0.5-0.5){\tiny{${\partial^{\prime}}$}};
\node()at(-4,-0.8-0.5){\tiny{$N_1$}};
\node()at(-2,1.3){\tiny{$M_2$}};
\node()at(2,1.3){\tiny{$M_i$}};
\node()at(-4+2.4,-0.8-0.5){\tiny{$N_2$}};
\node()at(4-2.4,-0.8-0.5){\tiny{$N_{i}$}};

\draw[line width=1pt,dotted](0.6,1.2)--(1.1 ,1.2);
\draw[line width=1pt,dotted](0.8,-0.8-0.5)--(1.2,-0.8-0.5);
\draw[line width=1pt,dotted](0+0.8,0.5)--(0.5+0.8 ,0.5);
\draw[line width=1pt,dotted](-0.5-0.5-0.2+2,-0.5)--(-0-0.5-0.2+2,-0.5);
\node()at(4,1.3){\tiny{$M_{i+1}$}};
\node()at(4,-0.8-0.5){\tiny{$N_{i+1}$}}; 
\end{tikzpicture}
\]
\end{proposition}
\begin{proof}
Suppose that  the   positive bridging arcs in $\Gamma$ that end at $M_i$ are $\gamma_{i1},\ldots, \gamma_{ik_{i}}$ with  $\gamma_{i1} \preceq \ldots \preceq  \gamma_{ik_{i}}$ for all $1\leq i\leq h$.
Without loss of generality, we may assume that there are no points in $S_1$ between $M_i$ and $M_{i+1}$ for all $1\leq i\leq h$, where $i$ is taken module $h$.
Denote by $N_i$ the starting point of $\gamma_{i1}$. We claim that $N_i$ is also the starting point of $\gamma_{i+1 k_{i+1}}$. In fact, thanks to the maximality of $\Gamma$ there exists a  positive bridging arc $\gamma\in\Gamma$ starting from the starting point of $\gamma_{i+1,k_{i+1}}$ and ending with $M_i$. Thus by assumption,  $\gamma$ must be the bridging arc  $\gamma_{i1}$. Therefore,  the necessity follows from Definition~\ref{def:collection}.
Since the proof of sufficiency is similar, we omit here.
\end{proof} 
Next, we show  that any ordered exceptional collection can be transformed via mutations into one consisting of positive bridging arcs.

\begin{Convention}\label{normal}
From now on, for any $x,y,k,l\in \mathbb{Z}$ with $1\leq k\leq q$, $1\leq l \leq p$ and $k+l <p+q$, we always denote by
$\Theta[x,y](k,l)=(\delta_1,\delta_2,\ldots,\delta_{k+l+1})$ is an ordered exceptional collection,
where
     \begin{equation}\label{p of t}
\delta_i=\begin{cases}
[D_{\frac{x}{q}}^{\frac{y+i-1}{p}}],  &{\rm for}\ 1\leq i\leq l,\\
[D_{\frac{x+k+l+1-i}{q}}^{\frac{y+l}{p}}],  &{\rm for}\ l+1\leq i\leq k+l+1.\\
\end{cases} 
   \end{equation}
 Intuitively, they are described as  follows:
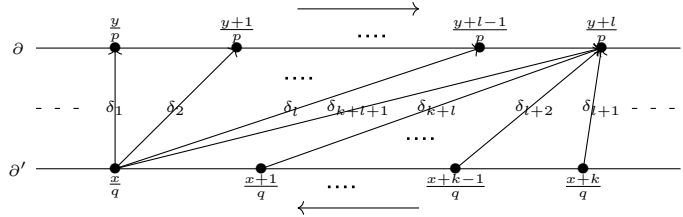
\[
\begin{tikzpicture}[scale=0.8]
\draw[->](-4,-0.5-0.5)--(-4,1);
\draw[->](3.7,-0.5-0.5)--(4,1);
\draw[->](-4,-1)--(4,1);
\draw[->](-4,-1)--(-2,1);
\draw[->](-4,-1)--(2,1);

\draw[-](-5.3,-0.5-0.5)--(5.3,-0.5-0.5);
\draw[-](-5.3,1)--(5.3,1);
\draw[->](-1,1.65)--(1,1.65);
\draw[->](1,-1.15-0.5)--(-1,-1.15-0.5);
\node()at(-4,-0.5-0.5){$\bullet$};
\node()at(-4,1){$\bullet$};
\node()at(-2,1){$\bullet$};
\node()at(2,1){$\bullet$};
\node()at(-4+2.4,-0.5-0.5){$\bullet$};
\node()at(4-2.4,-0.5-0.5){$\bullet$};
\draw[->](-4+2.4,-0.5-0.5)--(4,1);
\draw[->](4-2.4,-0.5-0.5)--(4,1);
\node()at(4,1){$\bullet$};
\node()at(3.7,-0.5-0.5){$\bullet$};


\node()at(-4,1.3){\tiny{$\frac{y}{p}$}};
\node()at(-5.6,1){\tiny{${\partial}$}};
\node()at(-5.6,-0.5-0.5){\tiny{${\partial^{\prime}}$}};
\node()at(-4,-0.8-0.5){\tiny{$\frac{x}{q}$}};
\node()at(-2,1.3){\tiny{$\frac{y+1}{p}$}};
\node()at(2,1.3){\tiny{$\frac{y+l-1}{p}$}};
\node()at(-4+2.4,-0.8-0.5){\tiny{$\frac{x+1}{q}$}};
\node()at(4-2.4,-0.8-0.5){\tiny{$\frac{x+k-1}{q}$}};
\node()at(4,1.3){\tiny{$\frac{y+l}{p}$}};
\node()at(3.7,-0.8-0.5){\tiny{$\frac{x+k}{q}$}};

\draw[line width=1pt,dotted](0,1.2)--(0.5 ,1.2);
\draw[line width=1pt,dotted](-0.5,-0.8-0.5)--(-0,-0.8-0.5);
\draw[line width=1pt,dotted](0+0.8-2,0.5)--(0.5+0.8-2 ,0.5);
\draw[line width=1pt,dotted](-0.5-0.5-0.2+2,-0.5)--(-0-0.5-0.2+2,-0.5);

\node()at(0,-0){\tiny{$\delta_{k+l+1}$}};
\node()at(-1.1,-0){\tiny{$\delta_l$}};
\node()at(-3,-0){\tiny{$\delta_2$}};
\node()at(-4,-0){\tiny{$\delta_{1}$}};
\node()at(4,-0){\tiny{$\delta_{l+1}$}};
\node()at(1.3,-0){\tiny{$\delta_{k+l}$}};
\node()at(2.9,-0){\tiny{$\delta_{l+2}$}};

\draw[dash pattern=on 2pt off 5pt on 2pt off 5pt](-5.3,0)--(-4.5,0);
\draw[dash pattern=on 2pt off 5pt on 2pt off 5pt](4.5,0)--(5.3,0);
\end{tikzpicture}
\]  
\end{Convention}

\begin{lemma}\label{2 to normal} 
 Let  $x,y,k,l$ be integers with $1\leq k\leq q$, $1\leq l \leq p$ and $k+l <p+q$.  
Suppose that $\Gamma$ is the ordered exceptional collection $(\gamma_{1},\ldots,\gamma_{k+l-1})$ in $A_{p,q}$, where 
   \[
\gamma_i=\begin{cases} 
[D_{\frac{x}{q}}^{\frac{y+l}{p}}],  &{\rm for}\ i=1,\\
[D_{\frac{a}{q},\frac{b}{q}}]\  {\rm or} \ [D^{\frac{c}{p},\frac{d}{p}}],  &{\rm for}\ 2\leq i\leq k+l-1,\\
\end{cases} 
    \] 
with $x\leq a<b\leq x+k$ and $y\leq c<d\leq y+l$. Then there exists an   element  $\sigma\in B_{k+l-1}$ such that $\sigma \Gamma= (\delta_{2},\ldots,\delta_l, \delta_{l+2},\dots,\delta_{k+l+1})$, where $\delta_i$ are given in \eqref{p of t} for all $2\leq i\leq l$ and $l+2\leq i \leq k+l+1$.
\end{lemma}
\begin{proof}
By Lemma~\ref{same collection} and the proof of Proposition~\ref{collection es}, we may assume that $\gamma_2,\gamma_3,\ldots,\gamma_{l}$ (resp. $\gamma_{l+1},\gamma_{l+2},\ldots,\gamma_{l+k-1}$) are  peripheral arcs whose endpoints on the inner (resp. outer) boundary of  $A_{p,q}$.
  According to Figure~\ref{The fourth orbit}(4)(5)(6), there exists an   element  $\sigma'\in B_{k+l-1}$ such that $$\sigma' \Gamma=(\gamma_2',\ldots,\gamma_{l}',\gamma_1,\gamma_{l+1},\ldots,\gamma_{k+l-1})$$ is an ordered exceptional collection satisfying
$
\gamma_i'=
[D_{\frac{x}{q}}^{\frac{y+i-1}{p}}]$ for $2\leq i\leq l$.
Similarly, according to Figure~\ref{The fourth orbit}(1)(2)(3), there exists an   element  $\sigma''\in B_{k+l-1}$ such that $$\sigma''\sigma' \Gamma=( \gamma_2',\ldots,\gamma_{k+l-1}',\gamma_1)$$ is an ordered exceptional collection satisfying
 $\gamma_i'=[D_{\frac{x+k+l-i}{q}}^{\frac{y+l}{p}}]$ for $l+1\leq i\leq k+l-1$. Now  our desired result follows from Lemma~\ref{same collection}.                \end{proof}

 \begin{proposition}
     \label{prop:1+1=2}
 Let  $k,l,m$ be integers with 
 $1\leq k < q$, $1\leq l < p$ and $2\leq m \leq q-k+1$. Suppose 
$\Gamma$  is an
ordered collection $(\gamma_1,\ldots,\gamma_{k+l+m})$, where $$(\gamma_1,\ldots,\gamma_{k+l+1})=\Theta[x,y](k,l)\ { \rm and} \ (\gamma_{l+1},\gamma_{k+l+2},\ldots,\gamma_{k+l+m})=\Theta[x+k,y+l](m-2,1),$$
 for some $x,y\in\mathbb{Z}$. Then $\Gamma$ is an ordered exceptional collection and  there exists an   element  $\sigma\in B_{k+l+m}$ such that $\sigma\Gamma=\Theta[x,y](k+m-2,l+1)$.
\[
\begin{tikzpicture}[scale=0.8]
\draw[->](-4,-0.5-0.5)--(-4,1);
\draw[->](3.7,-0.5-0.5)--(4,1);
\draw[->](-4,-1)--(4,1);
\draw[->](-4,-1)--(-2,1);
\draw[->](-4,-1)--(2,1);

\draw[-](-5.3,-0.5-0.5)--(8.8,-0.5-0.5);
\draw[-](-5.3,1)--(8.8,1);
\draw[->](-1,1.65)--(1,1.65);
\draw[->](1,-1.15-0.5)--(-1,-1.15-0.5);
\node()at(-4,-0.5-0.5){$\bullet$};
\node()at(-4,1){$\bullet$};
\node()at(-2,1){$\bullet$};
\node()at(2,1){$\bullet$};
\node()at(-4+2.4,-0.5-0.5){$\bullet$};
\node()at(4-2.4,-0.5-0.5){$\bullet$};
\draw[->](-4+2.4,-0.5-0.5)--(4,1);
\draw[->](4-2.4,-0.5-0.5)--(4,1);
\node()at(4,1){$\bullet$};
\node()at(3.7,-0.5-0.5){$\bullet$};
\node()at(6,1){$\bullet$};
\node()at(6,1.3){\tiny{$\frac{y+l+1}{p}$}};
\draw[->](3.7,-0.5-0.5)--(6,1);
\draw[->,red](5.8,-0.5-0.5)--(6,1);
\node()at(5.8,-1){$\bullet$};
\node()at(5.8,-0.8-0.5){\tiny{$\frac{x+k+m-3}{q}$}};
\draw[->,red](7.7,-0.5-0.5)--(6,1);
\node()at(7.7,-1){$\bullet$};
\node()at(7.7,-0.8-0.5){\tiny{$\frac{x+k+m-2}{q}$}};

\node()at(-4,1.3){\tiny{$\frac{y}{p}$}};
\node()at(-5.6,1){\tiny{${\partial}$}};
\node()at(-5.6,-0.5-0.5){\tiny{${\partial^{\prime}}$}};
\node()at(-4,-0.8-0.5){\tiny{$\frac{x}{q}$}};
\node()at(-2,1.3){\tiny{$\frac{y+1}{p}$}};
\node()at(2,1.3){\tiny{$\frac{y+l-1}{p}$}};
\node()at(-4+2.4,-0.8-0.5){\tiny{$\frac{x+1}{q}$}};
\node()at(4-2.4,-0.8-0.5){\tiny{$\frac{x+k-1}{q}$}};
\node()at(4,1.3){\tiny{$\frac{y+l}{p}$}};
\node()at(3.7,-0.8-0.5){\tiny{$\frac{x+k}{q}$}};

\draw[line width=1pt,dotted](0,1.2)--(0.5 ,1.2);
\draw[line width=1pt,dotted](-0.5,-0.8-0.5)--(-0,-0.8-0.5);
\draw[line width=1pt,dotted](0+0.8-2,0.5)--(0.5+0.8-2 ,0.5);
\draw[line width=1pt,dotted](-0.5-0.5-0.2+2,-0.5)--(-0-0.5-0.2+2,-0.5);
\draw[line width=1pt,dotted](-0.5-0.5-0.1+6,-0.5)--(-0-0.5-0.1+6,-0.5);
\node()at(0,-0){\tiny{$\gamma_{k+l+1}$}};
\node()at(-1.1,-0){\tiny{$\gamma_l$}};
\node()at(-3,-0){\tiny{$\gamma_2$}};
\node()at(-4,-0){\tiny{$\gamma_{1}$}};
\node()at(3.7,-0){\tiny{$\gamma_{l+1}$}};
\node()at(1.3,-0){\tiny{$\gamma_{k+l}$}};
\node()at(2.7,-0){\tiny{$\gamma_{l+2}$}};
\node()at(5.2+2,-0){\tiny{$\gamma_{k+l+2}$}};
\node()at(4.7,-0){\tiny{$\gamma_{k+l+m}$}};
\node()at(5+1,-0){\tiny{$\gamma_{k+l+3}$}};
\draw[dash pattern=on 2pt off 5pt on 2pt off 5pt](-5.2,0)--(-4.5,0);
\draw[dash pattern=on 2pt off 5pt on 2pt off 5pt](8.1,0)--(8.8,0);
\end{tikzpicture}
\]
 \end{proposition}

To prove this proposition, we require the following lemma.
\begin{lemma}\label{local} 
Let $\Gamma$ be an ordered exceptional collection in  $A_{p,q}$ arising from exceptional collection  
$\{[D_{\frac{x}{q}}^{\frac{y}{p}}],[D_{\frac{x+1}{q}}^{\frac{y}{p}}],[D_{\frac{x+2}{q}}^{\frac{y}{p}}],[D_{\frac{x+2}{q}}^{\frac{y+1}{p}}]\}$, where   $x,y\in\mathbb{Z}$. 
 Then there is an element $\sigma$ in $B_4$ such that $$\sigma\Gamma=([D_{\frac{x+2}{q}}^{\frac{y+1}{p}}],[D_{\frac{x+1}{q}}^{\frac{y+1}{p}}],[D_{\frac{x}{q}}^{\frac{y}{p}}],[D_{\frac{x}{q}}^{\frac{y+1}{p}}])=\Gamma'.$$  
\end{lemma}
\begin{proof}
According to Lemma~\ref{same collection}, we   suppose that $$\Gamma=([D_{\frac{x+2}{q}}^{\frac{y}{p}}],[D_{\frac{x+1}{q}}^{\frac{y}{p}}],[D_{\frac{x}{q}}^{\frac{y}{p}}],[D_{\frac{x+2}{q}}^{\frac{y+1}{p}}]).$$
It is directly to check that there is an element  $\sigma= \sigma_3^{2}\sigma_2^{2} \sigma_1 \sigma_2^{2}\sigma_3\sigma_1^{2} $ in $B_4$ such that $\sigma\Gamma=\Gamma'$. Thus this proof is completed by Lemma~\ref{same collection}.
\end{proof}

\begin{proof}[The proof of Proposition~\ref{prop:1+1=2}]
We first show the proof  of $m=2$. If $k=1$, then by Definition~\ref{mutation of arcs}, $$\sigma_{l+2}^{-1}\sigma_{l+1}^{-1}\Gamma=(\gamma_1,\ldots,\gamma_l,\gamma_{l+2},\gamma_{l+3},\gamma_{l+2}[e]).$$ We have done.
The case for $k=2$ follows from  Lemma~\ref{local}.
If $k\geq 3$, then $$\sigma_{l+4}\ldots\sigma_{k+l+1}\Gamma=(\gamma_1,\ldots,\gamma_{l+3},\gamma_{k+l+2},\gamma_{l+4},\ldots,\gamma_{k+l+1})=\Gamma_1.$$ Also by Lemma~\ref{local}, there exists an   element  $\nu_0\in B_{4}$ such that  $$\nu_0(\gamma_{l+1},\gamma_{l+2},\gamma_{l+3},\gamma_{k+l+2})
=(\gamma_{k+l+2},\gamma_{l+2}[e],\gamma_{l+3},\gamma_{l+3}[e]).$$
Furthermore, there exists an   element  $\nu_i\in B_{4}$ such that
$$\nu_{i}(\gamma_{l+2i+1},\gamma_{l+2i+1}[e],\gamma_{l+2i+2},\gamma_{l+2i+3})=(\gamma_{l+2i+1}[e],\gamma_{l+2i+2}[e],\gamma_{l+2i+3},\gamma_{l+2i+3}[e]),$$ for all $1\leq i\leq \lfloor \frac{k-2}{2}\rfloor$.
For all $0\leq i\leq \lfloor \frac{k-2}{2}\rfloor$, we denote by $\overline{\nu}_{i}$ the element in $B_{k+l+2}$ corresponding to $\nu_{i}$. Consequently, if $k$ is  even,  then
$$\overline{\nu}_{\frac{k-2}{2}}\ldots\overline{\nu}_1\overline{\nu}_0\Gamma_{1}=(\gamma_1,\ldots,\gamma_l,\gamma_{k+l+2},\gamma_{l+2}[e],\ldots,\gamma_{k+l-1}[e],\gamma_{k+l}[e],\gamma_{k+l+1},\gamma_{k+l+1}[e]);$$
If otherwise, then
$$\sigma_{k+l+1}^{-1}\sigma_{k+l}^{-1}\sigma_{k+l+1}\overline{\nu}_{\frac{k-3}{2}}\ldots\overline{\nu}_1\overline{\nu}_0\Gamma_{1}=(\gamma_1,\ldots,\gamma_l,\gamma_{k+l+2},\gamma_{l+2}[e],\ldots,\gamma_{k+l-1}[e],\gamma_{k+l+1},\gamma_{k+l}[e],\gamma_{k+l+1}[e]).$$ 
Thus the result follows from Lemma~\ref{same collection}.                                                        

Now we consider the case for $m>2$.
According to Lemma~\ref{same collection}, there is  an element $\mu_1\in  B_{k+l+m}$ such that $$\mu_1\Gamma=(\gamma_1,\ldots,\gamma_{l},\gamma_{k+l+2},\ldots,\gamma_{k+l+m-1},\gamma_{l+1},\ldots,\gamma_{k+l+1},\gamma_{k+l+m})=\Gamma_2.$$
 Proceeding as in the proof of $m=2$, we can show that  there exists an element $\mu_2\in  B_{k+l+m}$ such that
$$\mu_2\Gamma_2=(\gamma_1,\ldots,\gamma_{l},\gamma_{k+l+2},\ldots,\gamma_{k+l+m-1},\gamma_{k+l+1},\gamma_{k+l+m},\gamma_{l+2}[e],\ldots,\gamma_{k+l+1}[e]).$$
Therefore, the statement holds also by Lemma~\ref{same collection}. 
\end{proof}

\begin{theorem}\label{transitivity}
    Let $\Gamma = (\gamma_1, \ldots, \gamma_{p+q})$ be an ordered exceptional collection in $A_{p,q}$. Then there exists an element $\sigma \in B_{p+q}$ such that $\sigma \Gamma = \Theta[0,0](q-1, p)$.
\end{theorem}

\begin{proof}
    According to Convention~\ref{normal}, we denote $$\overline{\Gamma}=\Theta[0,0](q-1, p) = (\delta_1, \delta_2, \ldots, \delta_{p+q}).$$ Combining Proposition~\ref{prop:1+1=2} with Lemma~\ref{lem:5.7}, Proposition~\ref{parallel} and Lemma~\ref{2 to normal}, we can transform any maximal ordered exceptional collection via mutations into $\Theta[x, y](q-1, p)$ for some $x, y \in \mathbb{Z}$. Thus, it suffices to show that there exist elements $\sigma, \sigma' \in B_{p+q}$ such that $\sigma \overline{\Gamma} = ({_s(\delta_1)}, \ldots, {_s(\delta_{p+q})})$ and $\sigma' \overline{\Gamma} = ((\delta_{1})_e, \ldots, (\delta_{p+q})_e)$.

    From the proof of Lemma~\ref{same collection}, we have
    \[
    \sigma_2 \ldots \sigma_p \overline{\Gamma} = (\delta_1, \delta_{p+1}, \delta_2, \ldots, \delta_p, \delta_{p+2}, \ldots, \delta_{p+q}) = \Gamma_1.
    \]
By Definition~\ref{mutation of arcs},
    \[
    \sigma_1^2 \Gamma_1 = (\delta_{p+1}, \delta', \delta_2, \ldots, \delta_p, \delta_{p+2}, \ldots, \delta_{p+q}) = \Gamma_2,
    \]
    where $\delta' = [D_{-\frac{1}{q}, \frac{1}{q}}]$. Note that $\delta'$ has exceptional intersections with $\delta_i$ and $L_{\delta'} \delta_i = {_s(\delta_i)}$ for all $2 \leq i \leq p$, we get
    \[
    \sigma_p \ldots \sigma_2 \Gamma_2 = (\delta_{p+1}, {_s(\delta_2)}, \ldots, {_s(\delta_p)}, \delta', \delta_{p+2}, \ldots, \delta_{p+q}) = \Gamma_3.
    \]
Using the proof of Lemma~\ref{same collection} again, we obtain
    \[
    \sigma_{p+q-2} \ldots \sigma_{p+1} \Gamma_3 = (\delta_{p+1}, {_s(\delta_2)}, \ldots, {_s(\delta_p)}, \delta_{p+2}, \ldots, \delta_{p+q-1}, \delta', \delta_{p+q}) = \Gamma_4.
    \]
 Since $R_{\delta_{p+q}} \delta' =  {_s(\delta_{p+q}})$, 
 $\delta_{p+1} = {_s(\delta_1})$  and $\delta_i = {_s(\delta_{i-1}})$ for $p+2 \leq i \leq p+q$,
    \[
    \sigma_{p+q-1}^2 \Gamma_4 = ({_s(\delta_1)}, \ldots, {_s(\delta_{p+q})}).
    \]
Similarly, we can find $\sigma' \in B_{p+q}$ such that $\sigma' \overline{\Gamma}= ((\delta_{1})_{e}, \ldots, (\delta_{p+q})_{e})$, which completes the proof.
\end{proof}

Combining  with Theorem~\ref{collection and sequence}, this theorem implies that
 the action of braid  group $B_{p+q}$ on the set of complete exceptional sequences in $\cohx$ is transitivity.


\section{Combinatorics of  tilting  bundles and tilting sheaves in $\cohx$}\label{sec.6}

In this section, we provide a combinatorial description of tilting bundles and count the number of tilting sheaves in
  $\cohx$, up to the Auslander-Reiten translation. To facilitate the statement of the main theorems, we give the following definition.
\begin{definition}
  Two  tilting sheaves $T$ and $T'$ in $\cohx$ are said to be \emph{$\tau$-equivalent}, if there exists some  $k\in\mathbb{Z}$ such that
$T'=\tau^{-k}T$.  
\end{definition}
 \subsection{Combinatorial descriptions of tilting bundles} 
  In this subsection, we explore the combinatorial descriptions of tilting bundles in $\cohx$, beginning with   a key observation.

\begin{lemma}\label{lemma:tau}
    For any tilting bundle  $T$ in $\cohx$, there exists a tilting bundle $\overline{T}$ containing $\co$  as a direct summand  such that  $T$ is $\tau$-equivalent to $\overline{T}$.
\end{lemma}
\begin{proof}
Based on Definition~\ref{geometric} and Proposition~\ref{prop:tau}, without loss of generality, we can assume  that the tilting bundle $T$ has the decomposition:
\begin{equation}\label{form}
    T = \bigoplus_{i=1}^{p+q} \phi([D_{\frac{b_i}{q}}^{\frac{   a_i}{p}}]),
\end{equation}
 where   the parameters satisfy: $0=a_1\leq a_2\leq \dots\leq a_{p+q}\leq p $  and $b_1\leq b_2\leq \dots\leq b_{p+q}$ with $-q\leq b_1<q$.  
The result holds trivially when $b_1 = 0$.   We therefore focus on the  non-trivial cases $-q \leq b_1 < 0$; the dual case $0 < b_1 < q$ follows by analogous arguments.

For the case where $-q \leq b_1 < 0$, we can assume $a_2 = 1$ and $b_2 = b_1$. If this condition fails, then $a_2 = 0$ and $b_2 = b_1 + 1$. In this case, we define:
    \[
    [D_{\frac{b_{p+q}'}{q}}^{\frac{a_{p+q}'}{p}}] := [D_{\frac{b_1}{q} + 1}^1], \quad [D_{\frac{b_i'}{q}}^{\frac{a_i'}{p}}] := [D_{\frac{b_{i+1}}{q}}^{\frac{a_{i+1}}{p}}] \quad \text{for } 1 \leq i \leq p+q-1.
    \] 
This yields a  decomposition:
$$T = \bigoplus_{i=1}^{p+q} \phi([D_{\frac{b_i'}{q}}^{\frac{a_i'}{p}}])$$
preserving $a_1' = 0$ and $-q<b_1' \leq0$. Thus,  we may restrict our attention to the case where  $a_2 = 1$ and $b_2 = b_1$.
Now we use induction on  $|b_1|$ to  show that  the result holds for any tilting bundle $T$ decomposed as in equation \eqref{form}  with $-q \leq b_1 < 0$.  

For the base case where  $|b_1| = 1$, it follows from 
$a=1$ and $b=-1$ 
that $ \tau T$ contains $\co$  as a direct summand.  
For $|b_1|>1$, we consider the bundle $  \tau T$, which can be check that $\tau T$ decomposes as
    \[
   \tau T = \bigoplus_{i=1}^{p+q} \phi([D_{\frac{f_i}{q}}^{\frac{e_i}{p}}]),
    \]
    with $e_1 = 0$, $e_2 = 1$, and $b_1 + 1\leq f_2 = f_1 \leq 0$. By the induction hypothesis,  
    there exists a tilting bundle $\overline{T}$ containing $\co$  as a direct summand  such that  $\tau T$ is $\tau$-equivalent to 
    $\overline{T}$. The proof is now complete.
\end{proof}
 
We will give  a combinatorial description of the  
 $\tau$-equivalence classes of tilting bundles in  $\cohx$.  To this end,
  we recall the concept of a lattice path.
  \begin{definition}[\cite{dyck}]\label{lattice and dyck}
A sequence of lattice points 
$P=\{(x_1,y_1),(x_2,y_2)\dots, (x_k,y_k)\}$ in $\mathbb{Z}^{2}$
  is   a \emph{lattice path} from $(x_1,y_1)$ to $(x_k,y_k)$ 
if   $(x_{i+1}, y_{i+1})\in\{(x_i, y_i + 1),(x_i+1,y_i)\}$  for every $1\leq i\leq k-1$.
Furthermore,  a lattice path $P$ from $(0,0)$ to $(m,n)$ such that $m y_i\leq n x_i$ for all $i$   is
 called an \emph{$(m,n)$-Dyck path}. Denote by $\mathcal{L}(m,n)$  the set of all lattice paths from $(0,0)$ to $(m,n)$,   and   $\mathcal{D}(m,n)$ the set of all  $(m,n)$-Dyck paths. 
  \end{definition}

\begin{theorem}\label{number  tilting bundles}
There is a one-to-one correspondence between  the   $\tau$-equivalence classes of tilting bundles in  $\cohx$ 
and the  lattice paths from $(0,0)$ to $(p,q)$.  Consequently, 
the number of $\tau$-equivalence classes of tilting bundles in  $\cohx$ is given by
  $$C_{p+q}^{p}=\binom{p+q}{p}.$$  
\end{theorem}
\begin{proof}
Let $\mathcal{T}^\nu$ denote the set of $\tau$-equivalence classes of tilting bundles in $\cohx$. By Lemma~\ref{lemma:tau}, we may assume that each $T \in \mathcal{T}^\nu$ is of the form:
\[ T = \bigoplus_{i=1}^{p+q} \phi([D_{\frac{b_i}{q}}^{\frac{   a_i}{p}}]),\]
where $0 = a_1 \leq a_2 \leq \cdots \leq a_{p + q} \leq p$ and $0 = b_1 \leq b_2 \leq \cdots \leq b_{p + q} \leq q$.  
Moreover,  \cite[Theorem 4.3]{CRZ23} 
 implies that for each $i = 1,2,\dots,p+q-1$:
\[
(a_{i+1}, b_{i+1})\in\{(a_i, b_i + 1),(a_i+1,b_i)\} \text{ and } (a_{p+q}, b_{p+q}) \in\{(q, p-1),(q-1,p)\}.
\] Thus we define a map
\begin{equation}\label{tilting and path}
    \psi:\mathcal{T}^{\nu}\to \mathcal{L}(p,q),  \quad \bigoplus_{i=1}^{p+q} \phi([D_{\frac{b_i}{q}}^{\frac{   a_i}{p}}])\mapsto  \{(a_1,b_1),(a_2,b_2)\dots, (a_{p+q},b_{p+q}),(p,q)\}.
\end{equation}

To establish the reverse direction, consider a lattice path   $\{(x_1,y_1),(x_2,y_2)\dots, (x_{p+q},y_{p+q}),(p,q)\}$   from $(0,0)$ to $(p,q)$. According to the definition of  lattice paths,  the set
\[\{[D_{\frac{y_1}{q}}^{\frac{   x_1}{p}}],[D_{\frac{y_2}{q}}^{\frac{   x_2}{p}}], \dots,[D_{\frac{y_{p+q}}{q}}^{\frac{   x_{p+q}}{p}}]\}
\] forms a triangulation of $A_{p,q}$. Also by \cite[Theorem 4.3]{CRZ23},  the object 
$\bigoplus_{i=1}^{p+q} \phi([D_{\frac{y_i}{q}}^{\frac{   x_i}{p}}])$ is an element in  $\mathcal{T}^{\nu}$. Hence, by above construction, the map $\psi$ is a bijection from  $\mathcal{T}^{\nu}$ to $\mathcal{L}(p,q)$.

Given that the number of lattice paths from $(0, 0)$ to $(p, q)$ is the binomial coefficient  $C_{p+q}^{p}$, the proof is complete.
\end{proof} 
Recall from \cite{MR4173177} that a  tilting bundle that contains $\co$  as a direct summand of minimal degree is called a \emph{fundamental tilting bundle}.
As a corollary of Theorem~\ref{number  tilting bundles}, we give a  combinatorial description of fundamental tilting bundles. 

\begin{corollary}\label{Dyck 1}
    There exists a bijection between  the  fundamental tilting bundles in  $\cohx$ 
and the  $(p,q)$-Dyck paths.  Furthermore, 
the number of fundamental tilting bundles in  $\cohx$  is  
\begin{equation}\label{introeqmain}
C{(p,q)}=\sum_{a }\prod _{i=1}^d \left( \frac{1}{a_i!}A_{(\frac{i}{d}p,\frac{i}{d}q)}^{a_i} \right) ,
\end{equation}
where  $d={\rm  gcd}(p,q)$, $A_{(k,l)}$ is the binomial coefficient $C_{k+l}^{k}$ and the sum $\sum_a$ is taken over all sequences of non-negative integers $a=(a_1, a_2, \cdots )$ such that $\sum_{i=1}^{\infty }ia_i =d$.
\end{corollary}
\begin{proof}
Denote by $\mathcal{T}_{o}^{\nu}$   the set of  the   fundamental tilting bundle in the category $\cohx$.  Then $\mathcal{T}_{o}^{\nu}$ can be viewed as a subset of $\mathcal{T}^{\nu}$, as defined in the proof of Theorem~\ref{number tilting bundles}.  According to \cite[Theorem 4.3]{CRZ23}, 
a  tilting bundle including $\co$ as a direct summand  is of the form
   \begin{equation}\label{tilting with o}
       T = \bigoplus_{i=1}^{p+q} \phi([D_{\frac{b_i}{q}}^{\frac{   a_i}{p}}])=\bigoplus_{i=1}^{p+q} \co(a_i\vec{x}_1-b_i\vec{x}_2),
   \end{equation}  with $0=a_1\leq a_2\leq \dots\leq a_{p+q}\leq p $ and  $0=b_1\leq b_2\leq \dots\leq b_{p+q}\leq q$.  
 
By   Definition~\ref{lattice and dyck}, for any $T$ decomposed
as in equation~\eqref{tilting with o}, the image of $T$ under the  bijective  $\psi$ given in \eqref{tilting and path} 
 is a $(p,q)$-Dyck path if and only if $p  b_i\leq q  a_i$,   which is equivalent to $\delta(a_i \vec{x}_1 - b_i \vec{x}_2) \geq 0$. Therefore, the bijection $\psi: \mathcal{T}^{\nu} \to \mathcal{L}(p,q)$ restricts to a bijection between $\mathcal{T}_{o}^{\nu}$ and $\mathcal{D}(p,q)$. Consequently, \cite[Theorem~1.1]{dyck} confirms that the number of fundamental tilting bundles in $\cohx$ is $C(p,q)$.
\end{proof}

\subsection{Enumeration of tilting sheaves}
In this subsection, we focus on  counting the  tilting sheaves in the category $\cohx$, up to  $\tau$-equivalence. As shown in \cite[Theorem 4.3]{CRZ23}, there is a bijection between tilting sheaves in  $\cohx$ and triangulations of $A_{p,q}$. Consequently, our task reduces to computing the number of triangulations of  $A_{p,q}$ up to the  iterated elementary moves, which we formalize in the following definition.

\begin{definition}
Two  triangulations $\Gamma = \{\gamma_1,  \ldots, \gamma_{p+q}\}$ and $\Gamma' = \{\gamma_1', \ldots, \gamma_{p+q}'\}$ of $A_{p,q}$ are said to be \emph{$se$-equivalent}, denoted by $\Gamma' \sim \Gamma$, if there exists some $k \in \mathbb{Z}$ such that $\gamma_i' = {_{s^{k}}}(\gamma_i)_{e^{k}}$ for all $1 \leq i \leq p+q$. 
\end{definition}
 
Let $\mathcal{T}$ denote the set of triangulations of $A_{p,q}$. We introduce four distinct classes of triangulations for $A_{p,q}$ as follows:
Firstly, define
$$\Gamma(0,1):=\left\{\{\gamma_1,\gamma_2,\dots,\gamma_{p+q}\}\in\mathcal{T}\middle|\gamma_1=\left[D_{0}^{0}\right],\gamma_2=\left[D_{\frac{1}{q}}^{0}\right]\right\};$$
and
$$\Gamma(0,1)':=\left\{\{\gamma_1,\gamma_2,\dots,\gamma_{p+q}\}\in\mathcal{T}\middle|\gamma_1=\left[D_{0}^{0}\right],\gamma_2=\left[D_{0}^{\frac{1}{p}}\right]\right\};$$
For $a\in\mathbb{Z}\setminus\{0\}$ or $b\in\mathbb{Z}\setminus\{1\}$, we further define
$$\Gamma(a,b):=\left\{\{\gamma_1,\gamma_2,\dots,\gamma_{p+q}\}\in\mathcal{T}\middle|\gamma_1=\left[D_{\frac{a}{q}}^{0}\right],\gamma_2=\left[D_{\frac{b}{q}}^{0}\right],\gamma_3=\left[D_{\frac{a}{q},\frac{b}{q} }\right] \text{ with } a\leq 0<b\right\};$$
and
$$\Gamma(a,b)':=\left\{\{\gamma_1,\gamma_2,\dots,\gamma_{p+q}\}\in\mathcal{T}\middle|\gamma_1=\left[D_{\frac{a}{q}}^{0}\right],\gamma_2=\left[D_{\frac{a}{q}}^{\frac{b}{p}}\right],\gamma_3=\left[D^{0,\frac{b}{p} }\right] \text{ with }a<0 \text{ and } b>1 \right\}.$$  
 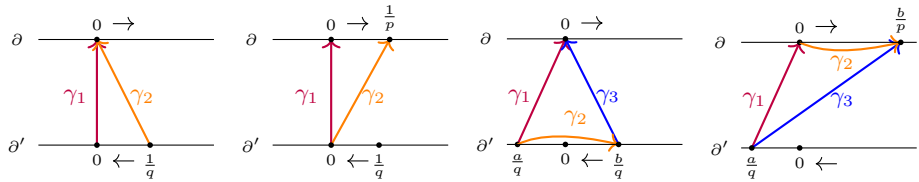
\begin{figure} [H]
\begin{tikzpicture}[scale=0.7]
\coordinate (3) at (1.2,-0.5);
\coordinate (5) at (1,-0.1);
\coordinate (6) at (3.1,-0.1);
\coordinate (7) at (0.1-0.5,0.35);
\coordinate (8) at (0,-0.35);
\coordinate (9) at (2.1,0.25);
\node[left,purple]  at (5){\small$\gamma_1$};
\draw [->,purple,thick](1,-1) --(2-1,1);
\draw [->,orange,thick](2,-1) --(2-1,1);
\node[right,orange]  at (1.45,-0.1){\small$\gamma_2$};
\fill  (1,1) circle[radius=1.5pt];
\node[above] at  (1,1){\tiny$0$}; 
\node[below] at  (1,-1){\tiny$0$}; 
\fill  (1,-1) circle[radius=1.5pt]; 
\node[below] at  (2,-1){\tiny$ \frac{1}{q}$};  
\fill (2,-1) circle[radius=1.5pt];
\draw (-0.1,1) -- (3.2,1);
\draw (-0.1,-1) -- (3.2,-1);
\node()at(-0.5,1){\tiny{${\partial}$}};
\node()at(-0.5,-1){\tiny{${\partial'}$}};
\draw[->] (1.3,1.3) -- (1.7,1.3);
\draw[<-] (1.3,-1.3) -- (1.7,-1.3);
\end{tikzpicture}
 \hspace{0.01cm}
\begin{tikzpicture}[scale=0.7]
\coordinate (3) at (1.2,-0.5);
\coordinate (5) at (0.95,-0.1);
\coordinate (6) at (3.1,-0.1);
\coordinate (7) at (0.1-0.5,0.35);
\coordinate (8) at (0,-0.35);
\coordinate (9) at (2.1,0.25);
\node[left,purple]  at (5){\small$\gamma_1$};
\draw [->,purple,thick](1,-1) --(2-1,1);
\draw [->,orange,thick](1,-1) --(2.1,1);
\node[right,orange]  at (1.4,-0.1){\small$\gamma_2$};
\fill  (1,1) circle[radius=1.5pt];
\node[above] at  (1,1){\tiny$0$}; 
\node[below] at  (1,-1){\tiny$0$}; 
\fill  (1,-1) circle[radius=1.5pt];
\node[below] at  (1.9,-1){\tiny$ \frac{1}{q}$};  
\fill (1.9,-1) circle[radius=1.5pt];
\node[above] at  (2.1,1){\tiny$ \frac{1}{p}$};  
\fill (2.1,1) circle[radius=1.5pt];
\draw (-0.1,1) -- (3.2,1);
\draw (-0.1,-1) -- (3.2,-1);
\node()at(-0.5,1){\tiny{${\partial}$}};
\node()at(-0.5,-1){\tiny{${\partial'}$}};
\draw[->] (1.3,1.3) -- (1.7,1.3);
\draw[<-] (1.3,-1.3) -- (1.7,-1.3);
\end{tikzpicture}
\hspace{0.01cm}
\begin{tikzpicture}[scale=0.7]
\coordinate (3) at (1.2,-0.5);
\coordinate (5) at (0.9-0.35,-0.1);
\coordinate (6) at (3.1,-0.1);
\coordinate (7) at (0.1-0.5,0.35);
\coordinate (8) at (0,-0.35);
\coordinate (9) at (2.1,0.25);
\node[left,purple]  at (5){\small$\gamma_1$};
\node[orange] at (3){\small$\gamma_2$};
\draw [->,purple,thick](-0.4+0.5,-1) --(2-1,1);
\draw [->,blue,thick](2,-1) --(2-1,1);
\node[right,blue]  at (1.4,-0.1){\small$\gamma_3$};
\fill  (1,1) circle[radius=1.5pt];
\node[above] at  (1,1){\tiny$0$}; 
\node[below] at  (1,-1){\tiny$0$}; 
\fill  (1,-1) circle[radius=1.5pt];
\node[below] at  (-0.4+0.5,-1){\tiny$ \frac{a}{q}$};  
\fill (-0.4+0.5,-1) circle[radius=1.5pt];
\draw [->,orange,thick](-0.4+0.5,-1) .. controls
(0+0.5,-0.8) and (0.4+0.75,-0.8) ..(1+1,-1);
\node[below] at  (2,-1){\tiny$ \frac{b}{q}$};  
\fill (2,-1) circle[radius=1.5pt];
\draw (-0.1,1) -- (3.2,1);
\draw (-0.1,-1) -- (3.2,-1);
\node()at(-0.5,1){\tiny{${\partial}$}};
\node()at(-0.5,-1){\tiny{${\partial'}$}};
\draw[->] (1.3,1.3) -- (1.7,1.3);
\draw[<-] (1.3,-1.3) -- (1.7,-1.3);
\end{tikzpicture}
\hspace{0.01cm}
\begin{tikzpicture}[scale=0.7]
\coordinate (3) at (1.2,-0.5);
\coordinate (5) at (0.9-0.35,-0.1);
\coordinate (6) at (3.1,-0.1);
\coordinate (7) at (0.1-0.5,0.35);
\coordinate (8) at (0,-0.35);
\coordinate (9) at (2.1,0.25);
\node[left,purple]  at (5){\small$\gamma_1$};
\node[orange] at (1.8,0.6){\small$\gamma_2$};
\draw [->,purple,thick](-0.4+0.5,-1) --(2-1,1);
\draw [->,blue,thick](-0.4+0.5,-1) --(2.9,1);
\node[right,blue]  at (1.4,-0.1){\small$\gamma_3$};
\fill  (1,1) circle[radius=1.5pt];
\node[above] at  (1,1){\tiny$0$}; 
\node[below] at  (1,-1){\tiny$0$}; 
\fill  (1,-1) circle[radius=1.5pt];
\node[below] at  (-0.4+0.5,-1){\tiny$ \frac{a}{q}$};  
\fill (-0.4+0.5,-1) circle[radius=1.5pt];
\draw [->,orange,thick](-0.4+0.5+0.9,1) .. controls
(0+0.5+0.9,0.8) and (0.4+0.75+0.9,0.8) ..(1+1+0.9,1);
\node[above] at  (2.9,1){\tiny$ \frac{b}{p}$};  
\fill (2.9,1) circle[radius=1.5pt];
\draw (-0.1,1) -- (3.2,1);
\draw (-0.1,-1) -- (3.2,-1);
\node()at(-0.5,1){\tiny{${\partial}$}};
\node()at(-0.5,-1){\tiny{${\partial'}$}};
\draw[->] (1.3,1.3) -- (1.7,1.3);
\draw[<-] (1.3,-1.3) -- (1.7,-1.3);
\end{tikzpicture}
\caption{The local figures of triangulations in  $\Gamma(0,1)$, $\Gamma(0,1)'$, $\Gamma(a,b)$ and  $\Gamma(a,b)'$, respectively}
\end{figure}
Additionally, let
$$\mathcal{T}_1:=\left\{\Gamma\in\Gamma(a,b)\middle|a=0,-1,\dots,1-q ;b=1,2,\dots,a+q\right\};$$
and
$$\mathcal{T}_2:=\left\{\Gamma\in\Gamma(a,b)'\middle|a=0,-1,\dots,1-p;b=1-a,2-a,\dots,p\right\}.$$

First, we show that all triangulations in $\mathcal{T}_1 \cup \mathcal{T}_2$ are pairwise distinct with respect to $se$-equivalence.

\begin{lemma}
    For any triangulation $\Gamma$ of $A_{p,q}$ in $\mathcal{T}_1\cup \mathcal{T}_2$, there is no distinct  triangulation $\Gamma'$ in $ \mathcal{T}_1\cup \mathcal{T}_2$ such that   $\Gamma' \sim\Gamma$.
\end{lemma}
\begin{proof}
Assume, towards a contradiction, that there exists a triangulation $\Gamma' \in \mathcal{T}_1 \cup \mathcal{T}_2$ with $\Gamma' \neq \Gamma$ and $\Gamma' \sim \Gamma$.
If $\Gamma\in\Gamma(a,b)$ for some $a,b\in\mathbb{Z}$ , then $\Gamma'$ contains the arcs $[D_{\frac{a-c}{q}}^{\frac{c}{p}}]$ and $[D_{\frac{b-c}{q}}^{\frac{c}{p}}]$ for some $0 < c \leq p$. This implies $\Gamma' \notin \mathcal{T}_1$, and the bridging arc in $\Gamma'$ ending at $0_{\partial}$ is of the form $[D_{\frac{d}{q}}^{0}]$ where $d \leq a - c$. However,   this directly contradicts the definition of $\mathcal{T}_2$.

Similarly, if $\Gamma\in\Gamma(a,b)'$ for some $a,b\in\mathbb{Z}$, then $\Gamma'$ is not in $\mathcal{T}_1$, and the bridging arc in $\Gamma'$ ending at $0_{\partial}$ is of the form $[D_{\frac{d}{q}}^{0}]$ with $d \leq a - c$ for some $0 < c \leq p - b$. The presence of this arc violates the conditions for $\Gamma'$ to be in $\mathcal{T}_2$, again resulting in a contradiction.

In both cases, the assumption of the existence of such a $\Gamma'$ leads to a contradiction, thereby establishing the lemma.
\end{proof}

Next we prove that every triangulation of $A_{p,q}$ is $se$-equivalent to a triangulation in $\mathcal{T}_1 \cup \mathcal{T}_2$.
\begin{lemma}
    For any triangulation $\Gamma$ of $A_{p,q}$, there exists a triangulation $\Gamma'\in \mathcal{T}_1\cup \mathcal{T}_2$ such that $\Gamma'\sim \Gamma$.
\end{lemma}
\begin{proof}
Without loss of generality, assume $[D_{\frac{a}{q}}^{0}] \in \Gamma$ with $-q \leq a \leq q$.  The assertion holds trivially if  $\Gamma \in \mathcal{T}_1 \cup \mathcal{T}_2$. Thus, we focus on the cases where
  $\Gamma \notin \mathcal{T}_1 \cup \mathcal{T}_2$, which implies either $-q \leq a < 0$ or $0<a\leq q$.
  
 
 First, consider $[D_{\frac{a}{q}}^{0}] \in \Gamma$ with $-q \leq a < 0$.
 We apply induction on $|a|$ to prove that if 
 $[D_{\frac{a}{q}}^{0}] \in \Gamma$  with $-q \leq a < 0$, then 
 the result is valid.  When  $|a| = 1$, $ \Gamma\notin \mathcal{T}_1 \cup \mathcal{T}_2$ implies  $[D_{-\frac{1}{q}}^{\frac{1}{p}}]\in\Gamma $. Hence, the triangulation 
${_{s^{-1}}\Gamma_{e^{-1}}} \in \mathcal{T}_1 \cup \mathcal{T}_2$. For  $|a|  > 1$, the condition 
 $\Gamma \notin \mathcal{T}_1 \cup \mathcal{T}_2$ implies that $\Gamma$ contains  $$\{[D_{\frac{c}{q}}^{0}], [D_{\frac{c}{q}}^{\frac{b}{p}}],[D^{0,\frac{b}{p}}]\}$$
 with $a\leq c<0$ and $0<b \leq -c\leq -a$. 
 \begin{itemize}
\item[-] For $c>a$, the result follows by the induction hypothesis.  

\item[-]
If $c=a$, then $[D_{\frac{a+b}{q}}^{0}] \in {_{s^{-b}}\Gamma_{e^{-b}}}$.  Since $a < a + b < 0$, ${_{s^{-b}}\Gamma_{e^{-b}}} \in \mathcal{T}_1$ if $[D_{\frac{d}{q}}^{0}] \in {_{s^{-b}}\Gamma_{e^{-b}}}$ for some $0 < d \leq q$. Otherwise, by the induction hypothesis, there exists $\Gamma' \in \mathcal{T}_1 \cup \mathcal{T}_2$ such that $\Gamma' \sim {_{s^{-b}}\Gamma_{e^{-b}}}$.  Consequently, $\Gamma'\sim\Gamma$.
\end{itemize}

Second, consider the case where $[D_{\frac{a}{q}}^{0}]$ with $0<a\leq q$.  Here, $\Gamma \notin \mathcal{T}_1 \cup \mathcal{T}_2$ implies that $\Gamma$ contains $[D_{\frac{a}{q}}^{\frac{b}{p}}]$  with $b < 0 < a$.
Hence,    $_{s^{|b|}}\Gamma_{e^{|b|}}$ contains $[D_{\frac{a+b}{q}}^{0}] $ and $[D_{\frac{a+b}{q}}^{-\frac{b}{p}}]$.    
 We also apply induction on $a$.
 The base case  $a=1$
   is easy to check. For $a>1$, consider the following subcases:
 \begin{itemize}
\item[-] 
 If $|b| < a$, then $0 < a + b < a$  and the result follows from the induction hypothesis.
 \item[-] If $|b| \geq a$, then $a + b < 0$ and $|a + b| < |b|$. It follows that ${_{s^{|b|}}\Gamma_{e^{|b|}}} \in \mathcal{T}_1 \cup \mathcal{T}_2$. 
 \end{itemize}
 Hence, the statement holds for this case as well.
\end{proof}

\begin{lemma}\label{cardinality of S}
    The cardinality of $\mathcal{T}_1\cup \mathcal{T}_2$  is given by
    \[\sum_{k=1}^{q} k C_{p+q-k}C_{k-1}+ \sum_{l=1}^{p}lC_{p+q-l}C_{l-1},
\]where $C_n$ denotes  the $n$-th Catalan number.
\end{lemma}
\begin{proof}
It can be verified that the sets $\mathcal{T}_1$ and $\mathcal{T}_2$ have the following disjoint unions:
\[
\mathcal{T}_1 = \bigcup_{a=0}^{1-q} \bigcup_{b=1}^{a+q} \{\Gamma \in \Gamma(a,b)\}, \quad \mathcal{T}_2 = \bigcup_{a=0}^{1-p} \bigcup_{b=1-a}^{p} \{\Gamma \in \Gamma(a,b)'\}.
\]
For given $a$ and $b$, the  triangulations in $\Gamma(a,b)$  are counted by $C_{b-a-1} C_{p+q+a-b}$, and the  triangulations in $\Gamma(a,b)'$  are counted by   $C_{b-1} C_{p+q-b}$. 
Thus, the cardinality of $\mathcal{T}_1 \cup \mathcal{T}_2$ is calculated as follows:
\[
 \sum_{a=0}^{1-q} \sum_{b=1}^{a+q} C_{b-a-1} C_{p+q+a-b} +  \sum_{a=0}^{1-p} \sum_{b=1-a}^{p} C_{b-1} C_{p+q-b}.
\]
By reindexing the sums, this expression simplifies to:
\[
\sum_{k=1}^{q} k C_{p+q-k}C_{k-1}+ \sum_{l=1}^{p}lC_{p+q-l}C_{l-1}.
\]
\end{proof}

We are now ready to present our final main result.
\begin{theorem}\label{numbertilting sheaves}
The number of tilting sheaves   in $\cohx$, up to
$\tau$-equivalence,  is given by
$$\sum_{k=1}^{q} k C_{p+q-k}C_{k-1}+ \sum_{l=1}^{p}lC_{p+q-l}C_{l-1},$$ where $C_n$ denotes  the $n$-th Catalan number. 
\end{theorem}
 \begin{proof}
 Let $\Gamma$ and $ \Gamma'$ be two triangulations of $A_{p,q}$.
     By Proposition~\ref{prop:tau} and \cite[Theorem~4.3]{CRZ23}, we have that $\Gamma' $ is $se$-equivalent to  $ \Gamma$ if and only if the 
 tilting sheave  $\phi(\Gamma')$ is $\tau$-equivalent to $\phi(\Gamma)$.  This theorem thus follows directly from Lemma~\ref{cardinality of S}.
 \end{proof} 

 \subsection*{Acknowledgments} The authors would like to thank Shiquan Ruan for helpful discussions. This work was partially supported by  the  National Natural Science Foundation of China (Grants Nos. 12371040 and 12131018) and the Fujian Provincial Natural Science Foundation of China (Grant No. 2024J010006).

\end{document}